\documentclass[a4paper,twoside,11pt]{article}
\usepackage{a4wide} 

\usepackage{cmbright}
\usepackage{microtype}
\usepackage[T1]{fontenc}
\usepackage{lmodern}
\usepackage[utf8]{inputenc}

\usepackage{bm} 
\usepackage{amsfonts,amssymb,amsmath,amsthm,dsfont} 
\usepackage{xspace}
\usepackage{graphicx}
\usepackage{xcolor}
\usepackage{cite}
\usepackage{hyperref}
\usepackage{subfig}
\definecolor{darkgreen}{rgb}{0,0.4,0}
\definecolor{BrickRed}{rgb}{0.65,0.08,0}
\hypersetup{colorlinks=true,linkcolor=blue,citecolor=red,filecolor=BrickRed,urlcolor=darkgreen}
\usepackage{tikz}
\usepackage{tkz-tab}
\usepackage{booktabs}
\usepackage{multirow}

\usepackage{paralist}   

\theoremstyle{plain}
\newtheorem{theorem}{Theorem}[section]
\newtheorem{corollary}[theorem]{Corollary}
\newtheorem{lemma}[theorem]{Lemma}
\newtheorem{proposition}[theorem]{Proposition}
\newtheorem{conjecture}[theorem]{Conjecture}

\theoremstyle{definition}
\newtheorem{definition}[theorem]{Definition}

\theoremstyle{remark}
\newtheorem{remark}[theorem]{Remark}

\newcommand{\E}{\mathbb{E}}
\newcommand{\V}{\mathbb{V}}
\newcommand{\KK}{\mathbb{K}}
\newcommand{\QQ}{\mathbb{Q}}
\newcommand{\Nc}{\mathcal{N}}
\newcommand{\Rc}{\mathcal{R}}

\newcommand{\proba}{\mathds{P}}
\newcommand{\bigO}{\mathcal{O}}
\newcommand{\Pc}{\mathcal{P}}

\newcommand{\integers}{\mathds{Z}}

\newcommand{\mF}{\mathcal{F}}

\newcommand*{\eg}{\textit{e.g.,}\@\xspace}
\newcommand*{\ie}{\textit{i.e.,}\@\xspace}
			
\newcommand*{\resp}{\textit{resp.}\@\xspace}

\newcommand{\vect}[1]{\bm{#1}}
\newcommand{\va}{\vect{a}}
\newcommand{\vb}{\vect{b}}
\newcommand{\vc}{\vect{c}}

\newcommand{\vr}{\vect{r}}
\newcommand{\vs}{\vect{s}}
\newcommand{\vt}{\vect{t}}

\newcommand{\vw}{\vect{w}}

\newcommand{\id}{\operatorname{Id}}
\newcommand{\indic}{\mathds{1}}

\newcommand{\reach}{\operatorname{reach}}

\newcommand{\integersets}{\mathcal{F}}
\newcommand{\mA}{\mathcal{A}}
\newcommand{\mB}{\mathcal{B}}
\newcommand{\mC}{\mathcal{C}}
\newcommand{\lowercond}{\operatorname{lowercondition}}
\newcommand{\uppercond}{\operatorname{uppercondition}}

\newcommand{\Ijbottom}{I_{\operatorname{bottom}}}
\newcommand{\Ijtop}{I_{j, \operatorname{top}}}
\newcommand*{\rone}{{\romannumeral 1}\@\xspace}
\newcommand*{\rtwo}{{\romannumeral 2}\@\xspace}
\newcommand*{\rthree}{{\romannumeral 3}\@\xspace}

\def\inlaw{ \stackrel{\scriptstyle \mathcal L}{\longrightarrow}  } 

\newcommand{\OEIS}[1]{\href{http://oeis.org/#1}{OEIS~#1}}
\newcommand{\OEISs}[1]{\href{http://oeis.org/#1}{#1}}

\newcommand{\wo}{y} 
\newcommand{\wt}[1]{p_{#1}}
\newcommand{\Bb}{B_2}

\newcommand{\Z}{\mathbb{Z}}

\RequirePackage{rotating}

\newcommand{\conjugate}{\overline}
\newcommand{\type}[5]{\operatorname{type}(#1, #2, #3, #4, #5)}

\newcommand{\anonymous}[2]{#1}

\newcommand{\usub}{%
   \setbox0=\hbox{$\hat{-}$}%
   \setbox1=\hbox{$-$}%
   \dimen0=1.35\ht1%
  \smash{\mskip4mu%
  \raise\dimen0\rlap{%
      \begin{turn}{180}$\hat{\phantom{-}}$\end{turn}}%
      -%
      }\mskip4mu%
{\vphantom{\widehat{-}}}}

\newcommand{\msub}{%
  \setbox0=\hbox{$\hat{-}$}%
  \setbox1=\hbox{$-$}%
  \dimen0=-0.5\ht1%
  \smash{\mskip4mu%
  \raise\dimen0\rlap{%
      $\hat{\phantom{-}}$}%
      -}\mskip4mu%
{\vphantom{\widehat{-}}}}


\newcommand{\RomanNumeralCaps}[1]{\MakeUppercase{\romannumeral #1}}
\newcommand{\romI}{\operatorname{\RomanNumeralCaps{1}}}
\newcommand{\romII}{\operatorname{\RomanNumeralCaps{2}}}

\graphicspath{{./}{Figures/}}
\makeatletter
\def\input@path{{./}{Figures/}}  
\makeatother

\begin{document}

\title{Combinatorics of nondeterministic walks}
\author{\anonymous{\'Elie de Panafieu\thanks{Nokia Bell Labs and Lincs, France} \\
\and
Michael Wallner\thanks{TU Wien, Austria}}{Authors}}
\date{}

\maketitle

\begin{abstract} 
This paper introduces \emph{nondeterministic walks}, a new variant of one-dimensional discrete walks. The main difference from classical walks is that its nondeterministic steps consist of \emph{sets of steps} from a predefined set, allowing for parallel exploration of all possible extensions.

We discuss in detail two particular nondeterministic step sets inspired by Dyck and Motzkin walks and show that several nondeterministic classes of lattice paths, such as nondeterministic bridges, excursions, and meanders, are algebraic. The key concept is the generalization of the ending point of a walk to its reachable points, \ie a set of ending points. We extend our results to general step sets: We show that nondeterministic bridges and several subclasses of nondeterministic meanders are always algebraic. We conjecture the same is true for nondeterministic excursions, and we provide Python and Maple packages to support our conjecture.

This research is motivated by the study of networks involving encapsulation and decapsulation of protocols. Our results are obtained using generating functions, analytic combinatorics, and additive combinatorics.

\bigskip

\textbf{Keywords.} Random walks, analytic combinatorics, additive combinatorics, generating functions, limit laws, networking, encapsulation.
\end{abstract}

\tableofcontents

		\section{Introduction}

Lattice paths are among the most natural and extensively studied objects in combinatorics, appearing across a wide range of fields including probability theory, computer science, biology, chemistry, and physics~\cite{Kn69,HOP17,BP08,BaWa17}.
One of their key features is their versatility as models to capture natural phenomena, as in up-to-date models of certain polymers~\cite{vRPR08}.
In this paper we continue this success story and present a new model of lattice paths capturing  
the encapsulation and decapsulation of protocols over networks.
The key to understand them is the ability to follow all trajectories in the network \emph{simultaneously}. 
For this purpose, we generalize the class of lattice paths to so called \emph{nondeterministic lattice paths}. 
In our context, this word does not mean ``\emph{random}''.
Instead it is understood in the same sense
as for automata and Turing machines.
A process is nondeterministic
if several branches are explored in parallel,
and the process is said to end in an accepting state
if at least one of those branches ends in an accepting state.
Before we give a precise definition, we recall the classical model of lattice paths we build on~\cite{BaFl02}.


  \paragraph{Classical walks.}

%
Given a finite set $S \subset \integers$ 
of integers, called the \emph{steps}, 
a \emph{walk} of length $n$ is a sequence
$v = (v_1, \ldots, v_n)$ of steps $v_i \in S$. 
In this paper we will always assume that our walks start at the origin.
Its \emph{endpoint} is equal to
the sum of its steps $\sum_{i=1}^n v_i$.
As illustrated in Figure~\ref{fig:walk},
a walk can be visualized by its \emph{geometric realization} as a sequence of ordinates $(\wo_0,\wo_1,\dots,\wo_n)$.
Since the walk starts at the origin it starts at $\wo_0 = 0$,
and after appending $k$ steps it has reached $\wo_k := \sum_{i=1}^k v_i$.
Thus, the relation $v_i = \wo_i - \wo_{i-1}$ directly connects the steps and the ordinates of a walk, and shows their equivalence.

We distinguish four different types of walks:
\begin{compactitem}
    \item A \emph{bridge} is a walk with endpoint $\wo_n=0$.
    \item A \emph{meander} is a walk never crossing the $x$-axis, \ie $\wo_j \geq 0$ for all $j=0,\ldots,n$.
    \item An \emph{excursion} is a meander with endpoint $\wo_n=0$.
    %
\end{compactitem}

  \paragraph{Nondeterministic walks.}

We are now ready to generalize classical walks to nondeterministic walks.
The only difference with respect to classical walks is that steps are now \emph{sets} of integers.

\begin{definition}[Nondeterministic walks]
An \emph{N-step} $\vs \subset \integers$ is a non-empty finite set of integers.
Given a set~$S$ of N-steps,
an \emph{N-walk} of length $n$ is a sequence $w = (\vw_1, \ldots, \vw_n)$ of N-steps $\vw_i \in S$.
\end{definition}

\noindent In addition, as for classical walks, we always assume that they start at the origin and we distinguish different families.

\begin{definition}[Families of N-walks] \label{def:types_of_n_walks}
An N-walk $w = (\vw_1, \ldots, \vw_n)$ and a classical walk $v = (v_1, \ldots, v_n)$
are \emph{compatible} if they have the same length $n$ and for each $1 \leq i \leq n$,
the $i^{\text{th}}$ step is included in the $i^{\text{th}}$ N-step, \ie $v_i \in \vw_i$.
An \emph{N-bridge} (\resp \emph{N-meander}, \resp \emph{N-excursion})
is an N-walk compatible with at least one bridge
(\resp meander, \resp excursion).
Thus, N-excursions are particular cases of N-meanders.
\end{definition}

\noindent The endpoints of classical walks are central to the analysis.
We define their nondeterministic analogues. 

\begin{definition}[Reachable points] \label{def:reachable_points}
The \emph{reachable points} of an N-walk (\resp N-meander) are the endpoints of all compatible walks (\resp meanders).
The minimum (\resp maximum) reachable point of an N-walk $w$ is denoted by $\min(w)$ (\resp $\max(w)$).
The minimum (\resp maximum) reachable point of an N-meander $w$ is denoted by $\min^+(w)$ (\resp $\max^+(w)$), to emphasize the positivity constraint.
\end{definition}
Note that the reachable endpoints of an N-meander are nonnegative.
The geometric realization of an N-walk
is the sequence of its reachable points after $j$ steps for $j=0,\dots,n$, as shown in Figure~\ref{fig:walks}.

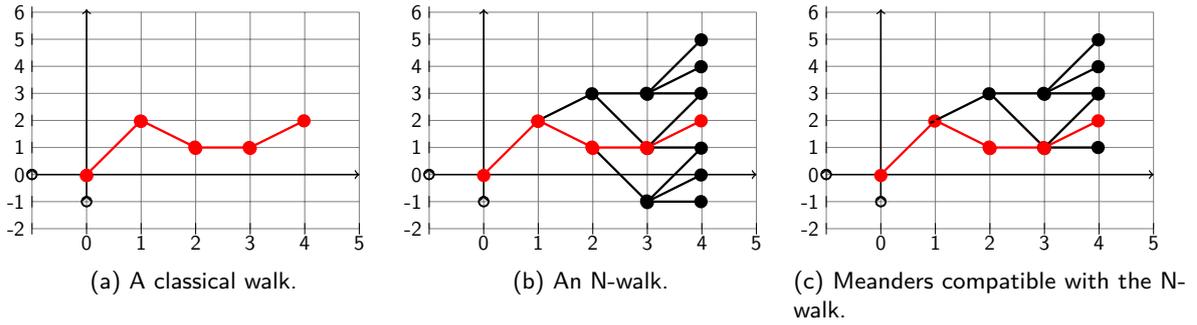
\begin{figure}[ht]
	\centering
	\subfloat[A classical walk.]{%
    \resizebox{0.32\textwidth}{!}{\begin{tikzpicture}[shorten >=-3pt,shorten <=-3pt, x=1cm, y=0.5cm]
	\draw [very thin, gray,ystep=1] (-1,-2) grid (5,6) ;

	\draw [-*, very thick,red] (0,0) edge (1,2) ;
	
	\draw [-*, very thick, red] (1,2) edge (2,1) ;
	
	\draw [-*, very thick, red] (3,1) edge  (4,2) ;
	\draw [-*, very thick, red] (2,1)  edge (3,1) ;
	
	\draw [o->,thick] (-1,0) edge (4.9,0);
	\draw [o->,thick] (0,-1) edge (0,5.9);

	\foreach \x in {-2,...,6}
	\draw (-1,\x)--(-1,\x)  node[left] {\x};
	
	\foreach \y in {0,...,5}
	\draw (\y,-2)--(\y,-2)  node[below] {\y};
	\draw [-*, very thick, red] (0,0) ;	
	\end{tikzpicture}}
		\label{fig:walk}}
        \hfill
\subfloat[An N-walk.]{%
    \resizebox{0.32\textwidth}{!}{\begin{tikzpicture}[shorten >=-3pt,shorten <=-3pt, x=1cm, y=0.5cm]
	\draw [very thin, gray,ystep=1] (-1,-2) grid (5,6) ;

	\draw [-*, very thick] (0,0) edge[red] (1,2) (1,2) edge (2,3) (2,3) edge (3,3) edge[-] (3,1) (3,3) edge (4,5) edge (4,4) edge (4,3);
	
	\draw [-*, very thick, red] (1,2) edge (2,1) ;
	
	\draw [-*, very thick] (3,1) edge[-] (4,1) edge[red] (4,2) edge[-] (4,3) ;
	\draw [-*, very thick] (2,1)  edge (3,-1) edge[red] (3,1) (3,-1) edge (4,-1) edge (4,0) edge (4,1);
	
	\draw [o->,thick] (-1,0) edge (4.9,0);
	\draw [o->,thick] (0,-1) edge (0,5.9);

\foreach \x in {-2,...,6}
	\draw (-1,\x)--(-1,\x)  node[left] {\x};
	
\foreach \y in {0,...,5}
\draw (\y,-2)--(\y,-2)  node[below] {\y};
	\draw [-*, very thick, red] (0,0) ;	
	\draw [-*, very thick, red] (2,1) ;	
	\draw [-*, very thick, red] (3,1) ;	
	\end{tikzpicture}}
		\label{fig:n_walk}}
        \hfill
	\subfloat[Meanders compatible with the N-walk.]{%
    \resizebox{0.32\textwidth}{!}{\begin{tikzpicture}[shorten >=-3pt,shorten <=-3pt, x=1cm, y=0.5cm]
	\draw [very thin, gray,ystep=1] (-1,-2) grid (5,6) ;
	
	\draw [-*, very thick] (0,0) edge[red] (1,2) (1,2) edge (2,3) (2,3) edge (3,3) edge[-] (3,1) (3,3) edge (4,5) edge (4,4) edge (4,3);
	
	\draw [-*, very thick, red] (1,2) edge (2,1) (2,1) edge (3,1) ;
	
	\draw [-*, very thick] (3,1) edge (4,1) edge[red] (4,2) edge (4,3) ;

	\draw [o->,thick] (-1,0) edge (4.9,0);
	\draw [o->,thick] (0,-1) edge (0,5.9);

	\foreach \x in {-2,...,6}
	\draw (-1,\x)--(-1,\x)  node[left] {\x};
	
	\foreach \y in {0,...,5}
	\draw (\y,-2)--(\y,-2)  node[below] {\y};
	\draw [-*, very thick, red] (0,0) ;	
	\draw [-*, very thick, red] (2,1) ;	
	\draw [-*, very thick, red] (3,1) ;	
	\end{tikzpicture}}
		\label{fig:n_meander}} 
    \caption{Geometric realization of the classical meander $v = (2,-1,0,1)$, 
    the N-walk $w = \left(\{2\}, \{-1,1\}, \{-2,0\}, \{0,1,2\}\right)$,
    and the N-meander with the same N-steps. 
    As highlighted in red, $v$ is compatible with $w$.}
   \label{fig:walks}
\end{figure}


As in the classical setting, the N-steps of an N-walk can be weighted. 
Then, each N-step is associated with a weight $p_i$ and the weight of the N-walk is the product of the weights associated with its N-steps.
In many cases, these weights $p_i$ will represent probabilities, \ie $p_i \geq 0$ and $p_1 + \dots + p_m=1$.
We talk about the \emph{unweighted model} when all weights are equal to one: $p_i = 1$.

\subsection{Main results}

Our main results are the analysis of the asymptotic number of nondeterministic walks of the Dyck and Motzkin type with step sets 
\[
    \Big\{ \{-1\}, \{1\}, \{-1, 1\} \Big\} 
    \qquad \text{ and }  \qquad
    \Big\{\{-1\}, \{0\}, \{1\}, \{-1,0\}, \{-1,1\}, \{0,1\}, \{-1,0,1\} \Big\},
\]
 respectively. The results for the unweighted case are summarized in Table~\ref{tab:compareDyckMotzkin}.
These results are derived using generating functions and singularity analysis (see Section~\ref{sec:introAC} for a brief introduction and \cite{FS09} for a comprehensive treatment).
Remarkably, for N-bridges, N-meanders, and N-excursions the leading order and the lower-order terms have different exponential growth rates.
This phenomenon arises due to a dominant polar singularity followed by a square root singularity.
Moreover, observe that for different types the leading orders only differ in the multiplicative constants. 
This is in sharp contrast to classical walks, in which we observe polynomial differences~\cite{BaFl02}. 
For example the number of classical bridges of size $n$ is of order $1 / \sqrt{n}$ compared to classical walks. 
Here, this polynomial change in order of magnitude is only visible in the lower order terms.
%
In particular, the limit probabilities for a Dyck N-walk of even length to be an N-bridge, an N-meander, or an N-excursion, are $1$, $\frac{1}{2}$, or $\frac{1}{4}$, and for Motzkin N-walks $1$, $\frac{3}{4}$, or $\frac{9}{16}$, respectively.

We also explore general N-steps and prove that the generating function of N-bridges with respect to length, minimum, and maximum reachable point is always algebraic in Theorem~\ref{th:general_bridges}.
We conjecture that the generating function of N-excursions is always algebraic and provide proofs in some specific cases in Section~\ref{sec:general_NExcursions}.

\renewcommand{\arraystretch}{1.5}
\begin{table*}[ht]
	\begin{center}
	\begin{tabular}{ccc}
		\toprule
        Type  & Dyck N-steps  & Motzkin N-steps\\
        & 
        $\Pc(\{-1,1\}) \setminus \emptyset $ &
        $\Pc(\{-1,0,1\}) \setminus \emptyset $\\        
		\midrule  
			N-Walk & 
            $3^n$ & 
            $7^n$ 
            \\
			N-Bridge & 
            $\frac{1+(-1)^n}{2} \left( 3^n - \frac{2\sqrt{2}}{\sqrt{\pi}} \frac{8^{n/2}}{\sqrt{n}} + \bigO\left(\frac{8^{n/2}}{n^{3/2}}\right) \right)$ & 
            $7^n - \sqrt{\frac{3}{\pi}} \frac{6^n}{\sqrt{n}} + \bigO\left(\frac{6^n}{n^{3/2}}\right)$ 
            \\
            N-Meander & 
            $\frac{3^n}{2} + \frac{3\sqrt{2}(1+(-1)^n) + 4(1-(-1)^n)}{\sqrt{\pi}} \frac{8^{n/2}}{\sqrt{ n^3}} + \bigO\left(\frac{8^{n/2}}{n^{5/2}}\right)$ & 
            $\frac{3}{4}7^n + \frac{3\sqrt{3}}{2 \sqrt{\pi}} \frac{6^n}{\sqrt{n^3}} + \bigO\left(\frac{6^n}{n^{5/2}}\right)$
            \\
            N-Excursion & $\frac{1+(-1)^n}{2}\left( \frac{3^n}{4} + 4\sqrt{2} \frac{8^{n/2}}{\sqrt{\pi n^3}}  + \bigO\left(\frac{8^{n/2}}{n^{5/2}}\right) \right)$ & 
            $\frac{9}{16}7^n - \gamma \frac{6^n}{\sqrt{\pi n^3}} + \bigO\left(\frac{6^n}{n^{5/2}}\right)$
            \\
		\bottomrule
  \end{tabular}
\end{center}
\caption{The asymptotic number of nondeterministic unweighted (all weights equal to one) Dyck and Motzkin N-walks, N-bridges, N-meanders, and N-excursions
with $n$ steps.
The constant $\gamma \approx 0.6183$ is an algebraic number defined as the positive real solution of $1024\gamma^4-8019\gamma^2+2916=0$.}
\label{tab:compareDyckMotzkin}
\end{table*}
\renewcommand{\arraystretch}{1.0}

	\subsection{Motivation and related work}
 \label{sec:motivation}

Our nondeterministic model of lattice paths has its roots in networking.
Let us first give a vivid description of the underlying mechanisms using Russian dolls.  

\paragraph{Russian dolls.}
Suppose we have a set of $n+1$ people arranged in a line. There are three kinds of people. A person of the first kind is only able to put a received doll in a bigger one. A person of the second kind is only able to extract a smaller doll (if any) from a bigger one. If she receives the smallest doll, then she throws it away. Finally, a person of the third kind can either put a doll in a bigger one or extract a smaller doll if any. 
We want to know if it is possible for the last person to receive the smallest doll after it has been given to the first person and then, consecutively, handed from person to person while performing their respective operations. 
This is equivalent to asking if a given N-walk with each N-step $\in \left\{\{1\},\{-1\},\{-1,1\}\right\}$ is an N-excursion, \ie if the N-walk is compatible with at least one excursion. The probabilistic version of this question is: what is the probability that the last person can receive the smallest doll according to some distribution on the set of people over the three kinds?

	\paragraph{Networks and encapsulations.}

The original motivation for this work comes from networking. In a network, some nodes are able to encapsulate protocols (put a packet of a protocol inside a packet of another one), decapsulate protocols (extract a nested packet from another one), or perform any of these two operations (albeit most nodes are only able to transmit packets as they receive them). Typically, a tunnel is a subpath starting with an encapsulation and ending with the corresponding decapsulation. Tunnels are very useful for achieving several goals in networking (\eg interoperability: connecting IPv6 networks across IPv4 ones~\cite{wu2013transition}; security and privacy: securing IP connections~\cite{seo2005security}, establishing Virtual Private Networks~\cite{rosen2015multicast}, etc.). Moreover, tunnels can be nested to achieve several goals. Replacing the Russian dolls by packets, we see that an encapsulation can be modeled by a $\{1\}$-step and a decapsulation by a $\{-1\}$-step, while a passive transmission of a packet is modeled by a $\{0\}$-step.

Given a network with some nodes that are able to encapsulate or decapsulate protocols, a path from a sender to a receiver is \textit{feasible} if it allows the latter to retrieve a packet exactly as dispatched by the sender. 
Computing the shortest feasible path between two nodes is polynomial-time~\cite{LFCP16} if cycles are allowed without restriction. In contrast, the problem is  $\mathsf{NP}$-hard if cycles are forbidden or arbitrarily limited. In~\cite{LFCP16}, the algorithms are compared through worst-case complexity analysis and simulation. The simulation methodology for a fixed network topology is to make encapsulation (\resp decapsulation) capabilities available with some probability $p$ and observe the processing time of the different algorithms. It would be interesting, for simulation purposes, to generate random networks with a given probability of existence of a feasible path between two nodes. This work is the first step towards achieving this goal, since our results give the probability that a given path is feasible (\ie is an N-excursion) according to a probability distribution of encapsulation and decapsulation capabilities over the nodes.

	\paragraph{Lattice paths.}

Nondeterministic walks naturally connect lattice paths and branching processes. 
This is underlined by our usage of many well-established analytic and algebraic tools previously used to study lattice paths.
In particular, we rely on the robustness of D-finite functions with respect to the Hadamard product, as well as on the kernel method~\cite{FS09,BM10,BaFl02,BaWa19,BKKK}.

The N-walks are nondeterministic one-dimensional discrete walks.
We will see that their generating functions require three variables:
one marking the lowest point $\min(w)$ that can be reached by the N-walk $w$,
another one marking the highest point $\max(w)$,
and the last one marking its length.
For this reason they are also closely related to two-dimensional lattice paths,
when interpreting $\left(\min(w), \max(w)\right)$ as coordinates in the plane.

\paragraph{Nondeterminism and context-free grammars.}

Nondeterministic walks are very closely linked to nondeterministic finite automata with a stack, and hence to context free grammars.
In particular, 
    for any finite set of N-steps the class of N-excursions, N-meanders, and N-bridges can be encoded by a context-free grammar.
%
    To see this, note that
    languages generated by context-free grammars are equivalently recognized by pushdown automata with a single stack. 
    These automata are by definition nondeterministic, hence they allow $\varepsilon$-transitions which we use to follow all trajectories in the N-walk simultaneously. 
    Then the stack is used to keep track of the distance of the current trajectory to the $x$-axis, thereby distinguishing between different types of N-walks. 

If a context-free language admits an unambiguous grammar, the generating function of the number of its words of length $n$ is algebraic. 
These words of length $n$ are exactly the N-walks of length $n$, and hence, this gives a strong hint that the N-walks we analyze in this paper are algebraic.
However, it is not obvious how to find a unambiguous grammar directly, and even if it is found, how to derive the characterizing algebraic equations.
Moreover, the associated system of equations is already for small examples nearly impossible to solve (resultants of large degree).
For this reason, we show in this article an alternative approach for the enumeration and asymptotics building on the \emph{kernel method}.
It has the advantage that it captures the algebraic nature of the generating functions and allows to compute the asymptotics via singularity analysis.

We leave as an interesting open problem
to determine which N-step sets $S$
correspond to non-ambiguous context-free languages
for N-bridges or N-excursions.

\subsection{Outline of the paper}

We start in Section~\ref{sec:introAC} with a short introduction into the techniques from Analytic Combinatorics that we will use repeatedly.
After that we commence with the study of $N$-walks. 
We study Dyck N-walks in Section~\ref{sec:DyckNWalks}.
We derive closed forms for the generating functions and two limit laws for reachable points in N-excursions: maximal value at the end and the frequency of $\{0\}$.
We continue in Section~\ref{sec:MotzkinNWalks} with the analysis of Motzkin N-walks and obtain functional equations and closed forms for special cases.
All these computations are made explicit in an accompanying Maple worksheet.
In Section~\ref{sec:general_NBridges} we show that the generating function of N-bridges (with respect to length, minimum and maximum reachable point) is algebraic for any finite N-step set using concepts from additive combinatorics.
In Section~\ref{sec:general_NExcursions} we show that the generating function of N-meanders (with respect to length and maximum reachable point) is also algebraic for any finite N-step set and we conjecture that the same is true for N-excursions.
We also present Python code that we implemented to derive the functional equations for any N-step set for N-excursion and N-bridges.

This paper is the full version of an extended abstract\anonymous{~\cite{DePanafieuLamaliWallner2019Nondet} of 12 pages that was published
in the proceedings of the \emph{Sixteenth Workshop on Analytic Algorithmics and Combinatorics (ANALCO)} in 2019 in San Diego.}{ that will be cited in the final version.}
Here we provide full proofs that were previously only sketched. 
We greatly extend the proof of the algebraicity of general N-bridges in Section~\ref{sec:proof_general_bridges}, in particular by providing an extended definition of types and proving the finiteness of the automaton.
We additionally give limit law results for Dyck N-walks in Section~\ref{sec:limit_laws_NExcursions}. We also study subclasses of general N-excursions in Section~\ref{sec:general_NExcursions} and present our Python and Maple code~\cite{gitlabproject}.

\section{Techniques from Analytic Combinatorics}
\label{sec:introAC}

We present a short introduction to the field and follow the reference book~\cite{FS09}. 
The reader familiar with the concepts may skip this section.

\subsection{The symbolic method}\label{sec:introsymbmethod}

\paragraph{Combinatorial family.}
A \emph{combinatorial family} is a set $\mA$ equipped with a function $| \cdot |$ from $\mA$ to $\integers_{\geq 0}$, called the \emph{size}.
For all $n$, the number $a_n$ of elements of size $n$ in $\mA$, \ie
\[
    a_n = \left| \left\{ a \in \mA \mid |a| = n \right\} \right|,
\]
is assumed to be finite.


\paragraph{Generating function.}
We associate to any combinatorial family $\mA$ a \emph{generating function} $A(z)$
\[
    A(z)
    =
    \sum_{a \in \mA} z^{|a|}
    =
    \sum_{n \geq 0}
    a_n z^n.
\]
It is a formal power series, in the sense that convergence is not assumed.
We write $[z^n] A(z)$ for the $n$-th coefficient of the series:
\[
    [z^n] A(z) = a_n.
\]
This is called \emph{coefficient extraction}. By convention, we write $A(0)$ for $a_0 = [z^0] A(z)$.

\paragraph{Symbolic method.}
The \emph{symbolic method} is a dictionary that translates operations on combinatorial families into analytic relations on their generating functions. We present the following combinatorial operations: disjoint union, Cartesian product, sequences of length $k$, and sequences of arbitrary length.

\begin{itemize}
    
\item 
\textbf{Disjoint union.}
Consider two disjoint combinatorial families $\mA$ and $\mB$ and their union $\mC = \mA \cup \mB$. It is equipped with the size function $|(a, b)| := |a| + |b|$, where $|a|$ (resp.~$|b|$) refers to the size defined on $\mA$ (resp.~$\mB$). Then
\[
    C(z)
    =
    \sum_{c \in \mA \cup \mB}
    z^{|c|}
    =
    A(z) + B(z).
\]
We see that the disjoint union, an operation on combinatorial families, translates into a sum of generating functions.

\item 
\textbf{Cartesian product.}
Now consider the Cartesian product $\mC = \mA \times \mB$ and define the size of a pair as the sum of the sizes of its two elements $|(a,b)| = |a| + |b|$. Then
\[
    C(z)
    =
    \sum_{(a, b) \in \mA \times \mB}
    z^{|(a, b)|}
    =
    \sum_{\substack{a \in \mA \\ b \in \mB}}
    z^{|a|} z^{|b|}
    =
    A(z) B(z).
\]
The Cartesian product translates into the product of generating functions.

\item 
\textbf{Sequence of length $k$.}
It follows by induction that the family $\mC = \mA^k$ of sequences of $k$ elements from $\mA$ has the generating function $C(z) = A(z)^k$.

\item 
\textbf{Sequence.}
We denote the sequences of arbitrary length of objects from $\mA$ by
\[
    \operatorname{Seq}(\mA) = \bigcup_{k \geq 0} \mA^k.
\]
When $\mA$ contains some element $\varepsilon$ of size $0$, any sequence of copies of $\varepsilon$ has size $0$, so there are infinitely many objects of size $0$ in $\operatorname{Seq}(\mA)$, which is then not a combinatorial family. So let us assume $\mA$ contains no object of size $0$, \ie $A(0) = 0$. Then the generating function of $\operatorname{Seq}(\mA)$ is $\sum_{k \geq 0} A(z)^k$, which is easily shown to be the multiplicative inverse of $1 - A(z)$, so we write it $\frac{1}{1 - A(z)}$.

\end{itemize}

\paragraph{Example: classical walks.}
Consider the combinatorial family composed of classical walks on the step set $\{-1, 1\}$. The size of a walk is defined as its length.
The generating function of the family containing only the empty walk is $1 \cdot z^{0} = 1$ (one object of size $0$, no other object).
The generating function of steps (\ie walks of length $1$) is $2 z$.
Since a walk is a sequence of steps, the generating function of all walks is $\frac{1}{1 - 2 z}$ and we recover by coefficient extraction the elementary result that the number of walks of length $n$ is $[z^n] \frac{1}{1 - 2 z} = 2^n$.

An excursion with step set $\{-1,1\}$ is called a \emph{Dyck path}. Any nonempty Dyck path starts with a step $1$ and contains a first step $-1$ that returns to the x-axis. The portion between these two steps and the remaining suffix are themselves Dyck paths (possibly empty). This unique decomposition translates, via the symbolic method, into the functional equation for the generating function $D(z)$ of Dyck path
\[
    D(z)
    =
    1
    + z^2 D(z)^2.
\]
This quadratic equation has solutions
\[
    \frac{1 \pm \sqrt{1 - 4 z^2}}{2 z^2}.
\]
Since there is exactly one Dyck path of length $0$ (the empty walk), we have $D(0) = 1$, so the associated generating function is
\[
    D(z)
    =
    \frac{1 - \sqrt{1 - 4 z^2}}{2 z^2}.
\]
This is an even function, because there are no Dyck paths of odd length.
By coefficient extraction using, \eg Newton's generalized binomial theorem, we recover the classic enumeration of Dyck paths of even length by Catalan numbers
\[
    [z^{2n}] D(z)
    =
    \frac{1}{n + 1} \binom{2 n}{n}.
\]

\subsection{Singularity analysis}\label{sec:introsingana}

All asymptotics are extracted using a combination of the following two foundational results. More comprehensive details and proofs are available in~\cite{FS09}; see in particular Theorems~VI.1 and VI.3 therein.

\begin{theorem}[Standard scale]
For any $\alpha \in \mathbb{R} \setminus \mathbb{Z}_{\leq 0}$, we have
\[
    [z^n] (1 - z / \rho)^{-\alpha}
    =
    \rho^{-n}
    \frac{\alpha (\alpha + 1) \cdots (\alpha + n - 1)}{n!}
    \sim
    \rho^{-n}
    \frac{n^{\alpha - 1}}{\Gamma(\alpha)}.
\]
\end{theorem}

Consider three values $0 < \phi < \pi / 2$, $\rho \in \mathbb{C}$ and $R > |\rho|$. The $\Delta(\phi, \rho, R)$-domain is defined by~\cite[Definition~VI.1]{FS09} as
\[
    \Delta(\phi, \rho, R)
=
    \{
    z \mid
    |z| < R,\ 
    z \neq \rho,\ 
    |\arg(z - \rho)| > \phi
    \}.
\]
Given a function $F(z)$ and a complex value $\rho$, we say that $F(z)$ is analytic in a $\Delta$-domain if there exist $0 < \phi < \pi / 2$ and $R > |\rho|$ such that $F(z)$ is analytic in $\Delta(\phi, \rho, R)$.
To prove that a series $F(z)$ of radius of convergence $\rho > 0$ is analytic in a $\Delta$-domain, it is sufficient to prove that $F(z)$ is analytic at every point of modulus $\rho$, except at $\rho$, and that there is a neighborhood $\Omega$ of $\rho$ such that $F(z)$ is analytic on $\Omega \setminus [\rho, +\infty)$.
This condition suffices for all applications except the proof of Theorem~\ref{th:LawDyckExcMax}, where the full definition is needed.

\begin{theorem}[Transfer]
Consider a series $F(z)$ with nonnegative coefficients and a positive radius of convergence $\rho$, analytic in a $\Delta$-domain and such that as $z \to \rho$ in this domain,
\[
    F(z)
    =
    \bigO \left( (1 - z / \rho)^{-\alpha} \right)
\]
for some $\alpha \in \mathbb{R} \setminus \mathbb{Z}_{\leq 0}$. Then
\[
    [z^n] F(z)
    =
    \bigO \left( \rho^{-n} n^{\alpha - 1} \right).
\]
\end{theorem}

As an illustrative example, consider the series
\[
    F(t) = \frac{1 - 6 t^2}{\sqrt{1 - 8 t^2} (1 - 9 t^2)}
\]
arising from Theorem~\ref{theo:DyckNBridges}. 
First, observe that $F(t)$ is an even function, meaning every coefficient of an odd power of $t$ is strictly zero. 
Therefore, we can rewrite it as $F(t) = G(t^2)$, where 
$$G(z)=\frac{1 - 6 z}{\sqrt{1 - 8 z} (1 - 9 z)}.$$ 
Consequently, it holds that $[t^{2n}]F(t) = [z^n] G(z)$.

Next, we locate the dominant singularity (the one closest to the origin). 
Since $G(z)$ has non-negative coefficients, Pringsheim's Theorem~\cite[Theorem~IV.6]{FS09} guarantees that a singularity lies on the positive real axis. 
The radius of convergence of $G(z)$ is $\rho=1/9$, which constitutes a simple pole and acts as the sole dominant singularity. 
Computing its singular expansion at $z=1/9$ yields the polar contribution:
\begin{align*}
    G(z) = \frac{1}{1-9z} - 2 + \bigO(1-9z).
\end{align*}
To extract further asymptotic terms, we subtract this polar part from $G(z)$. The remainder $G(z) - \frac{1}{1-9z}$ is now analytic for $|z| < 1/8$. 
From the closed form, we directly observe that this remainder becomes singular at the branch point $z=1/8$. 
Expanding it at this new dominant singularity gives:
\begin{align*}
    G(z) - \frac{1}{1-9z} = 8 - \frac{2}{\sqrt{1-8z}} + \bigO(\sqrt{1-8z}). 
\end{align*}
Finally, by applying the standard function scale and transfer theorems to both the polar part and the algebraic remainder, we obtain the expansion:
\begin{align*}
    [z^n] G(z) = 9^n - \frac{2}{\sqrt{\pi}} \frac{8^n}{\sqrt{n}} + \bigO\left(\frac{8^n}{n^{3/2}}\right).
\end{align*}

\subsection{The kernel method}\label{sec:introkernel}

A powerful algebraic technique to analyze lattice paths, and combinatorial objects obeying a linear recurrence in general, is the so-called \emph{kernel method}.
As a primary source, Exercise 2.2.1--4 in Knuth's book~\cite{Kn69} is frequently cited; see also~\cite{BaWa19} and the references therein for modern applications.

Let us consider Dyck meanders: paths that never cross below the $x$-axis, utilizing the step set $\{-1,1\}$.
Let $M(t,u) = \sum_{n,k \geq 0} m_{n,k} t^n u^k$ be their bivariate generating function, where $m_{n,k}$ denotes the number of meanders composed of $n$ steps and ending at altitude $k$. 

Symbolically, a meander is either the empty walk (with generating function $1$), or it is constructed by appending a new step to a shorter meander. However, a structural restriction applies: if the meander currently ends on the $x$-axis (altitude $0$), we cannot append a down-step ($-1$).
Translating this construction via the symbolic method yields the functional equation:
\begin{align*}
    M(t,u) &= 1 + t \left(u + u^{-1}\right)M(t,u) - t u^{-1} M(t,0).
\end{align*}
Note that by definition of $M(t,u)$, the generating function of meanders ending on the $x$-axis is $M(t,0) = D(t)$, which represents the generating function of standard Dyck paths.

At first glance, this functional equation appears under-determined, presenting only one equation for two unknown functions, $M(t,u)$ and $M(t,0)$. 
Here, the kernel method exploits the specific algebraic structure of the equation.
First, we group the terms involving $M(t,u)$ and multiply by $u$ to clear the denominator:
\begin{align*}
    u\left(1-t\left(u + u^{-1}\right)\right)M(t,u) &= u - t M(t,0).
\end{align*}
Observe that on the left-hand side, the prefactor $K(t,u) := u\left(1-t\left(u + u^{-1}\right)\right)$ is fully explicit. This polynomial is known as \emph{the kernel}. 
The underlying strategy parallels the method of characteristics for solving partial differential equations: we bind the variables $u$ and $t$ in such a way that the kernel $K(t,u)$ intentionally vanishes. This algebraic binding isolates $M(t,0)$ on the right-hand side. 
In our specific case, the kernel equation $K(t,u)=0$ simplifies to the quadratic:
\begin{align*}
    t u^2 - u + t = 0.
\end{align*}
This yields two fractional power series solutions in $t$:
\begin{align*}
    u_1(t) &= \frac{1 - \sqrt{1-4t^2}}{2t} 
    & \text{ and } &&
    u_2(t) &= \frac{1 + \sqrt{1-4t^2}}{2t}.
\end{align*}
Analytically, for $t \to 0$, we have $u_1(t) = \bigO(t)$ while $u_2(t) = \bigO(1/t)$. 
We intend to substitute a root $u=u_i(t)$ into the functional equation, and consequently into $M(t,u)$. 
However, for $M(t,u_i(t))$ to remain a well-defined formal power series in $t$, the only legitimate choice is the small branch $u=u_1(t)$.
Substituting the large branch $u=u_2(t)$ would introduce arbitrary negative powers of $t$, and the resulting algebraic structure would fail to be an integral domain.

Therefore, safely substituting $u=u_1(t)$ into our rearranged equation annihilates the left-hand side:
\begin{align*}
    0 &= u_1(t) - t M(t,0).
\end{align*}
From this, we recover the classic generating function for Dyck paths: $M(t,0) = \frac{u_1(t)}{t} = \frac{1 - \sqrt{1-4t^2}}{2t^2}$.
Finally, by substituting this resolved expression for $M(t,0)$ back into the original functional equation, we obtain the explicit closed form for the generating function of all Dyck meanders:
\begin{align*}
    M(t,u) &= \frac{1 - u_1(t)/u}{1-t(u+u^{-1})}.
\end{align*}

\subsection{Algebraic generating functions and formal languages}\label{sec:introalgebraic}

The nature of generating functions is traditionally classified into three main types: rational, algebraic, and D-finite. A generating function $F(t)$ is defined as:
\begin{enumerate}
    \item \emph{Rational} if it is of the shape $F(t) = \frac{P(t)}{Q(t)}$ for two polynomials $P(t), Q(t) \in \mathbb{Q}[t]$;
    \item \emph{Algebraic} if it is the root of a non-zero polynomial $P(u,t) \in \mathbb{Q}[u,t]$, such that $P(F(t),t) = 0$;
    \item \emph{D-finite} (or \emph{holonomic}) if it satisfies a linear differential equation with polynomial coefficients:
    \[
        p_d(t) F^{(d)}(t) + \dots + p_{1}(t) F'(t) + p_0(t) F(t) = 0,
    \]
    with $p_i(t) \in \mathbb{Q}[t]$ for $i=0,1,\dots, d$.
\end{enumerate}
Note that this forms a strict hierarchy: every rational function is algebraic, and every algebraic function is D-finite (see~\cite[Section~6]{S01} for details). These classes satisfy various closure properties; for instance, the sum and product of two algebraic functions remain algebraic (see~\cite{KauersPaule2011Tetrahedron} for applications to computer algebra).
In the accompanying Maple worksheets~\cite{gitlabproject} we make heavy use of these properties implemented in the \emph{gfun} Maple package~\cite{SalvyZimmermann1994gfun}.

Furthermore, algebraic functions possess a remarkable characterization via \emph{diagonals} of multivariate series. If $R(x,y) = \sum_{n,m \geq 0} a_{n,m} x^n y^m$ is a rational function in two variables, its diagonal $\Delta R(t) = \sum_{n \geq 0} a_{n,n} t^n$ is necessarily algebraic. Conversely, by a theorem of Furstenberg~\cite{Furstenberg1963Algebraic}, every algebraic function can be expressed as the diagonal of a bivariate rational function. Note that diagonals of rational functions in more than two variables are D-finite, but not necessarily algebraic~\cite{Lipshitz1988Diagonal}.

These analytic classes mirror the hierarchy of formal languages in a fundamental way~\cite{HopcroftMotwaniUllman2006}. It is well-known that the counting sequence of a \emph{regular language}, which is the class of languages recognized by \emph{finite automata}, always yields a rational generating function. Moving up the hierarchy to context-free languages, we find a deep connection to algebraic functions, formalized by the following seminal result:

\begin{theorem}[{Chomsky--Schützenberger~\cite{ChomskySchuetzenberger1963CF}}]
    The formal power series associated with an unambiguous context-free grammar is algebraic. Consequently, the ordinary generating function enumerating the number of words of length $n$ in an unambiguous context-free language is an algebraic function.
\end{theorem}

While regular languages are recognized by finite automata (deterministic or non-deterministic), context-free languages require a more complex machine model: the \emph{pushdown automaton}. Specifically, the class of deterministic context-free languages is recognized by the \emph{deterministic pushdown automaton}, a distinction that is crucial in parsing theory, though both typically yield algebraic generating functions under non-ambiguity constraints~\cite{HopcroftMotwaniUllman2006}.

		\section{Dyck N-walks}
        \label{sec:DyckNWalks}

One of the most fundamental families of classical lattice paths is that of Dyck paths, introduced in Section~\ref{sec:introsymbmethod}, which are excursions associated with the step set $\{-1,1\}$.
They are enumerated by the ubiquitous Catalan numbers, which also enumerate many other combinatorial objects. 
In this section we study their corresponding N-walks.
A \emph{Dyck N-walk} is an N-walk associated with the N-step set
\[
    S_D = \Big\{ \{-1\}, \{1\}, \{-1, 1\} \Big\}.
\]
Additionally, every step is associated with a weight
$p_{-1}, p_{1}$, and $p_{-1,1}$, respectively.
In the following we will talk of Dyck N-meanders, Dyck N-bridges, and Dyck N-excursions if the walks use the above N-step $S_D$ and are N-meanders, N-bridges, or N-excursions, respectively.

\subsection{Dyck N-walks and N-bridges}

In order to enumerate any class of Dyck N-walks we need to study their reachable points. 

\begin{lemma} \label{lem:dyck_reachable_points}
The reachable points of a Dyck N-walk $w$ are
$$
    \left\{\min(w) + 2 i \mid i  \in \left\{ 0, \dots, \frac{\max(w) - \min(w)}{2} \right\} \right\}.
$$
The same result holds for Dyck N-meanders,
with $\min(w)$ and $\max(w)$ replaced by $\min^+(w)$ and $\max^+(w)$; see Definition~\ref{def:reachable_points}.
\label{lemma:minmax}
\end{lemma}

\begin{proof}
We perform an induction on the number of steps.
A walk starts at $0$.
After appending a new N-step the reachable points are either shifted down or up by one unit by $\{-1\}$ or $\{1\}$, respectively, or they are the sum of both shifts by $\{-1,1\}$. 
Hence, there is always a distance~$2$ between the reachable points.
\end{proof}

By Lemma~\ref{lem:dyck_reachable_points} the reachable points are fully determined by their minimum and maximum.
Therefore, we define the generating functions $D(x,y; t)$ and $D^+(x,y; t)$,
of Dyck N-walks and Dyck N-meanders, respectively, as
\begin{align*}
  D(x,y;t) &:= \sum_{\text{Dyck N-walk $w$}}
   \bigg( \prod_{s \in w} p_s \bigg)
  x^{\min(w)}
  y^{\max(w)}
  t^{|w|},
  \\
  D^+(x,y; t) &:= \sum_{\text{Dyck N-meander $w$}}
  \bigg( \prod_{s \in w} p_s \bigg)
  x^{\min^+(w)}
  y^{\max^+(w)}
  t^{|w|}.
\end{align*}
Note that by construction these are power series in $t$ with Laurent polynomials in $x$ and $y$, as each of the finitely many N-walks of length $n$ has a finite minimum and maximum reachable point.

\begin{remark}
The variables $x$ and $y$ are called \emph{catalytic variables}, as they facilitate the enumeration of such paths by length but are not part of the initial enumeration question. 
The reachable points used here are a generalization of the coordinate of the end point that is used for classical lattice paths; see, \eg \cite{BaFl02,BousquetMelouJehanne2006Catalytic}.
\end{remark}

Observe that Lemma~\ref{lemma:minmax} is coherent with the fact that
all N-bridges and N-excursions have even length.
The number of Dyck N-bridges and Dyck N-excursions are then, respectively, given by
\begin{align}
    \label{eq:bridgesandexcursions}
  [x^{\leq 0} y^{\geq 0} t^{2n}] D(x,y; t)
  \qquad \text{and} \qquad
  [t^{2n}] 
  D^+(0,1; t),
\end{align}
where 
the nonpositive part extraction operator $[x^{\leq 0}]$ is defined as
$[x^{\leq 0}] \sum_{k \in \Z} f_k x^k := \sum_{k \leq 0} f_k x^k$
(and analogously for $[y^{\geq 0}]$).
In the following two subsections we derive closed-form expressions for these quantities. 
The generating function of Dyck N-walks has a particularly simple shape.

\begin{proposition}
    The generating function of Dyck N-walks is given by
    \begin{align}
        \label{eq:DyckNWalkGF}
        D(x,y;t) = \frac{1}{1 - t \left( \frac{p_{-1}}{x y} + p_{1} x y + p_{-1,1} \frac{y}{x} \right)}.
    \end{align}    
\end{proposition}

\begin{proof}
    The only N-walk of length $0$ is the empty walk.
    We obtain each non-empty N-walk, by appending an N-step to another N-walk. 
    In terms of generating functions this gives 
    $
        D(x,y;t) = 1 + t D(x,y;t) \left( \frac{p_{-1}}{x y} + p_{1} x y + p_{-1,1} \frac{y}{x} \right),
    $
    which proves the claim.
\end{proof}

The representation~\eqref{eq:DyckNWalkGF} has a direct bijective interpretation in terms of two-dimensional lattice paths.
The idea is to interpret the change in the minimum and maximum of the reachable points, as a change in the $x$- and $y$-direction.
We get the following mapping from N-steps to 2D-steps:
\begin{align}
    \label{eq:DyckNWalkBijection}
    \{-1\} &\mapsto (-1,-1), &
    \{1\} &\mapsto (1,1), &
    \{-1,1\} &\mapsto (-1,1).
\end{align}
To finalize the bijection, we map the origin to the origin. 
An example of this bijection is shown in Figure~\ref{fig:DyckNWalkMeander2D}.
This bijection will help us to study N-meanders and N-excursions in the next section.

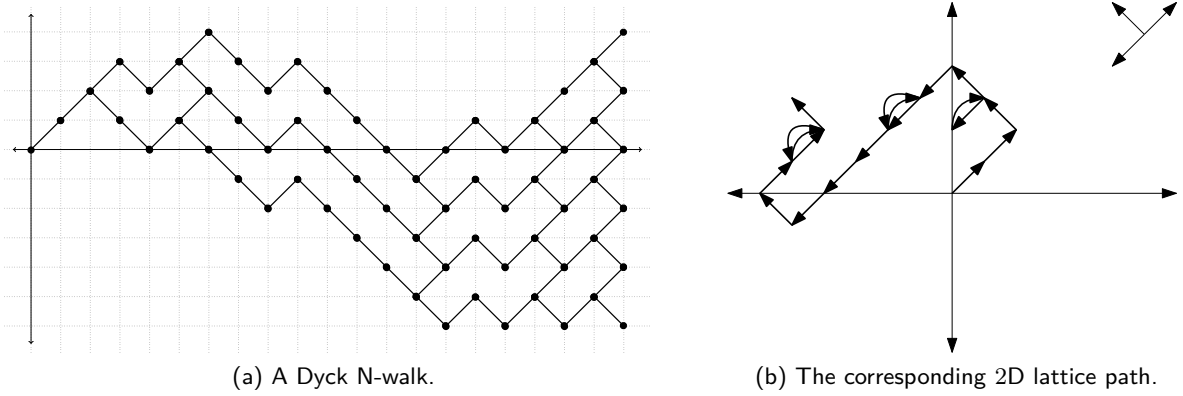
\begin{figure}[ht]
	\centering
	\subfloat[A Dyck N-walk.]{%
	\centering
    \resizebox{0.55\textwidth}{!}{


\begin{tikzpicture}[shorten >=-3pt,shorten <=-3pt, x=1cm, y=1cm]
    \draw [dotted,gray,ystep=1] (-0.9,-6.9) grid (20.9,4.9) ;
    
    \draw [<->,thick] (-0.5,0) -- (20.5,0);
	\draw [<->,thick] (0,-6.5) -- (0,4.5);

    
    \def\mypath{1,1,0,-1,1,0,-1,-1,1,-1,-1,-1,-1,0,1,1,-1,1,0}
    \def\tmp{0}

    \def\mmm{0}
    \def\nrp{0}  

    \def\i{0}
    \foreach \s in \mypath {
    
        \ifthenelse{\s=0}{
            \foreach \j in {0,...,\nrp}{
                \pgfmathtruncatemacro{\tmp}{\mmm+2*\j}
                \draw [-*, very thick] (\i,\tmp) edge (\i+1,\tmp+1);
                \draw [-*, very thick] (\i,\tmp) edge (\i+1,\tmp-1);               
            }

            \pgfmathtruncatemacro{\tmp}{\mmm-1}
            \xdef\mmm{\tmp}
            \pgfmathtruncatemacro{\tmp}{\nrp+1}
            \xdef\nrp{\tmp}
        }{     
            \foreach \j in {0,...,\nrp}{
                \pgfmathtruncatemacro{\tmp}{\mmm+2*\j}
                \draw [-*, very thick] (\i,\tmp) edge (\i+1,\tmp+\s);           
            }

            \pgfmathtruncatemacro{\tmp}{\mmm+\s}
            \xdef\mmm{\tmp}
        }

        \xdef\i{\i+1}
        
    }

\end{tikzpicture}

    
   }
		\label{fig:DyckNWalkto2D}}
        \hfill
    \subfloat[The corresponding $2$D lattice path.]{%
    \includegraphics[width=0.38\textwidth]{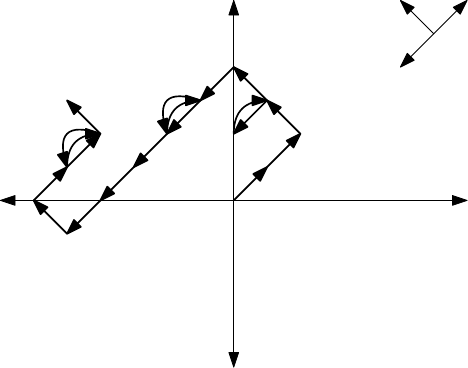}
		\label{fig:DyckNWalk2D}}
    \caption{The bijection of Dyck N-walks to two-dimensional lattice paths transforms the N-steps $\{-1\},\{1\},\{-1,1\}$ into the steps $(-1,-1),(1,1),(-1,1)$, respectively.}
   \label{fig:DyckNWalkMeander2D}
\end{figure}


We now turn our attention to Dyck N-bridges. Their generating function is defined as
\[
	B(x,y,t) = \sum_{n,k,\ell \geq 0} b_{n,k,\ell} x^{-k} y^{\ell} t^{n}.
\]
Recall from Lemma~\ref{lemma:minmax} that bridges have to be of even length, \ie $[t^{2n+1}]B(x,y,t) = 0$. 
Moreover, by~\eqref{eq:bridgesandexcursions} they are linked to the generating function of N-walks as follows
$
	[t^{2n}]B(x,y,t) = [x^{\leq 0} y^{\geq 0} t^{2n}] D(x,y; t).
$
In the following theorem we will reveal a great contrast to classical walks: nearly all N-walks are N-bridges.

\begin{theorem}
	\label{theo:DyckNBridges}
	The generating function of Dyck N-bridges $B(x,y,t)$ is algebraic of degree~$4$. 
    For $\wt{-1}=\wt{1}=\wt{-1,1}=1$ the generating function $B(1,1,t)$ is algebraic of degree $2$ (see \OEIS{A368164}\footnote{The On-Line Encyclopedia of Integer Sequences: \url{http://oeis.org}.}):
    \begin{align*}
    	B(1,1,t) &= \frac{1-6t^2}{\sqrt{1-8t^2}(1-9t^2)}
       = 1 + 7t^2 + 63t^4 + 583t^6 + 5407t^8 + \ldots.
    \end{align*}
    The number $[t^n] B(1,1,t)$ of unweighted Dyck N-bridges is asymptotically equal to
    \begin{align*}
    	\frac{1+(-1)^n}{2} \left( 3^n - \frac{2\sqrt{2}}{\sqrt{\pi}} \frac{8^{n/2}}{\sqrt{n}} + \bigO\left(\frac{8^{n/2}}{n^{3/2}}\right) \right).
    \end{align*}
\end{theorem}

\begin{proof}
In order to improve readability we drop the parity condition on $t$ and define
\[
  \Bb(x,y,t) := [x^{\leq 0} y^{\geq 0}] D(x,y; t).
\]
Then, $[t^{2n}]\Bb(x,y,t) = [t^{2n}]B(x,y,t)$, while in general $[t^{2n+1}]\Bb(x,y,t) \neq 0$ and is therefore not equal to $[t^{2n+1}]B(x,y,t)$.
The key observation is the following link between N-bridges and N-walks:
An N-bridge is an N-walk 
whose minimum is neither strictly positive, nor is its maximum strictly negative.\footnote{We thank \anonymous{Mireille Bousquet-M\'elou}{a momentarily anonymous (for the submission) professor} for suggesting us this approach.} 
This gives
\begin{align} \label{eq:NDyckBridgesInterpretation}
	\begin{aligned}
	\Bb(x,y,t) &= D(x,y; t) - [x^{>0}] D(x,y,t)  - [y^{<0}] D(x,y,t).
    \end{aligned}
\end{align}
Note that in general we would need to correct the right-hand side by $[x^{>0} y^{<0}]D(x,y,t)$, which is however equal to zero by construction, as the minimum can never be larger than the maximum.
Recall that $D(x,y,t)$ has a rational closed form given in~\eqref{eq:DyckNWalkGF} that admits a bijective interpretation in terms of two-dimensional lattice paths. 
Then, Equation~\eqref{eq:NDyckBridgesInterpretation} means that the 2D-walk has to end in the second quadrant. 
To simplify the remaining discussion we introduce the following shorthand:
\begin{align*}
    F(x,y,t) &:= [x^{>0}] D(x,y,t)
    &&
    \text{ and } 
    &
    G(x,y,t) &:= [y^{<0}] D(x,y,t).
\end{align*}

Then, $F(x,y,t)$ corresponds to 2D-walks ending in the right half-plane, and $G(x,y,t)$ to 2D-walks ending in the lower half-plane.
Now, the theory of 
formal Laurent series with positive coefficients (which applies to these problems) implies automatically that they are algebraic. 
Therefore, the generating function of bridges is algebraic; see, \eg~\cite[Section~6]{Gessel80} and Section~\ref{sec:introalgebraic}.

It remains to compute the generating functions explicitly and asymptotically evaluate their coefficients. 
First, observe that due to the symmetry of the step set we have $F(x,y,t)=G(1/y,1/x,t)$ after additionally interchanging the role of $\wt{-1}$ and $\wt{1}$. 
Thus, it suffices to compute $F(x,y,t)$. 
Following~\cite{BM10}, this can be achieved by computing the roots of the denominator of $D(x,y,t)$ and performing a partial fraction decomposition. 
Note that $D(x,y,t)$ is a formal power series in $t$ with Laurent polynomial coefficients in $x$ and $y$, \ie $D(x,y,t) \in \mathbb{Q}[x,1/x,y,1/y][[t]]$.
Finally, we compute the asymptotics using singularity; see Section~\ref{sec:introsingana}.
These computations are performed in the accompanying Maple worksheet~\cite{gitlabproject}.
\end{proof}


\subsection{Dyck N-meanders and N-excursions}

We continue by deriving a functional equation for Dyck N-meanders using a recursive step-by-step decomposition.

\begin{proposition} \label{th:Dyck_Nmeanders}
The generating function of Dyck N-meanders is characterized by the relation
\begin{align}
  D^+(x,y; t) =
  1 &+
  t \left( p_{-1} x^{-1} y^{-1} + p_1 x y + p_{-1,1} x^{-1} y \right) \left(D^+(x,y; t) - D^+(0,y; t)\right) \notag
  \\ &+
  t \left(p_{-1} x y^{-1} + (p_1 + p_{-1,1}) x y \right) \left(D^+(0,y; t) - D^+(0,0; t)\right) \label{eq:DyckNMeandersFuncEq}\\
  &+
  t \left( p_1 + p_{-1,1} \right) x y D^+(0,0; t). \notag
\end{align}
\end{proposition}

\begin{proof}
Applying the symbolic method (see Section~\ref{sec:introsymbmethod}),
we translate the following combinatorial characterization
of N-meanders into the claimed equation.
An N-meander is either of length $0$,
or it can be uniquely decomposed into an N-meander $w$
followed by an N-step.

If $\min^+(w)>0$ then any N-step can be appended.
The generating function of N-meanders
with positive minimum reachable point
is $D^+(x,y; t) - D^+(0,y; t)$.
If $\min^+(w)=0$ but $\max^+(w)>0$, which corresponds to the generating function $D^+(0,y; t) - D^+(0,0; t)$,
then an additional N-step $\{-1\}$ increases $\min^+(w)$
(the path ending at~$0$ disappears, and the one ending at~$2$ becomes the minimum) and decreases $\max^+(w)$,
while an additional N-step $\{1\}$ or $\{-1,1\}$ increases both $\min^+(w)$ and $\max^+(w)$.
Finally, if $\min^+(w)=\max^+(w)=0$,
which corresponds to the generating function $D^+(0,0; t)$,
then the N-step $\{-1\}$ is forbidden,
and the two other available N-steps both increase $\min^+(w)$ and $\max^+(w)$ by one.
\end{proof}

The representation~\eqref{eq:DyckNMeandersFuncEq} admits again an interpretation in terms of two-dimensional lattice paths.
For this purpose, we use the bijection~\eqref{eq:DyckNWalkBijection} and add spatial constraints:
The $x$-axis acts as an absorbing barrier, while the $y$-axis as a reflecting one; see Figure~\ref{fig:DyckReflectionAbsorption}.
This ensures that all paths stay in the first quadrant.
In particular, N-excursions are mapped to walks that end on the nonnegative $y$-axis.

\begin{figure}[ht]
	\centering
    \includegraphics[width=0.38\textwidth]{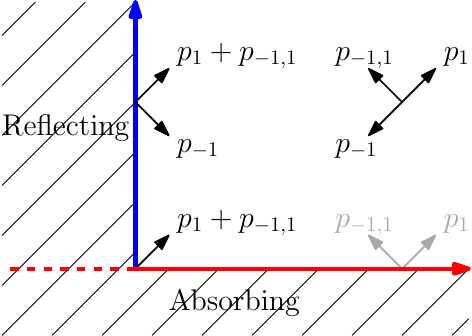}
    \caption{An N-Dyck meander corresponds to a two-dimensional walk with an absorbing $x$-axis and a reflecting (half) $y$-axis and jump polynomial $S(x,y)=p_{1} xy + p_{-1,1} \frac{y}{x} + \frac{p_{-1}}{xy}$.}
   \label{fig:DyckReflectionAbsorption}
\end{figure}

There is rich literature on such two-dimensional models and variations. 
Most variations make the problem much more complicated and often change the nature of the generating function. 
It is known that the generating function of the same steps as above with both barriers being absorbing is algebraic~\cite{BM10}, while there are other cases where it does not even satisfy a (linear) differential equation~\cite{MishnaRechnitzer09Nonholonomic,Dreyfus2017Walks}. 
Furthermore, one can add weights to the steps~\cite{CourtielEtAl17Weighted}, but also let the path accumulate  weights every time it touches one of the axes~\cite{BOR18,TOR14,XBO19}.
However, for the problem of a changing step sets depending on the current spatial position not much is known. Recently, the problem where the position depends on a deterministic automaton was treated for some special cases in~\cite{buchacher2018inhomogeneous}, yet this does not fit into our framework above. 
Only in a one-dimensional setting with an absorbing or reflecting barrier it was shown that the generating functions are always algebraic~\cite{bawa15c}. 
Since this step set is with respect to the spatial constraints only one-dimensional ($x$-positivity implies $y$-positivity), we expect a similar result. 
Yet, the result does not follow directly, as there are several boundary constraints and not just one here.

Let us introduce the \emph{min-max-change polynomial} $S(x,y)$
and the \emph{kernel} $K(x,y)$:
\begin{align} 
    \label{eq:minmaxchangepoly}
    \begin{aligned}
	S(x,y) &:= \frac{p_{-1}}{x y} + p_{1} x y + p_{-1,1} \frac{y}{x},
  \\
  K(x,y)  &:= xy(1- t S(x,y)). 
    \end{aligned}
\end{align}
%
The generating function of Dyck N-walks admits now the compact form
\begin{align}
    \label{eq:NWalksGeneric}
    D(x,y;t) = \frac{1}{1 - t S(x,y)}.
\end{align}

\begin{remark}[N-walks for general step sets]
    \label{rem:NWalksGeneric}
    Note that for any finite N-step set $S$ where each step $s \in S$ has weight $p_s$, we can define the min-max-change polynomial $S(x,y)$ analogously: 
    \begin{align*}
        S(x,y) &= \sum_{s \in S} p_s \, x^{\min(s)} y^{\max(s)}.
    \end{align*}
    Then, equation~\eqref{eq:NWalksGeneric} for the generating function of N-walks with this step set stays valid, as each N-walk is simply the (unconstrained) concatenation of N-steps.
\end{remark}

A key role in the following result on the closed form of Dyck N-meanders
is played by $Y(t)$ and $X(y,t)$, the unique power series solutions satisfying $K(1,Y(t)) = 0$, and $K(X(y,t),y)=0$ which have the following closed forms
\begin{align}
\label{eq:XYDyck}
\begin{aligned}
	Y(t) &= \frac{1-\sqrt{1-4p_{-1}(p_{1}+p_{-1,1})t^2}}{2(p_{1}+p_{-1,1})t},
  \\
  X(y,t) &= \frac{1-\sqrt{1-4p_{1}(p_{-1}+p_{-1,1}y^2)t^2}}{2p_{1} y t}.
\end{aligned}
\end{align}

\begin{theorem} \label{th:other_Dyck_Nmeanders}
	The generating function $D^+(x,y;t)$ of Dyck N-meanders is algebraic of  degree~$4$ (and degree $2$ if $p_i=0$ for some $i \in \{-1, 1, (-1, 1)\}$), and equal to
    \begin{align*}
    	 D^+(x,y;t) &= 
    	 \begin{cases}
    	    \frac{x-X(y,t)}{1-X(y,t)^2} \frac{y-xY(t)-X(y,t)Y(t)+xyX(y,t)}{xy(1-tS(x,y))}, &
    	    \text{ for } p_{1}>0,\\
    	    \frac{y^2-xyY(t)-t(p_{-1,1} y^2+p_{-1})(Y(t)-xy)}{y^2-t^2(p_{-1,1} y^2+p_{-1})^2}, &
    	    \text{ for } p_{1}=0.
    	  \end{cases}
    \end{align*}
    The generating functions $D^+(1,y,t)$ and $D^+(0,0,t)$ are algebraic of degree $2$.
    The generating function $D^+(0,1;t)$ of Dyck N-excursions is algebraic of degree $4$ (degree $2$ for $p_{1}=p_{-1})$.
\end{theorem}

\begin{proof}
Using the kernel $K(x,y)$ from~\eqref{eq:minmaxchangepoly}, we first rewrite the functional equation~\eqref{eq:DyckNMeandersFuncEq} into
\begin{align*}
    K(x,y)D^+(x,y; t) = xy + t(x^2-1)(p_{-1,1} y^2 + p_{-1}) D^+(0,y; t) - tx^2 p_{-1} D^+(0,0; t).
\end{align*}
We observe that after substituting $x=1$ the unknown $D^+(0,y; t)$ vanishes, and we are left with the equation
\begin{align*}
    K(1,y)D^+(1,y; t) = y - t p_{-1} D^+(0,0; t).
\end{align*}
Now, we can use the kernel method (see Section~\ref{sec:introkernel}): 
Substituting $y=Y(t)$ from~\eqref{eq:XYDyck}, sets the left-hand side to zero, and allows us to solve for $D^+(0,0;t)$. 
This is legitimate, as $D^+(x,y;t)$ is a power series in $t$ with polynomial coefficients in $x$ and $y$.
We get
\begin{align}
    \label{eq:DyckD00}
    D^+(0,0;t) = \frac{Y(t)}{p_{-1}t} 
    \qquad \text{ and } \qquad
    D^+(1,y;t) = \frac{y-Y(t)}{K(1,y)}.
\end{align}  
Hence, these two generating functions are algebraic of degree $2$.
After substituting these results back into the initial equation we apply the kernel method again with respect to the variable $x$.
Substituting now $x=X(y,t)$ from~\eqref{eq:XYDyck} we find a closed-form solution for $D^+(0,y;t)$. 
Note that for $p_{1}>0$ the kernel is quadratic in $x$, while for $p_{1}=0$ it is linear. Combining these results proves the claim on $D^+(x,y;t)$.
From this expression, we directly get the generating function $D^+(x,y;t)$ of N-excursions, and a short computation in a computer algebra system like Maple shows the claimed algebraic degrees.
All mentioned computations can be found in the accompanying Maple worksheet~\cite{gitlabproject}.
\end{proof}

Note that by construction of N-meanders, if the maximum of the reachable points is $0$ also the minimum is $0$. 
In terms of generating functions, we have $D^+(x,0;t) = D^+(0,0;t)$.

\begin{remark}[Case $p_{1}=0$]
    The closed form for $D^+(x,y;t)$ in the case of $p_{1}=0$ is the same as the limit $p_{1} \to 0$ of the form $p_{1}>0$. 
    Hence, we do not distinguish these cases from now on, and the case $p_{1}=0$ should be interpreted as $p_{1} \to 0$.
\end{remark}

\begin{corollary}
    \label{cor:DyckNExcSymmetric}
    The generating function of Dyck N-excursions $D^+(0,1;t)$ is symmetric in $p_{-1}$ and~$p_{1}$.
\end{corollary}

\begin{proof}
    This follows immediately from the closed form of Theorem~\ref{th:other_Dyck_Nmeanders}. 
    Alternatively, we will now give a combinatorial explanation of this fact not using the generating function.
    
    Define the following involution $\Phi : S_D \mapsto S_D$:
    \begin{align*}
        \Phi(\{-1\}) &= \{1\}, &
        \Phi(\{1\}) &= \{-1\}, &
        \Phi(\{-1,1\}) &= \{-1,1\}.
    \end{align*}    
    Note that $\Phi$ is a generalization of a well-known involution on classical Dyck paths, reading the path right-to-left (or, equivalently, flipping the path horizontally); see~\cite{Banderieretal2020Latticepathology} for more such constructions.
    Now, let an N-excursion $w = (\vw_1,\dots,\vw_n)$ be given.
    The idea is now to reverse the time.
    We define another N-walk $w' = (\vw'_1,\dots,\vw'_n)$ by $\vw'_i = \Phi(\vw_{n+1-i})$.
    
    Let $v = (v_1,\dots,v_n)$ be an excursion compatible with $w$, \ie a classical excursion such that $v_i \in \vw_i$ for all $i=1,\dots,n$.
    Now, let $v'$ be the time-reversed classical excursion of $v$ (\ie perform a horizontal flip).
    Then, $v'$ is compatible with $w'$, as if $v_i \in \{-1\}$ then $v'_i \in \{1\}$ and vice versa, and if $v_i \in \{-1,1\}$ then $v'_i \in \{-1,1\}$.
    Hence, $w'$ is an N-excursion.
    Finally, note that this involution swaps the N-steps $\{-1\}$ and $\{1\}$ and therefore also the weights $p_{-1}$ and $p_{1}$.
\end{proof}

\begin{corollary}
    \label{cor:DyckNMeanderpq}
    The generating functions $D^+(1,y,t)$ and $D^+(0,0,t)$ depend only on $p_{-1}$ and $q = p_{1}+p_{-1,1}$.
\end{corollary}

\begin{proof}
    This result follows again directly from the closed form of Theorem~\ref{th:other_Dyck_Nmeanders}. 
    Alternatively, the claim (even for general N-steps) follows directly from Theorem~\ref{theo:NMeanderalgebraic} and Remark~\ref{rem:NMeanderColors}.
\end{proof}
 
Next we will answer the counting problem in the unweighted model. 
Moreover, we discover bijections to subfamilies of \emph{bicolored Dyck walks}, which are classical Dyck walks in which the up steps have one of two possible colors.

\begin{corollary}
    \label{cor:NMeandersDyckOnlyt}
  For $p_{-1}=p_{1}=p_{-1,1}=1$ the generating functions of unweighted Dyck N-meanders, N-excursions, and N-excursions with reachable set $\{0\}$ are
  \begin{align*}
      D^+(1,1,t)
      &= -\frac{1-4t-\sqrt{1-8t^2}} {4t(1-3t)} 
      = 1 + 2t + 6t^2 + 16t^3 + 48t^4 + \ldots, 
      & (\text{\OEISs{A151281}})
      \\
      D^+(0,1,t) &= \frac{1-8t^2-(1-12t^2)\sqrt{1-8t^2}} {8t^2(1-9t^2)} 
      = 1 + 4t^2 + 28t^4 + 224t^6 + 1888t^8 + \ldots, 
      & (\text{\OEISs{A368234}})
      \\
      D^+(0,0,t) 
      &= \frac{1-\sqrt{1-8 t^2}}{4 t^2} 
      = 1 + 2t^2 + 8t^4 + 40t^6 + 224t^8 + \ldots      
      & (\text{\OEISs{A151374}})
  \end{align*}
    Asymptotically, we get
  \begin{align*}
  	[t^n]D^+(1,1,t) &= \frac{3^n}{2} + (1 + (-1)^n) \left(3 \sqrt{2} + 4 \right) \frac{8^{n/2}}{\sqrt{\pi n^3}} + \bigO\left(\frac{8^{n/2}}{n^{5/2}}\right), \\
  	[t^n]D^+(0,1,t) &= \frac{1+(-1)^n}{2}\left( \frac{3^n}{4} + 4\sqrt{2} \frac{8^{n/2}}{\sqrt{\pi n^3}}  + \bigO\left(\frac{8^{n/2}}{n^{5/2}}\right) \right), \\
  	[t^n]D^+(0,0,t) &= (1+(-1)^n) \frac{\sqrt{2} 8^{n/2}}{\sqrt{\pi n^3}} \left( 1 - \frac{9}{4n} + \bigO\left(\frac{1}{n^2}\right)\right).
  \end{align*}
  Furthermore, N-meanders are in bijection with bicolored Dyck meanders of the same length enumerated by~\OEIS{A151281},
  and N-excursions ending in $\{0\}$ are in bijection with bicolored Dyck excursions enumerated by~\OEIS{A151374}.
\end{corollary}

\begin{proof}
    The asymptotic results follow by singularity analysis directly from the closed forms; see Section~\ref{sec:introsingana}. 
    The idea of the bijections is to follow the top-most trajectory in the N-meanders. 
    In terms of steps, we map the $\{-1\}$ to $(1,-1)$, and both $\{-1,1\}$ and $\{1\}$ to $(1,1)$, each with a different color.
\end{proof}

\begin{remark}[Classical excursions contained in N-excursions]
    Let us answer the question of how many classical excursions are contained in an N-excursion in the unweighted model, \ie $p_{-1}=p_{1}=p_{-1,1}=1$.
    Let $c_{2n}$ be the total number of classical excursions contained in all N-excursions of length $2n$. 
    To compute them, we interpret every $\{-1,1\}$-N-step either as a classical up- or down-step. 
    Hence, in total there are $2$ possible up- and $2$ possible down-steps,
    and therefore $c_{2n} = \frac{4^n}{n+1} \binom{2n}{n}$.
    We conclude that the average number of classical excursions among all N-excursions of length $2n$ is asymptotically equal to
    \begin{align*}
        \frac{c_{2n}}{[t^{2n}]D^+(0,1,t)} \sim \frac{4}{\sqrt{\pi n^3}} \left(\frac{4}{3}\right)^{2n}.
    \end{align*}
    Therefore, on average, a random N-excursion contains an exponential amount of classical excursions. 
\end{remark}

Let us now interpret the weights as probabilities. Again we observe significant simplifications.

\newcommand{\exGr}{R}
\begin{corollary}
    For $p_1+p_{-1}+p_{-1,1}=1$ with $p_1, p_{-1}, p_{-1,1} \geq 0$, the (probability) generating function $D^+(1,1,t)$ of Dyck N-meanders depends only on $p_{-1}$ and is given by
  \begin{align*}
      D^+(1,1,t) &= - \frac{1 - 2(1-p_{-1})t - \sqrt{1-4p_{-1}(1-p_{-1})t^2}}{2(1-p_{-1})t(1-t)}.
  \end{align*}
  Let $\exGr := 4p_{-1}(1-p_{-1})$.
  Then, the asymptotic probability of N-meanders of length $n$ satisfies
  \begin{align*}
  [t^n]D^+(1,1,t) =
        \begin{cases}
            \frac{1-2p_{-1}}{1-p_{-1}} + \bigO\left(\frac{\exGr^n}{n^{3/2}}\right) & \text{ if } \quad 0 \leq p_{-1} < \frac{1}{2},\\
            \sqrt{\frac{2}{\pi n}} - \frac{\sqrt{2} \, (2-(-1)^n)}{4 \sqrt{\pi n^3}} + \bigO(n^{-5/2}) & \text{ if } \quad p_{-1} = \frac{1}{2},\\
            \frac{\sqrt{2} \, p_{-1}}{\sqrt{\pi}} \frac{1+(-1)^n+(1-(-1)^n)\sqrt{\exGr}}{(2p_{-1}-1)^2} \frac{\exGr^n}{n^{3/2}} + \bigO\left(\frac{\exGr^n}{n^{5/2}}\right) & \text{ if } \quad \frac{1}{2} < p_{-1} < 1.
        \end{cases}
  \end{align*}
\end{corollary}

\begin{proof}
    After identifying the different regimes of convergence, the proof is a simple application of singularity analysis (see Section~\ref{sec:introsingana}).
	Most of these computations can be performed automatically; see the accompanying Maple worksheet~\cite{gitlabproject}.
\end{proof}

Finally, we come back to one of the starting questions from the networking motivation discussed in Section~\ref{sec:motivation}.

\begin{theorem}
\label{theo:dyckexcasymptitotics}
The probability for a random Dyck N-walk
of length $2n$ to be an N-excursion has for $n \to \infty$ the following asymptotic form where the roles of $p_{-1}$ and $p_{1}$ are interchangeable:
\begin{align*}
    \begin{cases}
        \frac{(1-2p_{1})(1-2p_{-1})}{(1-p_{1})(1-p_{-1})} + \bigO\left(\frac{\left(4p_{-1}(1-p_{-1})\right)^n}{n^{3/2}}\right) &
        \text{ if } \quad 0 \leq p_{1} \leq p_{-1} < \frac{1}{2},\\[2mm]
        \frac{1-2p_{1}}{(1-p_{1})\sqrt{\pi n}} + \bigO\left(\frac{1}{n^{3/2}}\right) &
        \text{ if } \quad 0 \leq p_{1} < \frac{1}{2}$ and $p_{-1} = \frac{1}{2},\\[2mm]
        \frac{1}{\sqrt{\pi n^3}} + \bigO\left(\frac{1}{n^{5/2}}\right) &
        \text{ if } \quad p_{1} = p_{-1} = \frac{1}{2}, \\[2mm]
        \gamma \frac{\left(4p_{-1}(1-p_{-1})\right)^n}{\sqrt{n^3}} + \bigO\left(\frac{\left(4p_{-1}(1-p_{-1})\right)^n}{n^{5/2}}\right) &
        \text{ if } \quad 0 \leq p_{1} < \frac{1}{2} < p_{-1} < 1$ and $p_{-1} + p_{1} \leq 1,
    \end{cases} \\
    \text{ where } \quad \gamma = 
    \begin{cases}
        \frac{1}{\sqrt{\pi}}, & \text{ if } p_{-1}+p_{1}=1,\\  
        \frac{2 p_{-1}}{\sqrt{\pi}} \frac{\sqrt{p_{-1, 1}p_{-1}(1-p_{-1})(p_{-1}-p_{1})}-p_{-1}(1-p_{-1})+(1-p_{1})/2}{(1-p_{1})(2p_{-1}-1)^2}, & \text{ if } p_{-1}+p_{1}<1.\\    
    \end{cases}
\end{align*}
\end{theorem}

\begin{proof}
    Substituting $x=0$ and $y=1$ into $D^+(x,y;t)$ from Theorem~\ref{th:other_Dyck_Nmeanders} we get the generating function of N-excursions. 
    The next step is to apply singularity analysis (see Section~\ref{sec:introsingana}) to compute the asymptotics.   
    For technical reasons it is simpler to work with the function $F(z) = D^+(0,1;\sqrt{z})$. This is legitimate as every excursion has  even length.
    The possible singularities are the roots of the radicands in $Y$ and $X$ from~\eqref{eq:XYDyck}, as well as the root of the denominator of $F(z)$, which leads, respectively, to the following three candidates:
    \begin{align}
        \rho_1 &= \frac{1}{4p_{-1}(1-p_{-1})}, &
        \rho_2 &= \frac{1}{4p_{1} (1-p_{1})}, \label{eq:DyckNExcSingat1} &
        \rho_3 &= 1.
    \end{align}
    Depending on the choice of the weights, different singularities become dominant, by being closest to the origin. 
    This gives the following singular expansions at the respective singularities
    \begin{align}
    \label{eq:DyckExcSingExp}
        F(z) &\!=\!
            \begin{cases}
                \frac{(1-2p_{1})(1-2p_{-1})}{(1-p_{1})(1-p_{-1}) (1-z)} + \bigO(1-z/\rho_1) & \text{ if } 0 \leq p_{1} \leq p_{-1} < \frac{1}{2},\\
                %
                \frac{1-2p_{1}}{1-p_{1}} \frac{1}{\sqrt{1-z}} + \bigO(1) & \text{ if } 0 \leq p_{1} < \frac{1}{2} \text{ and } p_{-1} = \frac{1}{2},\\
                %
                2 - 2\sqrt{1-z} + \bigO(1-z) & \text{ if } p_{1} = p_{-1} = \frac{1}{2},\\
                %
                \gamma_0 + \gamma_1 \sqrt{1-z/\rho_1} + \bigO(1-z/\rho_1) & \text{ if } 0 \leq p_{1} < \frac{1}{2} < p_{-1} < 1 \text{ and } p_{-1} + p_{1} \leq 1,
            \end{cases}
    \end{align}
    where $\gamma_0$ and $\gamma_1= -2\sqrt{\pi} \gamma$ are two explicit constants in $p_{-1}$, $p_{-1,1}$, and $p_{1}$.
	As in the case of N-meanders, most of these computations can be performed automatically, yet the computations are slightly more tedious; see the accompanying Maple worksheet~\cite{gitlabproject}. 
	For more details see the below proof of Theorem~\ref{th:LawDyckExcMax}, where an additional parameter is considered as well.
\end{proof}

In Figure~\ref{fig2} we compare the theoretical results with simulations for three different probability distributions on $p_{-1}$, $p_{1}$, and $p_{-1,1}$. 
These nicely exemplify and confirm three of the four possible regimes of convergence.

\begin{figure*}[ht!]
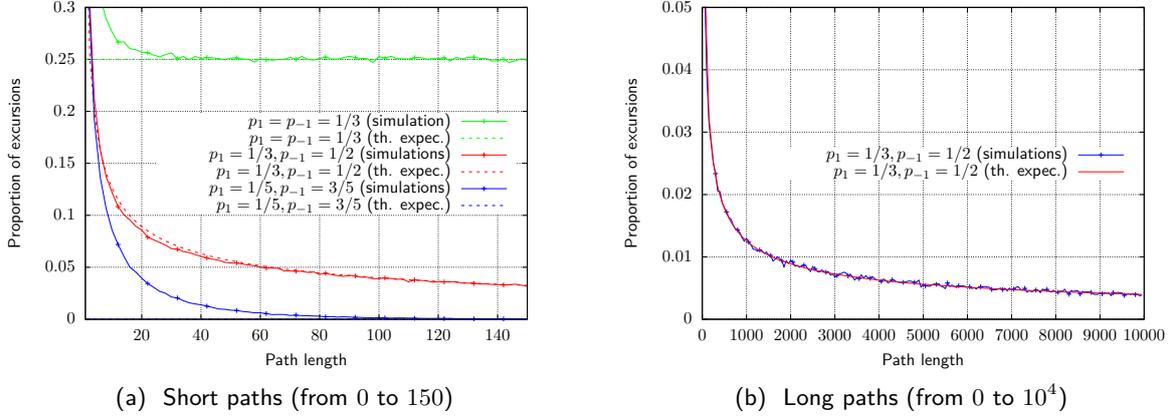

	\centering
	\subfloat[ Short paths (from $0$ to $150$)]{%
    \resizebox{0.47\textwidth}{!}{\input{Dyck_all}}
		\label{fig:short_dyck}}
        \hfill
	\subfloat[ Long paths (from $0$ to $10^4$)]{%
    \resizebox{0.47\textwidth}{!}{\input{Long_Dyck_inequal}}
		\label{fig:long_dyck}}
    \caption{Comparison of theoretical expectation and averaged simulation (over $10^5$ runs) of the proportion of Dyck N-excursions among Dyck N-walks.}
	\label{fig2} 
\end{figure*}

\subsection{Limit laws of N-excursions}
\label{sec:limit_laws_NExcursions}

In order to better understand the influence of the probabilities of the N-steps, we will analyze some parameters in N-excursions.
In particular, we will derive the limit law of the maximum reachable point and the number of returns to zero. 

The following result will depend on the relative change in the minimum and maximum reachable points. 
We define the \emph{$x$-drift $\delta_x= \E(x)$} and the \emph{$y$-drift $\delta_y= \E(y)$} in the probability space defined by the N-step set. 
The drift vector is given by $\delta = (\delta_x,\delta_y)$ where for Dyck N-walks we have
\begin{align}
\label{eq:DyckDriftVector}
\begin{aligned}
    \delta_x &= p_1-p_{-1,1} - p_{-1} = 2p_{1}-1, \\
    \delta_y &= p_1+p_{-1,1} - p_{-1} = 1-2p_{-1}.
\end{aligned}
\end{align}
This leads to the following theorem on the law of the final maximum reachable point, in which several phase transitions depending on the drift vector become visible; see also Figures~\ref{fig:DyckExcLawMax} and \ref{fig:lawsFinalandReturns}.

\begin{theorem}
    \label{th:LawDyckExcMax}
    Let $p_{-1,1} \neq 0$. Let $X_n$ be the random variable of the distribution of the final maximal point of even altitude of an N-excursion of length $2n$ (odd altitude and odd lengths do not exist) drawn uniformly at random defined by
    \begin{align*}
        \proba\left(X_n = k\right) &= \frac{[t^{2n} y^{2k}] D^+(0,y;t)}{[t^{2n}]D^+(0,1;t)}.
    \end{align*}
    Then, $X_n$ admits a limit distribution, with the limit law being dictated by the value of the drift~$\delta$ (see also Figure~\ref{fig:DyckExcLawMax}):
    \pagebreak[2]
    \begin{enumerate}
        \item For negative $y$-drift, $\delta_y < 0$, the limit distribution is discrete given by the following probability generating function that depends on $X(y,t)$ from~\eqref{eq:XYDyck}:
        \begin{align*}
            \proba\left(X_n = k\right) &= [y^{2k}] \omega(y) + \bigO\left(\frac{1}{n}\right),  \qquad \text{ where }\\
            \omega(y) &= \frac{p_{-1} + p_{-1,1}}{p_{-1}+p_{-1,1}y^2} \frac{X(y,r_1)^2}{X(1,r_1)^2} \frac{1-X(1,r_1)^2}{1-X(y,r_1)^2},
        \end{align*}
        and $r_1 = (4p_{-1}(1-p_{-1}))^{-1/2}$. 
        The mean is given by
        \begin{align*}
            \E(X_n) = \frac{2(1-p_{-1})}{2p_{-1}-1} + \frac{2(2p_{-1}-p_{1})}{(2p_{-1}-1)(1-p_{1})} \sqrt{\frac{p_{-1} p_{-1,1}(1-p_{-1})}{p_{-1}-p_{1}}} + \bigO\left(\frac{1}{n}\right).
        \end{align*}
        \item For zero $y$-drift, $\delta_y = 0$, and non-zero $x$-drift,  $\delta_x\neq 0$, the (rescaled) random variable $X_n$ converges in law to a Rayleigh distribution given by
            \begin{align*}
                 \frac{X_n}{\sqrt{n}} &\inlaw \Rc(2), 
            \end{align*}
            where $\Rc(\lambda)$ has the density $\lambda x e^{-\lambda x^2/2}$ for $x \geq 0$.
        \item For zero $x$-drift, $\delta_x = 0$, and non-zero $y$-drift, $\delta_y>0$ (the constraint $\delta_x = 0$ already implies $\delta_y \geq 0$), 
        the (rescaled) random variable $X_n$ converges in law to the convolution of a normal and negative Rayleigh distribution given by
            \begin{align*}
                \frac{X_n - \mu n}{\sqrt{\sigma^2 n}}  &\inlaw \Nc(0,1) * \Rc^-(\lambda), \qquad \text{ where } \\ 
                 \mu = 1-2p_{-1}, \qquad 
                 &\sigma^2 = 2 p_{-1} (1 - 2 p_{-1}), \qquad
                 \lambda = \frac{1-2p_{-1}}{p_{-1}},
            \end{align*}
            and $\Rc^-(\lambda)$ has the density $-\lambda x e^{-\lambda x^2/2}$ for $x \leq 0$, \ie it has negative support.
            Expected value and variance are given by
            \begin{align*}
                \E(X_n) = \mu n - p_{-1} \sqrt{\pi n} + \bigO(1) \quad \text{ and } \quad
                \V(X_n) = p_{-1} (2 - p_{-1} \pi) n + \bigO(\sqrt{n}).
            \end{align*}
        \item For positive $y$-drift, $\delta_y >0$, and non-zero $x$-drift, $\delta_x \neq 0$, the (rescaled) random variable $X_n$ converges in law to a normal distribution given by
        \begin{align*}
            \frac{X_n - \mu n}{\sqrt{\sigma^2 n}} \inlaw \Nc(0,1),
        \end{align*}
        where
        \begin{align*}
            \mu &= 
                \begin{cases} 
                    1-2p_{-1} & \mbox{for } \delta_x<0,\\
                    \frac{p_{-1,1}}{1-p_{1}} & \mbox{for } \delta_x>0,
                \end{cases} 
            &
            \sigma^2 &=  
                \begin{cases} 
                    2p_{-1}(1-p_{-1}) & \mbox{for } \delta_x<0,\\
                    \frac{p_{-1,1} p_{-1}}{(1-p_{1})^2} & \mbox{for } \delta_x>0.
                \end{cases} 
        \end{align*}
    \end{enumerate}
\end{theorem}

\begin{proof}
    By Theorem~\ref{th:other_Dyck_Nmeanders} the bivariate generating function of N-excursions with marked final reachable point is
    \begin{align*}
        D^+(0,y,t) &= \frac{X(y,t)}{1-X(y,t)^2} \frac{y-X(y,t)Y(t)}{p_{-1} t}.
    \end{align*}
    As a first step we define $F(u,z) = D^+(0,\sqrt{u},\sqrt{z})$ as all excursions are of even length end have their maximum therefore at even altitude.
    The possible singularities with respect to $t$ are given by the square-root singularities of $Y(\sqrt{z})$ and $X(\sqrt{u},\sqrt{z})$, as well as the polar singularity arising from the roots of $1-X(\sqrt{u},\sqrt{z})^2$. By~\eqref{eq:XYDyck} we get the following three respective candidates
    \begin{align}
        \rho_1 &= \frac{1}{4p_{-1}(1-p_{-1})}, \notag \\
        \rho_2(u) &= \frac{1}{4p_{1} (p_{-1} + (1-p_{1}-p_{-1})u)}, \label{eq:DyckNExcSing}\\
        \rho_3(u) &= \frac{u}{\left(p_{-1}+(1-p_{-1})u\right)^2}, \notag
    \end{align}
    where we have used that $p_{-1}+p_{1}+p_{-1,1}=1$ to eliminate $p_{-1,1}$.
    Note that $\rho_3(u)$ is only the singularity for $Re(u)>1$, as otherwise it is cancelled by a root in the numerator.
    Recall that for $u=1$ these singularities specialize to the ones in~\eqref{eq:DyckNExcSingat1}, which were the key to obtain Theorem~\ref{theo:dyckexcasymptitotics}.
    Now, we perform a more detailed analysis in the same spirit. 
    For this purpose, we distinguish five different cases depending on the drift vector; compare Figure~\ref{fig:DyckExcLawMax} and the accompanying Maple worksheet~\cite{gitlabproject}.
    
    \begin{enumerate}
        \item 
            In the case $\delta_y<0$ we have $p_{-1}>1/2$, and hence $p_{1}<1/2$. 
            An elementary computation shows that $\rho_1$ is dominant at $u=1$. Then, as $\rho_2(u)$ and $\rho_3(u)$ are continuous at $u=1$, we see that $\rho_1$ remains dominant in a small neighborhood.
            Therefore, the square-root singularity of $Y$ dominates.
            This smooth asymptotic behavior allows to extract the $n$th coefficient with respect to $z$ using real analysis asymptotics as in~\cite[Theorem~VI.12]{FS09}.
        %
        \item \label{item:Rayleigh}
            For $\delta_y=0$ and $\delta_x \neq 0$ we have $p_{-1}=1/2$ and $p_{1}<1/2$. 
            Then, at $u=1$ we see $\rho_1=\rho_3(1)=1$, while 
            $\rho_2(1)=1/(4p_{1}(1-p_{1})) >1$.
            Thus, the square-root singularity of $Y$ in the numerator and the polar singularity in the denominator coalesce leading to the following local expansion
            \begin{align}
                \label{eq:DrSoExpansion}
                F(u,z) = \frac{1}{g(u,z) + h(u,z)\sqrt{1-z}},
            \end{align}
            for $|u-1|<\varepsilon$, $|z-1|<\varepsilon$, $|\arg(z-1)|>3\pi/8$,
            where $0<\varepsilon<\sqrt{2}-1$ fixed, and $g(u,z)$ and $h(u,z)$ are analytic functions.
            Note that a better choice for the $\Delta$-domain is given by the angle $\phi = \arctan\left(\frac{1-\sqrt{1-2\varepsilon-\varepsilon^2}}{\varepsilon}\right)$ instead of $3\pi/8$.
            The analytic continuation is justified by the explicit structure of $X$ and $Y$.
            This situation is captured by a slight generalization of \cite[Theorem~1]{DrSo97} leading to a Rayleigh distribution with parameter $\lambda=2$.
            The generalization amounts to changing the slit plane given by $\arg(z-1) \neq 0$ to a $\Delta$-domain $|\arg(z-1)|>\phi>0$. 
            The proof stays exactly the same. 
        %
        \item
            We continue with $\delta_y>0$ and $\delta_x<0$, \ie $p_{-1}<1/2$ and $p_{1}<1/2$.
            In this case the polar singularity $\rho_3(u)$ dominates.
            Using Maple we get the local expansion for $z \sim \rho_3(u)$
            \begin{align*}
                F(u,z) &= \frac{((1-p_{-1})u-p_{-1}) ((p_{-1,1}-p_{1})u+p_{-1})}{((1-p_{-1})u+p_{-1})^2 (p_{-1,1}u+p_{-1}) (1-p_{-1})\rho_3(u)} \frac{1}{1-z/\rho_3(u)} + \bigO(1).
            \end{align*}
            Then, we may apply Hwang's Quasi-powers theorem; see~\cite[Theorem~IX.8]{FS09} or~\cite{H98}.
            Note that $B(u)=1/\rho_3(u)$ and hence satisfies the variability condition $B''(1)+B'(1)-B'(1)^2 = 2p_{-1}(1-p_{-1}) \neq 0$.
            Therefore, we get convergence to a normal distribution with the claimed parameters.
        %
        \item
            For $\delta_y>0$ and $\delta_x=0$ we have $p_{-1}<1/2$ and $p_{1}=1/2$. Then, at $u=1$ we see $\rho_2(1)=\rho_3(1)=1$ while $\rho_1>1$.
            Thus, like in case~\ref{item:Rayleigh} the square-root singularity of $Y$ in the numerator and the polar singularity in the denominator coalesce. 
            Again a local expansion of the form~\eqref{eq:DrSoExpansion} holds, yet this time it leads to the convolution of a Rayleigh and a normal distribution by a similar generalization as above of~\cite[Theorem~3]{DrSo97}.
            Note that the corresponding Rayleigh distribution is supported on the negative real axis.
        %
        \item
            Finally, we consider $\delta_y>0$ and $\delta_x>0$, \ie $p_{-1}<1/2$ and $p_{1}>1/2$.
            We deal with a square-root singularity arising from $X$ at $z=\rho_2(u)$.
            Applying Hwang's Quasi-powers theorem once more yields the final result of a normal distribution.
            We omit the straightforward computations due to their bulky structure. \qedhere
    \end{enumerate}
\end{proof}

\begin{figure}[ht]
	\centering
    \includegraphics[width=0.35\textwidth]{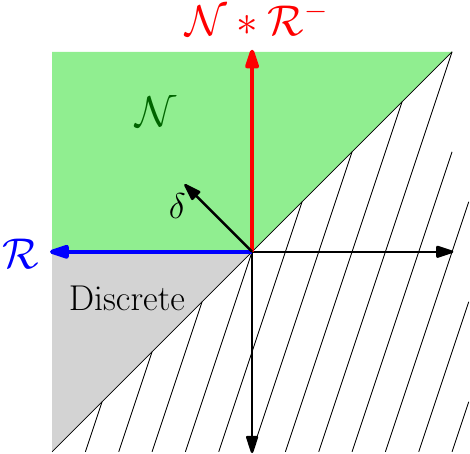}
    \caption{The law of the final maximal point of an N-excursion from Theorem~\ref{th:LawDyckExcMax} depends on the drift vector $\delta=(\delta_x,\delta_y)$ from~\eqref{eq:DyckDriftVector}. The limit law is either a discrete, normal $\Nc$, Rayleigh $\Rc$, or the convolution $\Nc * \Rc^-$ of a normal and a Rayleigh distribution with negative support. Above we see the drift vector $\delta=(-1/3,1/3)$ for the equidistributed case, \ie $S(x,y) = \frac{xy}{3} + \frac{y}{3x} + \frac{1}{3xy}$.}
   \label{fig:DyckExcLawMax}
\end{figure}

\begin{remark}[Phase transitions of the limit law of the maximum reachable point]    
The previous proof shows, that every sector in Figure~\ref{fig:DyckExcLawMax} corresponds to a different dominant singularity from~\eqref{eq:DyckNExcSing}. 
Starting with a drift vector corresponding to the discrete case and continuing in positive direction, we see the following behavior:
First, the dominant singularity is of square-root type arising from~$\rho_1$; then it coalesces with $\rho_3(u)$ creating a different square-root singularity giving a Rayleigh distribution; then it continues as a polar singularity arising from $\rho_3(u)$ giving a normal distribution; which then coalesces with $\rho_2(u)$ to form the convolution of a Rayleigh and a normal distribution; finally, it is again a square-root singularity arising from $\rho_2(u)$ giving a normal distribution.
\end{remark}

\begin{remark}[Non-degeneracy of N-meanders]
    The restriction $p_{-1,1}\neq0$ ensures the nondeterministic character of the walks. If $p_{-1,1}=0$ these are classical excursions where the set of reachable points consists of the single end point.
\end{remark}

\begin{remark}[Limit law of lattice paths in the quarter plane with reflecting/absorbing boundaries]
    The result of Theorem \ref{th:LawDyckExcMax} can also be interpreted in terms of two-dimensional lattice paths defined in Figure~\ref{fig:DyckReflectionAbsorption}. 
    It gives the limit law of the $y$-coordinate of walks ending on the nonnegative $y$-axis that are confined to the quarter plane and have a reflecting $y$-axis and an absorbing $x$-axis (or, for this step set, absorbing origin).
    
    Let us compare it to the model with an absorbing $y$-axis instead, \ie the step set at the $y$-axis is always $p_1 u$.
    The problem is easier in that case, and we invite the reader to solve this problem as a short exercise in the kernel method; see Section~\ref{sec:introkernel}.  
    The resulting closed form shows that the limit law follows a binomial distribution independent of the drift.
    Therefore, it converges in all cases after standardization to a normal distribution. 
\end{remark}

Next we consider the number of returns to zero of a random N-excursion.
Let us define a \emph{zero} of an N-excursion as a time where the set of reachable points is equal to the initial set: $\{0\}$.
Observe that the generating function of Dyck N-excursions that end at a zero is $D^+(0, 0; t)$.
Those N-excursions have a unique decomposition as a sequence of Dyck N-excursions with zeros only at the beginning and the end.
Denoting by $A(t)$ the generating function of those last N-excursions, we deduce
\[
    D^+(0, 0; t)
    =
    \frac{1}{1 - A(t)}
    \qquad \text{and} \qquad
    A(t) = 1 - \frac{1}{D^+(0, 0; t)}.
\]
Let $B(t)$ denote the generating function of Dyck N-excursions where the only zero is at the beginning.
Any Dyck N-excursion has a unique decomposition as a Dyck N-excursion ending at a zero, followed by a (possibly empty) Dyck N-excursion where the only zero is at the beginning, so
\[
    D^+(0, 1; t) =
    D^+(0, 0; t) B(t)
    \qquad \text{and} \qquad
    B(t) = \frac{D^+(0, 1; t)}{D^+(0, 0; t)}.
\]
Let $F(u,t)$ be the bivariate generating function of Dyck N-excursions, where $t$ marks the length and $u$ the number of returns to zero. 
It decomposes as a sequence of N-excursions with zeros only at the beginning and the end, followed by an N-excursion with a zero only at the beginning, so
\[
    F(u,t)
    =
    \frac{1}{1 - u A(t)}
    B(t).
\]
Injecting the expressions of $A(t)$ and $B(t)$ from above, we deduce
\begin{align}
    \label{eq:DyckReturnsBGF}
    F(u,t) &= \frac{1}{1-u \left(1 - \frac{1}{D^+(0,0;t)}\right)} \frac{D^+(0,1;t)}{D^+(0,0;t)}.
\end{align}

We conclude this section with the limit law of the returns to zero of Dyck N-excursions drawn uniformly at random.
In terms of the two-dimensional lattice path defined in Figure~\ref{fig:DyckReflectionAbsorption}, this corresponds to the law of the number of visits of the origin of a path ending on the $y$-axis.
The law is always discrete; see also Figure~\ref{fig:lawsFinalandReturns} for some explicit simulations matching the theoretical results.

\begin{theorem}
    \label{th:LawDyckExcReturns}
   Let $Y_n$ be the random variable of the distribution of the number of returns to zero in an N-excursion of length $2n$ drawn uniformly at random:
    \begin{align*}
        \proba\left(Y_n = k\right) &= \frac{[t^{2n} u^{k}] F(u,t)}{[t^{2n}]D^+(0,1;t)}.
    \end{align*}
    Then, $Y_n$ admits a discrete limit distribution of geometric, negative binomial, or mixed type that is given by
    \begin{align*}
        \proba\left(Y_n = k\right) &= 
             \begin{cases}
                %
                (1-p_{-1}) p_{-1}^k & \text{ if } 0 \leq p_{-1} < \frac{1}{2} \text{ and } 0 \leq p_{1} \leq \frac{1}{2},\\
                \frac{1}{2^{k+1}} & \text{ if } 0 \leq p_{1} < \frac{1}{2} \text{ and } p_{-1} = \frac{1}{2},\\
                %
                \frac{k}{2^{k+1}} & \text{ if } p_{1} + p_{-1} = 1,\\
                %
                \frac{1}{D^+(0,0;\rho_2)} \left( 1 - \frac{1}{D^+(0,0;\rho_2)} \right)^k & \text{ if } 0 \leq p_{-1} < \frac{1}{2} < p_{1} < 1 \text{ and } p_{-1} + p_{1} < 1,\\
                %
                (1-\eta)\frac{1}{2^{k+1}} + \eta\frac{k}{2^{k+1}} & \text{ if } 0 \leq p_{1} < \frac{1}{2} < p_{-1} < 1 \text{ and } p_{-1} + p_{1} < 1,
            \end{cases}
    \end{align*}
    where $\rho_2 = \frac{1}{4p_{1}(1-p_{1})}$ and $\eta = \frac{p_{-1}(p_{-1}-p_{1})-\sqrt{p_{-1}(1-p_{-1})(1-p_{1}-p_{-1})(p_{-1}-p_{1})}}{p_{-1}(1-p_{1})} \in [0,1]$.
\end{theorem}

\begin{proof}
Extracting the $k$th coefficient of $u$ in~\eqref{eq:DyckReturnsBGF} gives
\begin{align*}
    \left(1 - \frac{1}{D^+(0,0;t)}\right)^k  \frac{D^+(0,1;t)}{D^+(0,0;t)}.
\end{align*}
Then, it is tedious but straightforward to use singularity analysis (see Section~\ref{sec:introsingana}) to derive the asymptotic expansions. Therefore, we combine the singular expansions of $D^+(0,1;t)$ from~\eqref{eq:DyckExcSingExp} and of $D^+(0,0;t)$ derived from~\eqref{eq:XYDyck} and \eqref{eq:DyckD00}.
Note that the cases leading to a negative binomial distribution when $p_{1}+p_{-1}=1$ also follow from~\cite[Theorem~5]{BaFl02}.
\end{proof}

\begin{figure*}[ht!]
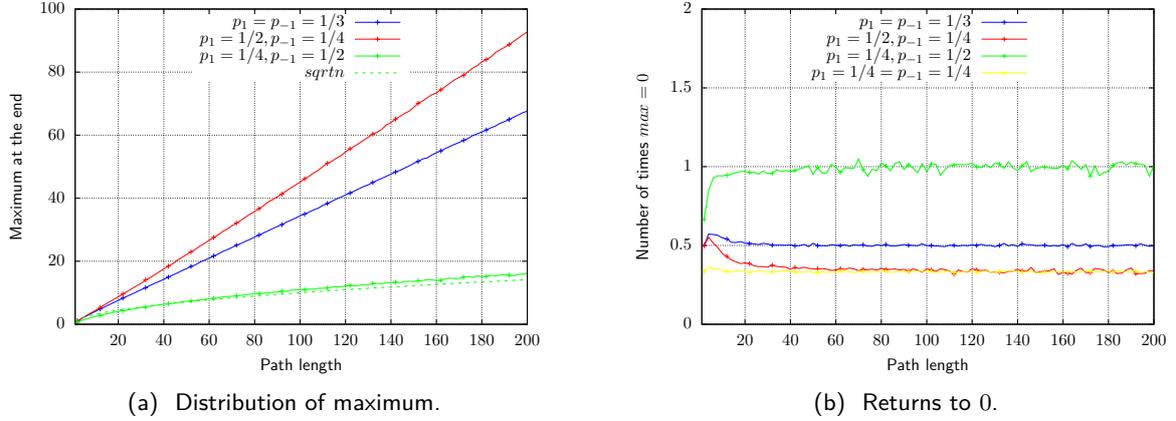

	\centering
	\subfloat[ Distribution of maximum.]{%
    \resizebox{0.47\textwidth}{!}{\input{Figures/Average_max_at_end.tex}}		}
        \hfill
	\subfloat[ Returns to $0$.]{%
    \resizebox{0.47\textwidth}{!}{\input{Figures/number_of_max_to_0.tex}}}
    \caption{Maximum point reached at the end according to path length (left) and returns to $0$ (right) of Dyck N-excursions (averaged simulation over $10^5$ runs).}
	\label{fig:lawsFinalandReturns} 
\end{figure*}

\medskip

After this detailed discussion of nondeterministic walks derived from Dyck paths, we turn to the probably next most classical lattice paths: Motzkin paths.

		\section{Motzkin N-walks}
        \label{sec:MotzkinNWalks}

The step set of classical Motzkin paths is $\{-1,0,1\}$.
The  N-step set of all non-empty subsets is
\[
	S_M = \Big\{ \{-1\}, \{0\}, \{1\}, \{-1,0\}, \{-1,1\}, \{0,1\}, \{-1,0,1\} \Big\},
\]
and we call the corresponding N-walks (N-meanders, N-excursions, N-brides) \emph{Motzkin N-walks} (\emph{Motzkin N-meanders}, \emph{Motzkin N-excursions}, \emph{Motzkin N-bridges}).
As before we associate with every step $\vs \in S_M$ a weight $p_{\vs}$, denoted by $p_{-1}$, $p_{0}$, $p_{1}$, $p_{-1,0}$, $p_{-1,1}$, $p_{0,1}$, and $p_{-1,0,1}$, respectively.
Recall that the weight of a path, is the product of its weights.
In Appendix~\ref{sec:SubclassesNMotzkin} we give conjectural connections to counting sequences in the OEIS for special choices of weights in N-excursions and N-meanders.

To simplify the subsequent discussion, we introduce the following notation.
Given an N-walk $w = (\vw_1, \ldots, \vw_n)$ and an N-step $\vs$, we denote by $w \cdot \vs$ the N-walk that is obtained by adding the N-step $\vs$ at the end of the N-walk $w$, \ie
\[
    w \cdot \vs
    =
    (\vw_1, \ldots, \vw_n, \vs).
\]

\subsection{Motzkin N-walks and N-bridges}

As for Dyck N-walks, we start with a description of the reachable points of Motzkin N-walks.
Here, we need to distinguish two cases.
From now on we will use the shorthand $\|w \| := \max(w) - \min(w)$.
\begin{definition}[Types of reachable points]
\label{def:MotzkinReachablePoints}
A Motzkin N-walk $w$ is 
\begin{itemize}
  \item of \emph{type $\romI$} if $\reach(w) = \left\{\min(w) + i ~:~ i=0, 1, \dots, \|w \|\right\}$ and $\|w \| \geq 1$.
  \item of \emph{type $\romII$} if $\reach(w) = \left\{ \min(w) + 2i ~:~ i=0,1,\dots, \frac{\|w \|}{2} \right\}$,  
\end{itemize}    
\end{definition}

In other words, the sets of reachable points are either $1$- or $2$-periodic finite subsets of integers. 
The condition $\| w \| \geq 1$ guarantees that the types are disjoint.
For Dyck N-walks we did not need to distinguish different types, as they were always of type $\romII$.
Now, the following proposition will explain how these two types suffice to characterize the structure of Motzkin N-walks.

\begin{proposition} \label{th:motzkin_structure}
A Motzkin N-walk $w$ is of type $\romII$ if and only if
it is constructed only from the N-steps  $\{-1\}$, $\{0\}$, $\{1\}$, and $\{-1,1\}$.
Otherwise, it is of type $\romI$; see Figure~\ref{fig:motzkin_walks_structure}.
\end{proposition}

\begin{proof}
We use induction on the number of N-steps.
If the Motzkin N-walk is empty, its set of reachable points
is $\{0\}$, which is of type $\romII$ by definition.
Now consider a Motzkin N-walk $w$
of reachable points $\vr$,
a Motzkin N-step $\vs$,
and denote the set of reachable points of $w \cdot \vs$ by $\vr'$.
Let us assume the proposition holds for $w$
and prove it for $w \cdot \vs$.

If $\vs$ has size $1$, like $\{-1\}$, $\{0\}$ or $\{1\}$,
$\vr'$ is just a translation of $\vr$,
so the type is unchanged.

If $\vr$ is of type $\romI$,
it is an integer interval of size $\ell \geq 2$,
so there exists $m \in \integers$ such that
$\vr = [m, m + \ell]$.
Then for any Motzkin N-step $\vs$,
$\vr'$ is also an integer interval,
of size $\ell + \max(\vs) - \min(\vs)$.

If $\vr$ is of type $\romII$
and $\vs = \{-1, 1\}$, then we saw in Lemma~\ref{lem:dyck_reachable_points}
that $\vr'$ has type $\romII$.
The last case is when $\vr$ has type $\romII$,
say $\vr = \{m, m + 2, \ldots, m + 2 \ell\}$,
and $\vs$ is equal to $\{-1,0\}$, $\{0,1\}$, or $\{-1,0,1\}$.
The value of $\vr'$, in each case,
is $[m-1, m + 2\ell]$, $[m, m + 2\ell+1]$ or $[m-1, m + 2\ell + 1]$.
It is always an interval of length at least $2$,
so it is of type $\romI$.
\end{proof}

\begin{figure}
\begin{center}
\includegraphics[scale=0.42]{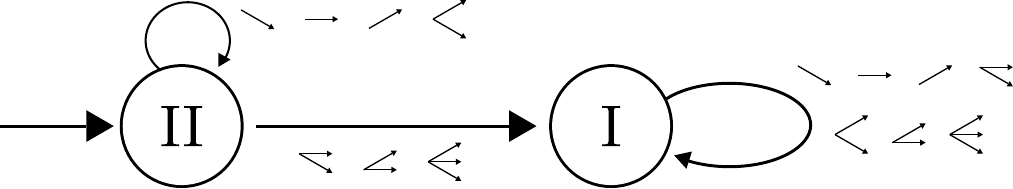}
\caption{The automaton representing the structure of reachable points of Motzkin N-walks.}
\label{fig:motzkin_walks_structure}
\end{center}
\end{figure}

The set of Motzkin N-walks of type $\romI$ (\resp $\romII$) is denoted by $M_{\romI}$ (\resp $M_{\romII}$),
and their generating functions are defined as
\begin{align*}
	M_{\romI}(x,y; t) &= \sum_{w \in M_{\romI}} \bigg( \prod_{\vs \in w} p_{\vs} \bigg) x^{\min(w)} y^{\max(w)} t^{|w|},\\
	M_{\romII}(x,y; t) &= \sum_{w \in M_{\romII}} \bigg( \prod_{\vs \in w} p_{\vs} \bigg) x^{\min(w)} y^{\max(w)} t^{|w|}.
\end{align*}
The generating function for all Motzkin N-walks is defined as $M(x,y;t) = M_{\romI}(x,y; t) + M_{\romII}(x,y; t)$.

\begin{theorem}
    The generating functions of Motzkin N-walks of type $\romI$ and $\romII$, as well as for all Motzkin N-walks are rational.
\end{theorem}

\begin{proof}    
    This result is a direct consequence of the analysis of the interactions of types and N-steps in the proof of Proposition~\ref{th:motzkin_structure}; see Figure~\ref{fig:motzkin_walks_structure}.
    First we translate the changes of minimum and maximum reachable points for each type into into min-max-change polynomials as introduced in Remark~\ref{rem:NWalksGeneric}:
    \begin{align*}
        S_{\romI}(x,y) &= \frac{p_{-1}}{x y} + p_{0} + p_{1} x y + \frac{p_{-1,0}}{x} + p_{0,1}y + (p_{-1,1} + p_{-1,0,1}) \frac{y}{x}, \\
        S_{\romII}(x,y) &= \frac{p_{-1}}{xy} + p_{0} + p_{1} xy + p_{-1,1}\frac{y}{x}.
    \end{align*}
    This directly gives the following system of equations
    \begin{align*}
        M_{\romI}(x,y;t) &=
        t \left( S_{\romI}(x,y) - S_{\romII}(x,y) \right) M_{\romII}(x,y;t) 
        + t S_{\romI}(x,y) M_{\romI}(x,y;t),
    \\
        M_{\romII}(x,y;t) &=        
        1 + t S_{\romII}(x,y) M_{\romII}(x,y;t),
    \end{align*}
    from which we get the following explicit rational solutions:
    \begin{align*}
        M_{\romI}(x,y;t) &= \frac{1}{1-t S_{\romI}(x,y) } - \frac{1}{1-t S_{\romII}(x,y) },\\
        M_{\romII}(x,y;t) &= \frac{1}{1-t S_{\romII}(x,y) }.
    \end{align*}
    Alternatively, these formulas have a direct combinatorial derivation: 
    N-walks of type $\romII$ are sequences of N-steps encoded by $S_{\romII}$. 
    N-walks of type $\romI$ are sequences of N-steps encoded by $S_{\romI}$, except for those that consist solely of N-steps encoded by $S_{\romII}$.    
    Hence, both generating function are rational and therefore also $M(x,y,t) = \frac{1}{1-t S_{\romI}(x,y) }$.
\end{proof}

\begin{theorem}
    \label{theo:MotzkinNBridges}
    The generating function $B_M(x,y,t)$ of Motzkin N-bridges is algebraic of degree~$16$.
    For all weights equal to one the generating function $B_M(1,1,t)$ of Motzkin N-bridges is algebraic of degree $4$:
    Let $\widetilde{B}_m := ( 1 - 6t )  ( 1-2t-7\,{t}^{2} ) \left(( 1+2t) ( 1-4t ) ( 1-7t ) {{B_M(1,1,t)}} \right)^{2}$ then
    \begin{align*}
        \widetilde{B}_m^2 &+2\, 
        \left( 3878\,{t}^{7}-1489\,{t}^{6}-892\,{t}^{5}+402
        \,{t}^{4}+116\,{t}^{3}-86\,{t}^{2}+16\,t-1 \right) \widetilde{B}_m\\
        &+ \left( 3290\,{t}^{7}-47\,{t}^{6}-2052\,{t}^{5}+750\,{t}^
        {4}+72\,{t}^{3}-84\,{t}^{2}+16\,t-1 \right) ^{2}
        =0.    
    \end{align*}
    The number $[t^n] B_M(1,1,t)$ of unweighted Motzkin N-bridges is asymptotically equal to
    \begin{align*}
    	7^n - \sqrt{\frac{3}{\pi}} \frac{6^{n}}{\sqrt{n}} + \bigO\left(\frac{6^{n}}{n^{3/2}}\right) .
    \end{align*}
\end{theorem}

\begin{proof}
An N-bridge $w$ of type $\romII$ is an $M_{\romII}$ N-walk that satisfies $\min(w) \leq 0$, $\max(w) \geq 0$, and $\min(w)$ even.
An N-bridge $w$ of type $\romI$ is an $M_{\romI}$ N-walk that satisfies $\min(w) \leq 0$ and $\max(w) \geq 0$.
Thus, the generating function of Motzkin N-bridges is equal to
\[
    [x^{\leq 0} y^{\geq 0}] \left(\frac{M_{\romII}(x,y; t) + M_{\romII}(-x,y; t)}{2} + M_{\romI}(x,y; t)
    \right).
\]
Since the generating functions of $M_{\romI}$ and $M_{\romII}$ are rational, the generating function of N-bridges is D-finite; see~\cite[Proposition~1]{BM10} and \cite{Lipshitz88}. 
Yet the generating function is even algebraic, which can be proved similarly to as done in the proof of Theorem~\ref{theo:DyckNBridges}.
The main observation is that, as in~\eqref{eq:NDyckBridgesInterpretation}, the simultaneous extraction of nonnegative and nonpositive parts can be rewritten into two single extractions from a rational generating function, which are therefore both algebraic.

The explicit computations, also for general degree are given in the Maple worksheet~\cite{gitlabproject}. 
We show that $B_M(x,y,t)$ lives in the algebraic extension of the base field $\QQ(p_1, \dots, p_{-1,0,1},x,y,t)$ given by the quadratic functions $X(y,t)$, $Y(x,t)$, $\widetilde{X}(y,t)$, and $\widetilde{Y}(x,t)$ that satisfy $1-tS_{\romI}(X,y)=0$, $1-tS_{\romI}(x,Y)=0$, $1-tS_{\romII}(\widetilde{X},y)=0$, and $1-tS_{\romII}(x,\widetilde{Y})=0$, respectively.
From these closed forms it is then straightforward to compute the claimed asymptotics
for all weights set equal to one using singularity analysis (see Section~\ref{sec:introsingana}). 
For this purpose, note that for all weights equal to one and $x=y=1$, it holds that $X(1,t) = 2 \, Y(1,t)$ and $\widetilde{X}(1,t) = 4 \, \widetilde{Y}(1,t)$.
\end{proof}

As the total number of Motzkin N-walks of length $n$ is $7^n$, the previous results shows that nearly all N-walks are N-bridges, as the quotient behaves like
\[
	1 - \sqrt{\frac{3}{\pi}} \frac{(6/7)^n}{\sqrt{n}} + \bigO\left(\frac{(6/7)^n}{n^{3/2}}\right).
\]

We now turn to the analysis of Motzkin N-meanders and N-excursions.

\subsection{Motzkin N-meanders and N-excursions}

The key to enumerate N-meanders is to understand their reachable points.
The difference to N-walks is that N-meanders cannot go below the $x$-axis.
Recall that if an N-step would go below the $x$-axis, only the steps staying weakly above it are appended.
The following proposition shows that the two types from Definition~\ref{def:MotzkinReachablePoints} suffice to characterize the reachable points.
However, the types interact with the N-steps 
differently than before.

\begin{figure}
\begin{center}
\includegraphics[scale=0.45]{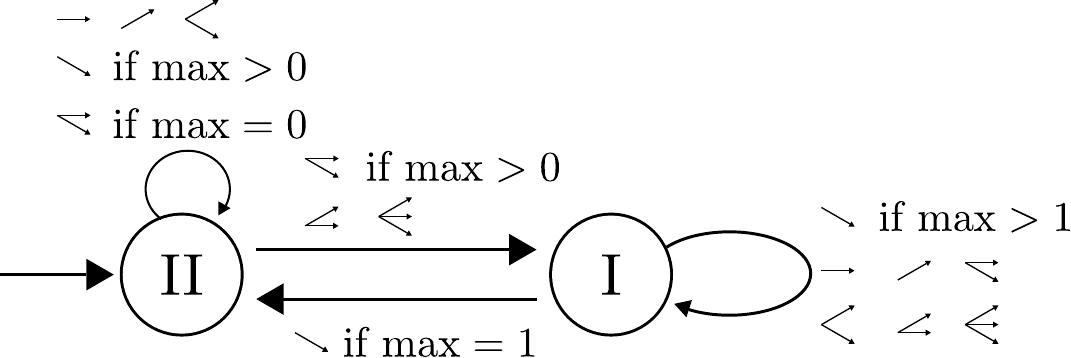}
\caption{The automaton representing the structure of reachable points of Motzkin N-meanders. Here, ``$\max$'' refers to the maximum reachable point of the N-meander.}
\label{fig:motzkin_Nmeanders_structure}
\end{center}
\end{figure}

\begin{proposition}
\label{prop:typesMotzkin}
Let $w$ be a Motzkin N-meander and $\vs \in S_M$.
\begin{itemize}
\item
If $w$ has type $\romII$,
then $w \cdot \vs$ has type $\romII$ in the following cases
\begin{itemize}
\item
$\vs \in \{\{0\}, \{1\}, \{-1, 1\}\}$,
\item
$\vs = \{-1\}$ and $\max^+(\vw) > 0$,
\item
$\vs = \{-1, 0\}$ and $\max^+(\vw) = 0$.
\end{itemize}
\item
If $w$ has type $\romII$,
then $w \cdot \vs$ has type $\romI$ in the following cases
\begin{itemize}
\item
$\vs \in \{\{0, 1\}, \{-1, 0, 1\}\}$
\item
$\vs = \{-1, 0\}$ and $\max^+(\vw) > 0$.
\end{itemize}
\item
If $\vw$ has type $\romI$,
then $w \cdot \vs$ has type $\romII$ if $\vs = \{-1\}$ and $\max^+(\vw) = 1$, otherwise, it has type $\romI$.
\end{itemize}
This result is illustrated by Figure~\ref{fig:motzkin_Nmeanders_structure}.
\end{proposition}
\begin{proof}
    The first two cases follow directly from Proposition~\ref{th:motzkin_structure}. 
    The last case is different, due to the interaction with the boundary, reducing the integer interval to a singleton.
\end{proof}

Let $M_{\romI}^+$ and $M_{\romII}^+$ denote the set of Motzkin N-meanders of type $\romI$ and $\romII$, respectively.
Their generating functions are defined as
\begin{align*}
	M_{\romI}^+(x,y; t) &= \sum_{w \in M_{\romI}^+} \bigg( \prod_{s \in w} p_s \bigg) x^{\min^+(w)} y^{\max^+(w)} t^{|w|},\\
	M_{\romII}^+(x,y; t) &= \sum_{w \in M_{\romII}^+} \bigg( \prod_{s \in w} p_s \bigg) x^{\min^+(w)} y^{\max^+(w)} t^{|w|}.
\end{align*}
The generating function for all Motzkin N-meanders is defined as $M^+(x,y;t) = M_{\romI}^+(x,y; t) + M_{\romII}^+(x,y; t)$.
Now, Proposition~\ref{prop:typesMotzkin} allows us to characterize Motzkin N-meanders and N-excursions.

\begin{theorem} \label{th:motzkin_N_meanders}
    The generating function of Motzkin N-meanders $M_{\romI}^+(x,y; t)$ of type $\romI$ is algebraic of degree $8$ and the generating function $M_{\romII}^+(x,y; t)$ of Motzkin N-meanders of type $\romII$ is algebraic of degree $16$. 
    In particular, Motzkin N-meanders $M^+(x,y;t)$ are algebraic of degree $16$ and N-excursions $M_{\romI}^+(0,y; t)$, $M_{\romII}^+(0,y; t)$, and $M^+(0,y;t)$ are algebraic of degree $16$.
\end{theorem}

\begin{proof}
Observe that $M_{\romI}^+(x,y;t)$
is divisible by $y$,
as N-meanders of type $\romI$
have maximum reachable point at least $1$.
Thus we define the column vector $\vect{M^+}(x,y; t)$ by
\begin{align*}
    \vect{M^+}(x,y; t) = \begin{pmatrix} M_{\romII}^+(x,y; t) \\ y^{-1} M_{\romI}^+(x,y; t) \end{pmatrix}.
\end{align*}
An N-meander is either empty, and associated with type $\romII$, or an N-meander to which we append an N-step.
This translates into a step-by-step construction analogous to the one in the proof of Proposition~\ref{th:Dyck_Nmeanders}, yet now for two types and with interactions between them.
Again we distinguish three cases: 
general case ($\min^+(w)>0$), 
boundary case ($\min^+(w)=0$) but non-minimal maximum ($\max^+(w)>0$ for type $\romII$ or $\max^+(w)>1$ for type $\romI$), 
and boundary case ($\min^+(w)=0$) with minimal maximum ($\max^+(w)=0$ for type $\romII$ or $\max^+(w)=1$ for type $\romI$).
We get the following system of equations characterizing the generating functions of the vector $\vect{M^+}(x,y; t)$:
\begin{align*}
	\vect{M^+}(x,y; t) &= 
    \vect{e_1} + t \Big( 
    A(x,y) (\vect{M^+}(x,y; t) - \vect{M^+}(0,y; t)) \\
    &\quad+ B(x,y) (\vect{M^+}(0,y; t) - \vect{M^+}(0,0; t)) 
    \\ & \quad+ C(x,y) \vect{M^+}(0,0; t) \Big),
\end{align*}
where $\vect{e_1}$ is the column vector $(1,0)$,
and $A(x,y)$, $B(x,y)$, $C(x,y)$ are two-by-two matrices with Laurent polynomials in $x$ and $y$ given in Figure~\ref{fig:matrices}.
These matrices model the interactions of the two types with the N-steps from Proposition~\ref{prop:typesMotzkin} and the change in the minima and maxima of the reachable points.  
\begin{figure*}
%
\begin{align*}
	A(x,y) &= {\small  
        \begin{pmatrix} 
            {\frac {p_{{-1}}}{xy}}+p_{{0}}+p_{{1}}xy+{
            \frac {p_{{-1,1}}y}{x}}
            &
            0\\ 
            \noalign{\medskip}{\frac {p_{{-1,0}}}{xy}}
            +p_{{0,1}}+{\frac {p_{{-1,0,1}}}{x}}
            &{\frac {p_{{-1}}}{xy}}+p_{{0}}+p_
            {{1}}xy+{\frac {p_{{-1,0}}}{x}}+p_{{0,1}}y+
            {\frac {(p_{{-1,1}}+p_{{-1,0,1}})y}{x}}
        \end{pmatrix} },\\[2mm]
    B(x,y) &= {\small 
    \begin{pmatrix} 
        {\frac {p_{{-1}}x}{y}}+p_{{0}}+(p_{{1}}+p_{{-1,1}})xy 
        &
        0\\ 
        \noalign{\medskip}{\frac {p_{{-1,0}}}{y}}+p_{{0,1}}+p_{{-1,0,1}}
        &
        {\frac {p_{{-1}}}{y}}+p_{{0}}+p_{{-1,0}}+p_{{1}}xy+(p_{{0,1}}+p_{{-1,1}}+p_{{-1,0,1}})y
    \end{pmatrix} }, \\[2mm]
    C(x,y) &= {\small 
    \begin{pmatrix} 
        (p_{{1}}+p_{{-1,1}})xy+p_{{0}}+p_{{-1,0}} &
        p_{{-1}}\\ 
        \noalign{\medskip}p_{{0,1}}+p_{{-1,0,1}} &
        p_{{1}}xy+(p_{{-1,1}}+p_{{0,1}}+p_{{-1,0,1}})y+p_{{0}}+p_{{-1,0}}
    \end{pmatrix}}.
\end{align*}
    \caption{Type transition matrices involved in the proof of Theorem~\ref{th:motzkin_N_meanders}.}
		\label{fig:matrices}
\end{figure*}
Next, we rearrange this equation into
\begin{align} 
\label{eq:motzkin_n_meanders}
\begin{aligned}
	\left(\id - t A(x,y)\right) \vect{M^+}(x,y; t) &= 
    \vect{e_1} - t \left(A(x,y) - B(x,y)\right) \vect{M^+}(0,y; t)
  \\& \quad - t \left(B(x,y) - C(x,y)\right) \vect{M^+}(0,0; t). 
\end{aligned}
\end{align}
On the left-hand side, we identify a factorization involving the explicit factor $K(x,y,t) = \id - t A(x,y)$.
If this system would consist of only a single equation, we could apply the classical kernel method; see Section~\ref{sec:introkernel}.
However, for this system we need to generalize this approach to matrix equations; see \cite{AsBaBaGi18} for similar generalizations.
Our idea is to choose specific values of $x$, such that $K(x,y,t)$ becomes singular, and then left-multiply by an element from its kernel to deduce a new identity. 
Due to the triangular nature of $A(x,y)$, the matrix $K(x,y,t)$ is singular when one of the diagonal entries is zero. 
This gives two candidates $X_1 \equiv X_1(y,t)$ and $X_2 \equiv X_2(y,t)$ that satisfy the following two quadratic equations, respectively:
\begin{align*}
    p_{1} t y^2 X_1^2 + (p_{0} t - 1) y X_1 + (p_{-1} + p_{-1, 1}y^2) t &= 0, 
    \\
    p_{1} t y^2 X_2^2 + \left( p_{0, 1} t y^2 + (p_{0} t - 1) y\right) X_2 + (p_{-1, 1} + p_{-1, 0, 1}) t y^2 + p_{-1, 0} t y + p_{-1} t &= 0.
\end{align*}
We always choose the unique power series solution in $t$. 
Note that the final generating functions are power series in $t$, hence we need to choose these two branches.
We then define the row vectors
\begin{align*}
	\vect{u_1} &= (1, 0),
  \\
  \vect{u_2}(y,t) &= \left(t A_{1,0}(X_2(y,t), y), 1 - t A_{0,0}(X_2(y,t), y)\right),
\end{align*}
so that the left-hand side of Equation~\eqref{eq:motzkin_n_meanders}
vanishes when evaluated at $x = X_1(y,t)$ \emph{and} left-multiplied by $\vect{u_1}$,
and also when evaluated at $x = X_2(y,t)$ \emph{and} left-multiplied by $\vect{u_2}(y,t)$.
Combining the corresponding two right-hand sides, we obtain a new two-by-two system of linear equations
\begin{equation} \label{eq:motzkin_n_meanders_second_kernel}
	t D(y,t) \vect{M^+}(0,y; t) = \vect{f}(y,t) -  E(y,t) \vect{M^+}(0,0; t),
\end{equation}
with a vector $\vect{f}(y,t)$ two (explicit) $2 \times 2$-matrices $D(y,t)$ and $E(y,t)$.
Again, the matrix $D(y,t)$ is upper-triangular, and we repeat the process from above now with respect to the variable~$y$. 
Now, $D(y,t)$ becomes singular, when $y$ is evaluated at one of the following branches satisfying the following equations
\begin{align*}
	(p_{-1, 1} + p_{1}) t Y_1^2 + (p_{0} t - 1) Y_1 + p_{-1} t &= 0,
  \\
  \left( p_{{1}}+p_{{0,1}}+p_{{-1,1}}+p_{{-1,0,1}} \right) t{{Y_2}
}^{2}+ \left( (p_{{0}}+p_{{-1,0}})t - 1 \right) {Y_2}+p_{{-1}}t &= 0.
\end{align*}
As before, the branch we will need is the unique power series solution. 
Note that from the algebraic equations it is easy to prove that $X_1(Y_1) = 1$ and $X_2(Y_2)=1$, \ie they are inverse with respect to composition. 
It would be desirable to find a combinatorial argument for this phenomenon.

Next, we deduce the following two row vectors as elements of the respective kernels:
\begin{align*}
	\vect{v_1} &= (1,0),
    \\
    \vect{v_2}(t) &= (- D_{1,0}(Y_2(t),t), D_{0,0}(Y_2(t), t)),
\end{align*}
Hence, the left-hand side of Equation~\eqref{eq:motzkin_n_meanders_second_kernel}
vanishes when evaluated at $y = Y_1(t)$ \emph{and} left-multiplied by $\vect{v_1}$,
and also when evaluated at $y = Y_2(t)$ \emph{and} left-multiplied by $\vect{v_2}(t)$.
Combining the corresponding two new equations we get
\[
	\vect{h}(t) = t F(t) \vect{M^+}(0,0; t),
\]
with a column vector $\vect{h}(t)$ and a $2 \times 2$-matrix $F(t)$.
The matrix $F(t)$ is invertible, so the generating function of Motzkin N-meanders with maximum reachable point $0$, \ie N-excursions ending in $\{0\}$, is equal to
\[
	\vect{M^+}(0,0; t) = \frac{1}{t} F(t)^{-1} \vect{h}(t).
\]
Then we substitute this expression into Equation~\eqref{eq:motzkin_n_meanders_second_kernel}
to express the generating function of Motzkin N-meanders with minimum reachable point $0$, \ie N-excursions, as
\[
	\vect{M^+}(0,y; t) = \frac{1}{t} D(y,t)^{-1} (\vect{f}(y,t) -  E(y,t) \vect{M^+}(0,0; t)).
\]
Finally, we substitute this expression into Equation~\eqref{eq:motzkin_n_meanders}
to express the generating function of Motzkin N-meanders as
\begin{align*}
	\vect{M^+}(x,y; t) &=
	\left(\id - t A(x,y)\right)^{-1}  (\vect{e_1} - t \left(A(x,y) - B(x,y)\right) \vect{M^+}(0,y; t) 
  \\ & \quad - t \left(B(x,y) - C(x,y)\right) \vect{M^+}(0,0; t)).
\end{align*}
The generating function of N-meanders and N-excursions are then, respectively,
$M_{\romI}^+(x,y; t) + M_{\romII}^+(x,y; t)$ and $M_{\romI}^+(0,y; t) + M_{\romII}^+(0,y; t)$.

Let $\KK := \QQ(p_{-1}, \dots, p_{-1,0,1},x,y,t)$ be the generic base field.
Then, the generating functions live in the algebraic extension $\KK(X_1(y), X_2(y), Y_1, Y_2)$ which is in general of degree $16$ over $\KK$. 
Note that none of the minimal polynomials factors is in the field generated by the other three. 
From the proof above, we deduce the following algebraic structure:
\begin{align*}
M_1(0,0) &\in \KK(Y_2), \\
M_2(0,0) &\in \KK(Y_1,Y_2), \\
M_1(0,y), M_1(x,y) &\in \KK(X_1(y),Y_1,Y_2), \\
M_2(0,y), M_2(x,y) &\in \KK(X_1(y),X_2(y),Y_1,Y_2).
\end{align*}
Note that for $y=1$ these relations stay valid and no simplifications happen. Thus, generically, the degrees are either $2$, $4$, $8$, or $16$. In particular, for generic weights the degrees of $M^+(x,y,t)$, $M^+(0,y,t)$ and $M^+(0,1,t)$ are of degree $16$.
\end{proof}

The previous proof shows that the generic degree of the generating functions of N-meanders is~$16$. 
However, for specific choices of weights the degree reduces. 
In particular, in the accompanying Maple worksheet we show that for N-excursions $M^+(0,1; t)$ over $\QQ(t)$ degrees $4$, $8$, or $16$ are possible.
Using resultants, we can derive the following general results.

\begin{corollary}
    The special choices of weights described below result in the following simplifications.
    \begin{itemize}
        \item If $p_{-1,1}=0$ or $p_{-1}=p_{1}$ we have $X_1(1) = \left(1 + \frac{p_{-1,1}}{p_{-1}}\right) Y_1$.
        \item If $p_{-1,0}=p_{0,1}$ and either $p_{-1}=p_{1}$ or $p_{-1,0} + p_{-1,1} + p_{-1,0,1}=0$ we have $X_2(1) = \left(1 + \frac{p_{-1,0} + p_{-1,1} + p_{-1,0,1}}{p_{-1}}\right) Y_2$.
    \end{itemize}
\end{corollary}

If all weights are equal to $1$ then generating functions of N-excursions $M^+(0,1;t)$ and N-meanders $M^+(1,1,t)$ are at most algebraic of degree $4$.
In particular, the generating function of N-meanders $M^+(1,1;t)$ is even algebraic of degree $2$ and given by
\begin{align*}
	M^+(1,1;t) &= \frac{10t-1+\sqrt{(1+2t)(1-6t)}}{8t(1-7t)}.
\end{align*}
Therefore, by singularity analysis (see Section~\ref{sec:introAC}), the total number of N-meanders behaves for $n \to \infty$ like
\begin{align*}
	\frac{3}{4}7^n + \frac{3\sqrt{3}}{2} \frac{6^n}{\sqrt{\pi n^3}} + \bigO\left(\frac{6^n}{n^{5/2}}\right).
\end{align*}
The generating function $E = M^+(0,1;t)$ of N-excursions is indeed algebraic of degree $4$:
{\small
\begin{align*}
&256\,{t}^{5} \left( 2\,t+1 \right) ^{2} \left( 7\,t-1 \right) ^{2}
 \left( 4\,t-1 \right) ^{2}{E}^{4}\\
 &+16\,{t}^{2} \left( 2\,t+1 \right) 
 \left( 7\,t-1 \right)  \left( 4\,t-1 \right)  
 \left( 564\,{t}^{5}-192
\,{t}^{4}-85\,{t}^{3}+69\,{t}^{2}-15\,t+1 \right) {E}^{3} \\
&+\left( 
197264\,{t}^{9}-44448\,{t}^{8}-68144\,{t}^{7}+36720\,{t}^{6}-864\,{t}^
{5}-4624\,{t}^{4}+1742\,{t}^{3}-303\,{t}^{2}+27\,t-1 \right) {E}^{2}\\
 &+\left( 36000\,{t}^{8}-12220\,{t}^{7}-14436\,{t}^{6}+8901\,{t}^{5}+121
\,{t}^{4}-1256\,{t}^{3}+375\,{t}^{2}-45\,t+2 \right) E  \\
&+4288\,{t}^{7}-
456\,{t}^{6}-1901\,{t}^{5}+591\,{t}^{4}+170\,{t}^{3}-108\,{t}^{2}+18\,
t-1=0.
\end{align*}
}%
It is possible to solve this equation explicitly to obtain a (complicated) closed form closed form solution for $E(t)$; see the accompanying Maple worksheet. 
As in the previous cases, the dominant singularity arises from a simple pole while lower order terms are of a smaller exponential growth.
In particular, the total number of Motzkin N-excursions behaves for $n \to \infty$ like
\begin{align*}
	\frac{9}{16}7^n - \gamma \frac{6^n}{\sqrt{\pi n^3}} + \bigO\left(\frac{6^n}{n^{5/2}}\right), 
\end{align*}
where $\gamma \approx 0.6183$ is the positive real solution of $1024\gamma^4-8019\gamma^2+2916=0$. 
This means that for large $n$ asymptotically $75\%$ of all N-walks are N-meanders and $75\%$ of these N-meanders are N-excursions.

In Appendix~\ref{sec:SubclassesNMotzkin} we considered all possible combination of weights $p_{\vs} \in \{0,1\}$ for $\vs \in S_M$ for N-excursions and N-meanders and found several connections with known sequences in the OEIS, yet most of them are not. 
For example, when all $p_i$'s are equal to $1$ neither the sequences of N-meanders nor N-excursions are in the OEIS.

\newcommand{\NtwoDmap}{\varphi}
\subsection{Bijections with two-dimensional lattice paths}

Any N-step set can be interpreted as a two-dimensional lattice path problem. 
The minimum reachable point is mapped to the $x$-coordinate and the maximum reachable point to the $y$-coordinate: 
We define the mapping $\NtwoDmap$ from N-steps to two-dimensional steps as
\begin{align*}
    \NtwoDmap(\vs) = \left( \min(\vs), \max(\vs) \right).
\end{align*}
This generalizes the previously discussed mapping~\eqref{eq:DyckNWalkBijection}.
For Motzkin N-steps $S_M$ this gives the following mapping to nearest neighbor steps:
\begin{align*}
    \{-1\} &\mapsto (-1,-1), &    
    \{0,1\} &\mapsto (0,1), &
    \{-1,1\} &\mapsto (-1,1)_1, \\
    \{0\} &\mapsto (0,0), &
    \{-1,0\} &\mapsto (-1,0), &
    \{-1,0,1\} &\mapsto (-1,1)_2, \\
    \{1\} &\mapsto (1,1). &&&
\end{align*}
Note that the step $(-1,1)$ comes in two colors, depending on whether it corresponds to $\{-1,1\}$ or $\{-1,0,1\}$.
We denote the two colors by $(-1,1)_1$ and $(-1,1)_2$.
Moreover, by construction the steps $(1,0)$, $(1,-1)$, and $(0,-1)$ will never be used. 
This played a crucial role in proving the algebraicity of N-bridges in Theorems~\ref{theo:DyckNBridges} and \ref{theo:MotzkinNBridges} as it led to Equation~\eqref{eq:NDyckBridgesInterpretation}.
In the following section, we will show that this fact holds for arbitrary finite N-step sets, and also allows to prove the algebraicity of general N-bridges.

Now we use $\NtwoDmap$ to give bijections between different families of Motzkin N-walks and two-dimensional walks.
In general the steps above are used, just on the axes we might need to modify the steps; compare with Figure~\ref{fig:DyckReflectionAbsorption}.
\begin{enumerate}
    \item N-walks are in bijection with unconstrained two-dimensional walks.
    \item N-bridges are in bijection with two-dimensional walks ending in the second quadrant, which, if only steps $\{(-1,-1),(1,1),(-1,1)_1\}$ are used it, have to be of even length.
    \item N-meanders are in bijection with two-dimensional walks confined to the first quadrant, where the positive $y$-axis and the origin use different steps that depend on the possible types.
    \item N-excursions are in bijection with two-dimensional walks associated with N-meanders that end on the $y$-axis.
\end{enumerate}

		\section{N-bridges with general N-steps}
        \label{sec:general_NBridges}

We have already seen in~\eqref{eq:NWalksGeneric} that the generating function of N-walks is always rational. 
In this section we treat the generating function of N-bridges for arbitrary finite N-step set. 
The following main result generalizes Theorems~\ref{theo:DyckNBridges} and \ref{theo:MotzkinNBridges}.

\begin{theorem} \label{th:general_bridges}
For any finite N-step set $S$, the generating function $B(x,y;t)$ of N-bridges with respect to length, minimum, and maximum reachable point is algebraic.
\end{theorem}

A method for computing this generating function is provided by the proof
in Subsection~\ref{sec:proof_general_bridges}.
In order to present this result, we first establish the required algebraic setting.

    \subsection{Sumset monoid}
    \label{sec:sumset:monoid}

This subsection contains the algebraic definitions
needed to describe sets of reachable points
of N-walks on a given N-step set.

\paragraph{Finite integer sets.}
We consider the set $\integersets$ containing
all finite sets of integers
\[
    \integersets = \{\vs \mid \vs \subset \integers,\ |\vs| < \infty\}.
\]
They will represent N-steps as well as
sets of reachable points. Thus, $\emptyset \in \mF$ corresponds to N-walks with an empty set of reachable points (this can be the case for N-meanders and N-excursions).

\paragraph{Sum.}
The endpoint of a classical walk is the sum of its steps. Extending this construction to N-walks requires a suitable notion of sum for N-steps. The following definition is central in additive combinatorics \cite{tao2006additive}.

\begin{definition}[Sumset]
The \emph{sumset} or \emph{Minkowski sum} of two integer sets $\vs$ and $\vt$ is defined as
\[
    \vs + \vt
    =
    \{a + b \mid a \in \vs,\ b \in \vt\}.
\]
The family $\mF$, equipped with the sumset operator $+$, forms a commutative monoid with neutral element $\{0\}$.
\end{definition}

Note that for any $\vs \in \mF$, we have $\vs + \emptyset = \emptyset$. Consequently, once an N-walk has no reachable point, adding further N-steps cannot create new reachable points. Furthermore, shifting a set $\vs$ by $c \in \integers$ corresponds to a sumset with the set $\{c\}$:
\[
    \vs + \{c\}
    =
    \{a + c \mid a \in \vs\}.
\]

\begin{lemma}[Reachable points as sumsets]
The set of reachable points of an N-walk $w = (\vw_1, \ldots, \vw_n)$ is equal to the sumset of its N-steps $\vw_1 + \cdots + \vw_n$.
\end{lemma}

\begin{proof}
A point $x$ is reachable if and only if there exists a compatible walk
$(v_1,\dots,v_n)$ with $v_i \in \vw_i$ such that $x = \sum_{i=1}^n v_i$,
which is equivalent to $x \in \vw_1 + \cdots + \vw_n$.
\end{proof}

The commutativity of the monoid $(\mF, +)$ implies that the reachable points of an N-walk are independent of the order of the N-steps.

\smallskip

The subtraction operation
\[
    \vs - \vt
    =
    \{a - b \mid a \in \vs,\ b \in \vt\}
\]
is defined in the same way.

\paragraph{Product.}
Given a nonnegative integer $n$ and $\vs \in \integersets$,
we define $n \times \vs$ as the $n$-fold sumset of $\vs$ with itself
\[
    n \times \vs 
    = \vs + \cdots + \vs
    = \{a_1 + \cdots + a_n \mid 
    a_i \in \vs\}.
\]
For example, we have
\[
    n \times \{0, m\} =
    \{k\, m \mid 0 \leq k \leq n\}
\]
and $0 \times \vs = \{0\}$ by convention.

\paragraph{Norm.}
The \emph{norm} of any $\vs \in \integersets$ is defined as
\[
    \|\vs\| =
    \begin{cases}
        0 & \text{if } \vs = \emptyset,\\
        \max(\vs) - \min(\vs) & \text{otherwise.}
    \end{cases}
\]

\paragraph{Bottom pruning.}
Let $\vs, \vt \in \integersets$ be two sets of integers. 
Then \emph{bottom pruning} is defined as
\[
    \vs \usub \vt =
    \begin{cases}
        \emptyset & \text{if } \vs = \emptyset,\\
        \vs \setminus (\{\min(\vs)\} + \vt)
        & \text{otherwise.}
    \end{cases}
\]
It corresponds to shifting $\vt$ to the bottom of $\vs$,
and then removing the corresponding elements from~$\vs$.
In other words, $\vt$ accounts for the elements relative to $\min(\vs)$.
Note that negative elements in~$\vt$ will always lie outside the range of $\vs$ after shifting.
For example, removing the smallest element of~$\vs \neq \emptyset$
corresponds to bottom pruning by $\{0\}$, and 
\[
    \{3, 4, 6, 7, 8\} \usub \{-2,1, 3, 5\} =
    \{3, 4, 6, 7, 8\} \setminus \{1,4, 6, 8\} =
    \{3, 7\}.
\]

\paragraph{Top pruning.}
Let $\vs, \vt \in \integersets$ be two sets of integers. 
Then \emph{top pruning} is defined as
\[
    \vs \msub \vt =
    \begin{cases}
        \emptyset & \text{if } \vs = \emptyset,\\
        \vs \setminus (\{\max(\vs)\} - \vt)
        & \text{otherwise.}
    \end{cases}
\]
Top pruning corresponds to removing from the top,
where $\vt$ accounts for the elements relative to $\max(\vs)$
from top to bottom.
Again, negative elements in $\vt$ will lie outside the range after shifting.
For example, removing the largest element of $\vs \neq \emptyset$
corresponds to a top pruning by $\{0\}$, and
\[
    \{0, 1, 2, 3, 4\} \msub \{1, 2\} = \{0, 1, 4\}.
\]
Compare the previous example to the bottom pruning
\[
    \{0, 1, 2, 3, 4\} \usub \{1, 2\} = \{0, 3, 4\}.
\]

\paragraph{Equivalence.}
Two integer sets $\vs, \vt \in \integersets$
are \emph{equivalent}, denoted by $\vs \sim \vt$,
if there exists an integer $m$ such that
$\vs = \vt + \{m\}$.
In other words, one can be obtained from the other by a shift.

\paragraph{Conjugate.}
The \emph{conjugate} of an element of $\vs \in \integersets$ is defined as
\[
    \conjugate{\vs}
    =
    \begin{cases}
    \emptyset & \text{if } \vs = \emptyset,\\
    \{\min(\vs) + \max(\vs) - a \mid a \in \vs\} & \text{otherwise.}
    \end{cases}
\]
For example, the conjugate of $\{1, 3, 4\}$ is $\{1, 2, 4\}$.
Furthermore, bottom and top pruning are linked by conjugation
\[
    \conjugate{\vs} \usub \vt
    \sim
    \conjugate{\vs \msub \vt}.
\]

\paragraph{Type.}
For two nonnegative integers $n$, $k$
and $\va, \vb, \vc \in \integersets$,
we define $\type n k \va \vb \vc$ as the following subset of~$\integersets$:
\[
    \type n k \va \vb \vc =
    \left\{
        j \times \{0, n\} + \vb \usub \va \msub \vc + \{m\}
        \mid
        m \in \integers,\ j \in \integers_{\geq k}
    \right\}.
\]
For example, 
\[
    \type{3}{1}{\emptyset}{\{0,1\}}{\{1,2,3\}}  =
        \left\{ \{0,4\}, \{0,1,3,7\}, \{0,1,3,4,6,10\}, \{0,1,3,4,6,7,9,13\}, \dots \right\}.
\]

Provided $n \geq 1$ and $\vb \neq \emptyset$,
$\type n k \va \vb \vc$ contains elements of arbitrarily large norm.
However, the key property is that each $\vt \in \type n k \va \vb \vc$ is uniquely characterized
by  $\min(\vt)$ and $\max(\vt)$.
We will prove in Lemma~\ref{th:type:of:sumsets}
that there are finitely many
possible types for the reachable points
of an N-walk on a given N-step set.
In our generating functions, we will keep track
of the minimum and maximum reachable points,
as well as the type of the reachable points.
This information will be enough
to reconstruct the complete set of reachable points.
Below are a few examples of relations between types:
\begin{align*}
    \type 0 k \va \vb \vc &=
    \{\vb \usub \va \msub \vc + \{m\} \mid m \in \integers\},
    \\
    \type n k \va \emptyset \vc &=
    \{\emptyset\},
    \\
    \type{1}{0}{\emptyset}{\{0\}}{\emptyset} &=
    \{[u, v] \mid v \geq u\},
    \\
    \type n k \va {\vb + \{m\}} \vc &=
    \type n k \va \vb \vc,
    \\
    \type n k \va \vb \vc &=
    \{\vs \usub \va \msub \vc \mid \vs \in \type n k \emptyset \vb \emptyset\}.
\end{align*}

\paragraph{Proper type.}
A type $(n,k,\va,\vb,\vc)$ is \emph{proper}
if either
\begin{itemize}
\item $n=0$, then $k=0$, $\va = \vc = \emptyset$ and
either $b = \emptyset$ or
$\min(\vb) = 0$, or
\item $n > 0$, then $\min(\va) \geq 0$, $\min(\vc) \geq 0$, $\min(\vb) = 0$, $\vb \subset [0,n-1]$,
and $k n > \max(\va) + \max(\vc)$.
\end{itemize}
The condition on~$\vb$ guarantees that in the sumset of $j \times \{0, n\}$ and~$\vb$ there are no overlaps.
Therefore, all $\vs$ in a proper type are, except for the bottom and the top, constructed using a periodic repetition of a pattern defined by~$\vb$ of size $n$.

Since for any $m \in \integers$ and type $(n,k,\va,\vb,\vc)$,
we have
\[
    \type n k \va \vb \vc
    =
    \type n k
    {\va \cap \integers_{\geq 0}}
    {\vb + \{m\}}
    {\vc \cap \integers_{\geq 0}},
\]
and
\[
    \type 0 k \va \vb \vc
    =
    \type 0 0 \emptyset {\vb \usub \va \msub \vc} \emptyset,
\]
any type with $n=0$ 
is equal to a proper type,
and we will prove in Lemma~\ref{th:to:proper}
that for any type $(n, k, \va, \vb, \vc)$ with $n > 0$ and $\vb \neq \emptyset$,
there exist $k'$ large enough and parameters $\va', \vb', \vc'$ such that it can be represented by the proper type $(n, k', \va', \vb', \vc')$.

        \subsection{Automaton for N-bridges}
        \label{sec:proof_general_bridges}

For N-walks and in particular N-bridges, we will show that finitely many types suffice to capture all possible sets of reachable points (of which in general, however, infinitely many will exist).

\begin{proposition} \label{th:type:of:sumsets}
For any finite subset $S \subset \integersets$,
there is a finite set of types $(n_i, k_i, \va_i, \vb_i, \vc_i)_{1 \leq i \leq k}$
such that any element generated by $(S, +)$
(\ie obtained by any finite number of sumsets of elements of $S$)
belongs to $\type{n_i}{k_i}{\va_i}{\vb_i}{\vc_i}$ for some $1 \leq i \leq k$.
\end{proposition}

The last proposition is almost enough to prove
that general N-bridges have algebraic generating functions.
In fact, we rely on the following, more effective, generalization.

\begin{proposition} \label{th:automatic:type:of:sumsets}
For any finite subset $S \subset \integersets$ of N-steps,
there is a finite automaton $\mA$ on the alphabet $S$
where each state corresponds to a proper type,
such that for any N-walk $w = (\vw_1, \ldots, \vw_m) \in S^m$,
the sumset $\vw_1 + \cdots + \vw_m$
belongs to the type corresponding to the state that
$\mA$ reaches after reading $w$.
Furthermore, if loops are removed,
the directed graph underlying this automaton is acyclic.
\end{proposition}

Before proving this proposition,
we show how it implies the proof of the main result of this section, Theorem~\ref{th:general_bridges}, 
showing that for any finite N-step set $S$,
the generating function of N-bridges is algebraic.

{
\begin{proof}[Proof of Theorem~\ref{th:general_bridges}]
    We start with the automaton from Proposition~\ref{th:automatic:type:of:sumsets}.
    By construction, each N-walk, read step-by-step, corresponds to a walk on this automaton such that the current state corresponds to the current type of the reachable points.
    Now, we associate to each state $q$ a generating function 
    \[
        F_q(x,y;t) = \sum_{n,i,j \geq 0} f^{(q)}_{i,j,n} x^i y^j t^n,
    \]
    where $f^{(q)}_{i,j,n}$ is the number of N-walks of length $n$ whose reachable points have type associated to~$q$, minimum $i$, and maximum $j$.
    As in the previous proofs of Theorems~\ref{theo:DyckNBridges} and \ref{theo:MotzkinNBridges} for the Dyck and Motzkin cases, it is now easy to build a characterizing system of equations for $F_q(x,y;t)$.
    This system is finite and therefore each generating function is rational. 
    
    It remains to extract N-bridges, \ie N-walks $w$ such that $0 \in \reach(w)$.
    We can do this now for each type individually. 
    For all types, it is necessary that $i \leq 0$ and $j \geq 0$.
    In terms of generating functions, this translates to
    \newcommand{\BgenA}{\tilde{B}}
    \begin{align*}
        \BgenA_q(x,y;t) := F_q(x,y;t) - [x^{>0}] F_q(x,y;t) - [y^{<0}] F_q(x,y;t).
    \end{align*}
    Indeed, the key observation is as before that $[x^{>0}y^{<0}] F_q(x,y;t)=0$, since the maximum can never by smaller than the minimum.
    Recall that the positive (and negative) part of a rational generating function is algebraic~\cite[Section 6]{Gessel80}. 
    Moreover, for $0 \in \reach(w)$ and a given type $\type g k \va \vb \vc$,
    we additionally require $i \in -\vb \mod g$, $i \notin -\va$, and $j \notin \vc$.
    First, we extract the arithmetic progressions of the minimum associated with $\vb$. This is a standard method, which we recall for completeness: 
    Let $\omega$ be a $g$th root of unity. Then,
    \newcommand{\BgenB}{\hat{B}}
    \begin{align*}
        \BgenB_q(x,y;t) := \sum_{b \in \vb} \frac{1}{g} \sum_{i=0}^{g-1} \omega^{i-b} \BgenA(\omega^{i} x,y;t).
    \end{align*}    
    Second, it remains to exclude the finite amount of values determined by $\va$ and $\vb$ to get the generating function $B_q(x,y;t)$ of bridges associated with state $q$ as
    \begin{align*}
        B_q(x,y;t) = \BgenB(x,y;t) - \sum_{a \in \va} [x^{-a}] \BgenB(x,y;t) - \sum_{c \in \vc} [y^c] \BgenB(x,y;t).
    \end{align*}
    Note that here it is important that the types are proper to avoid double counting.
    Finally, the generating function of bridges is the sum of the generating functions $B_q(x,y;t)$ over all states $q$.
    Algebraic functions are closed under a range of operations, including finite sums, derivatives, and substitutions of constants, among others; see~\cite[Section 6]{S01} and Section~\ref{sec:introalgebraic}.
\end{proof}
}

The rest of this section is dedicated
to the proof of Proposition~\ref{th:automatic:type:of:sumsets}.
Let $\sum_{\geq m} S$ denote the family
of sumsets containing at least $m$ occurrences
of each N-step from $S$
\[
    \sum_{\geq m} S
    =
    \Big\{
    \sum_{\vs \in S} k_{\vs} \times \vs
    \mid
    \ k_{\vs} \in \integers_{\geq m}
    \Big\}.
\]
The following lemma is the first key result to prove that finitely many types are sufficient. 
It shows that the reachable points of any N-walk after sufficiently many N-steps are, except for the bottom and the top, equivalent to a periodic integer interval.

\begin{lemma} \label{th:geqmS}
For any N-step set $S$,
there is an integer $m$
and a type $(g,0,\va,\{0\},\vc)$ such that
\[
    \sum_{\geq m} S \subseteq
    \type g 0 \va {\{0\}} \vc.
\]
\end{lemma}

\begin{proof}
If $S$ is empty, we fix $m=0$ and have
\[
    \sum_{\geq m} S = \{\{0\}\},
\]
as the empty sumset is equal
to the neutral element $\{0\}$.
In that case, we choose
$g=0$ and $\va = \vc = \emptyset$.
If $S$ contains the N-step $\emptyset$, then $\sum_{\geq m} S = \{ \emptyset \}$ for all $m\geq1$.
Hence, we may choose $m = 1$, and again $g=0$ and $\va = \vc = \emptyset$.

Let us now assume $S$ does not contain $\emptyset$.
Let $S'$ denote the set of N-steps from $S$
of size at least $2$, shifted so that their minimum is $0$
\[
    S' =
    \{ \{x - \min(\vs) \mid x \in \vs\} \mid
    \vs \in S,\ |\vs| \geq 2
    \}.
\]
Any sumset of N-steps from $S$
can be rewritten as a shifted version
of a sumset of N-steps from~$S'$.
Indeed, consider such a sumset where
$(\vs_i)_{1 \leq i \leq n}$ denote the N-steps of size at least $2$,
and $(\{c_i\})_{1 \leq i \leq m}$ denote the N-steps of size $1$.
Then
\begin{align*}
    \vs_1 + \cdots + \vs_n + \{c_1\} + \cdots + \{c_m\} =\ &
    \{x - \min(\vs_1) \mid x \in \vs_1\} + 
    \cdots + \{x - \min(\vs_n) \mid x \in \vs_n\}
    \\&+
    \{c_1 + \cdots + c_m + \min(\vs_1) + \cdots + \min(\vs_n)\}.
\end{align*}
Since by definition types are stable by shift,
it is sufficient to prove the lemma for $S'$ instead of $S$.
Let $g$ denote the greatest common divisor of the union of the N-steps from $S'$.
We define another set $S''$ as the set of N-steps contracted by $g$:
\[
    S'' = \left\{ \left\{ \frac{x}{g} \mid x \in \vs \right\} \mid \vs \in S' \right\}.
\]
Any sumset of N-steps from $S'$ is obtained
from the corresponding sumset of N-steps from $S''$
by dilation by $g$.
Thus, if we prove the lemma for $S''$,
obtaining a type $(1, k, \va, \{0\}, \vc)$,
then the lemma holds for $S'$ with type
$(g, k, \{g x \mid x \in \va\}, \{0\}, \{g x \mid x \in \vc\})$.

Let $\va$ denote the points that cannot be obtained
as nonnegative linear combinations of elements from the N-steps
\[
    \va =
    \integers_{\geq 0}
    \setminus
    \bigg\{
        \sum_{t \, \in \!\! \bigcup\limits_{\vs \in S''} \! \vs} k_t \, t
        \mid k_t \in \integers_{\geq 0}
    \bigg\}.
\]
According to Schur's Theorem \cite{alfonsin2005diophantine}, for any set $T$
of positive integers of greatest common divisor~$1$, there is a smallest number $F_T \in \integers_{\geq 0}$,
called the \emph{Frobenius number}
such that for any $n > F_T$,
there exist nonnegative integers $(k_t)_{t \in T}$ such that
\[
    n = \sum_{t \in T} k_t \, t.
\]
Let $F$ denote the Frobenius number
for $\bigcup_{\vs \in S''} \vs$.
Schur's Theorem implies that $\va$ is a finite set,
with maximum $F$.
Let also $\conjugate{S''}$ denote the family
of conjugate N-steps from $S''$
\[
    \conjugate{S''} =
    \{
    \conjugate{\vs} \mid \vs \in S''
    \},
\]
and $\conjugate{F}$ denote the Frobenius numbers for
$\bigcup_{\vs \in S''} \conjugate{\vs}$.
The integer set $\vc$ is defined like $\va$, but for $\conjugate{S''}$
instead of $S''$.
The maximum of $\vc$ is equal to $\conjugate{F}$.
Moreover, let $M$ and $m_0$ denote the integers
\[
    M = \max_{\vs \in S''} \max_{x \in \vs} x
    \quad \text{ and } \quad
    m_0 = F + \conjugate{F} + M.
\]
The reason for the definition of $M$
is that for any $n \geq M$, any interval $[a,b]$ of length $n$,
and any N-step $\vs \in S''$, it holds that $[a,b] + \vs$ is still an interval now
of length $n + \|\vs\|$.
By definition of $\va$,
the integers from $[0, F + M]$ that belong
to a sumset of N-steps from $S''$ are
\[
    [0, F + M] \usub \va.
\]
Those points $y$ are characterized by the existence of
nonnegative integers $(k_t)_{t \in \bigcup_{\vs \in S''} \vs}$
such that
\[
    y = \sum_{t \in \bigcup_{\vs \in S''} \vs} k_t \, t.
\]
We have
\[
    m_0 \geq y \geq \sum_{t \in \bigcup_{\vs \in S''} \vs} k_t.
\]
Now, note that as all N-steps in $S''$ contain $0$, it holds that once $y \in \vr$ we have $y \in \vr + \vs$ for all $\vs \in S''$.
Thus, $y$ is in particular reached by $m_0 \times \sum_{\vs \in S''} \vs$.
We conclude that the bottom of any element $\vr \in \sum_{\geq m_0} S''$
contains all integers, except the elements from $\va$:
\[
    \vr \cap [0, F + M]
    =
    [0, F + M] \usub \va.
\]
Looking at the conjugates, we obtain
\[
    \vr \cap [\max(\vr) - \conjugate{F} - M, \max(\vr)] =
    \left( [\max(\vr) - \conjugate{F} - M, \max(\vr)] \right) \msub \vc.
\]
Since $\max(\va) < F$ and $\max(\vc) < \conjugate{F}$,
the integer set $\vr = m_0 \times \sum_{\vs \in S''} \vs$
contains the following two (possibly overlapping) intervals
of length $M$: 
$[F+1, F+M]$
and
$[\max(\vr) - \conjugate{F} - M, \max(\vr) - \conjugate{F} - 1]$.
Note that if they overlap, then the claim follows. 
Otherwise, we show now how to extend these intervals until they overlap.

Let $\vt_1, \vt_2 \in \integersets$ be two integer sets.
If $\min(\vt_1) = 0$ and $\vt_2$ contains
an interval $[u, v]$ with $v - u \geq \max(\vt_2)$,
then $\vt_1 + \vt_2$ contains the interval $[u, v + \max(\vt_2)]$.
By our choice of $M$, this holds for the previous two intervals. 
Recursively, we deduce that
for any $\vs_1, \ldots, \vs_n \in S''$,
the sumset $\vr + \vs_1 + \cdots + \vs_n$ contains
the intervals
$[F+1, F + M + \sum_{i=1}^n \max(\vs_i)]$
and
$[\max(\vr) - \conjugate{F} - M, \max(\vr) + \sum_{i=1}^n \max(\vs_i) - \conjugate{F} - 1]$.
When $\sum_{i=1}^n \max(\vs_i)$ is large enough,
those two intervals overlap.
Specifically, consider the sumset
$\vr = m_0 \times \sum_{\vs \in S''} \vs$
and set $m = m_0 + \max(\vr)$.
For any $\vt \in \sum_{\geq m} S''$, we have
$\vt = \vr + \vs$ for some $\vs \in \sum_{\geq \max(\vr)} S''$,
thus satisfying $\max(\vs) \geq \max(\vr)$.
Therefore, $\vt$ contains the full interval
\[
    [F+1, \max(\vt) - \conjugate{F} - 1],
\] 
and $\vt$ has type $(1, 0, \va, {\{0\}}, \vc)$,
concluding the proof.
\end{proof}

Next we show that types are closed under sumsets with finite integer sets. 
For this purpose, we extend the sumset to sets of sets as follows:
For $\mathcal{R} \subseteq \integersets$ and $\vs \in \integersets$ we define
\begin{align*}
    \mathcal{R} + \vs := \{ \vr + \vs \mid \vr \in \mathcal{R} \}.
\end{align*}

\begin{lemma} \label{th:sumset:to:type}
For any integer $g$ and finite integer sets $\va$, $\vb$, $\vc$, $\vs$,
there exist an integer $k$ and a type $(g,k',\va',\vb',\vc')$
such that
\[
    \type g k \va \vb \vc + \vs = 
    \type g {k'} {\va'} {\vb'} {\vc'}.
\]
\end{lemma}

\begin{proof}
If $g=0$, then for $\vr = \vb \usub \va \msub \vc$ we have
\[
    \type g k \va \vb \vc =
    \{ \vr + \{m\} \mid m \in \integers \}.
\]
We then choose $k=0$ and the type
\[
    (g,k',\va',\vb',\vc') =
    (0, 0, \emptyset, \vr + \vs, \emptyset).
\]

Let us now assume $g > 0$.
Both sides of the equality of the lemma
are stable by shift,
so we assume without loss of generality
$\min(\vs) = 0$ and $\min(\vb) = 0$.
We say that an integer set $\vt$ is $g$-periodic
on an interval $[\alpha, \beta]$
if for all $x, y \in [\alpha, \beta]$
with $x - y \equiv 0 \mod g$, we have
$x \in \vt$ if and only if $y \in \vt$.
Let $k$ be an integer satisfying
\[
    k g \geq 2 + \max(\va) + \max(\vb) + \max(\vc) + \max(\vs).
\]
Let $\vr_j$ denote the integer set
\[
    \vr_j =
    j \times \{0, g\} + \vb \usub \va \msub \vc,
\]
so that
\[
    \type g k \va \vb \vc =
    \{ \vr_j + \{m\} \mid j \geq k,\ m \in \integers \}.
\]
The integer set
$j \times \{0, g\}$ is $g$-periodic
on the interval $[0, j g]$.
Note in general that for any $\vt$ and $\vb$ with $\min(\vb)=0$, to decide whether an integer $x \in \vt + \vb$,
one only needs to consider if
\[
    x \in (\vt \cap [x - \max(\vb), x]) + \vb.
\]
Thus, $j \times \{0, g\} + \vb$
is $g$-periodic on the interval $[\max(\vb), j g]$.
Observe that on the interval
$[\max(\va) + \max(\vb) + 1, j g - \max(\vc) - 1]$,
the integer sets $r_j$ and $j \times \{0, g\} + \vb$ are equal,
so $r_j$ is $g$-periodic on this interval.
Applying the same reasoning,
$\vr_j + \vs$
is $g$-periodic on the interval $I_j$ defined as
\[
    I_j =
    [\max(\va) + \max(\vb) + \max(\vs) + 1, j g - \max(\vc) - 1].
\]
The integer $k$ has been chosen to ensure
this interval is non-empty whenever $j \geq k$.
On this interval, $\vr_j + \vs$ is equal to
\[
    I_j \cap (\vr_j + \vs)
    =
    I_j \cap (j \times \{0, g\} + \vb + \vs)
\]
because the bottom pruning by $\va$
and the top pruning by $\vc$ do not overlap.
For any integer sets $\vt \subseteq \vr$, we have
\[
    \vt + \vs \subseteq \vr + \vs.
\]
This implies
\[
    \vr_j + \vs
    \subseteq
    j \times \{0, g\} + \vb + \vs.
\]
Let $\Ijbottom$ and $\Ijtop$
denote the intervals
\begin{align*}
    \Ijbottom &=
    [0, \max(\va) + \max(\vb) + \max(\vs)],
    \\
    \Ijtop &=
    [j g - \max(\vc), j g + \max(\vb) + \max(\vs)],
\end{align*}
so that $\Ijbottom$, $I_j$, and $\Ijtop$ form a partition
of the interval $[0, \max(\vr_j + \vs)]$.
There exists
$\va' \subseteq [0, \max(\va) + \max(\vb) + \max(\vs)]$,
independent of $j$, such that
\[
    \Ijbottom
    \cap
    (\vr_j + \vs)
    =
    \Ijbottom
    \cap
    (j \times \{0, g\} + \vb + \vs \usub \va').
\]
Similarly, working on the conjugates,
there exists $\vc' \subseteq [0, \max(\vc) + \max(\vb) + \max(\vs)]$
independent of $j$ such that
\[
    \Ijtop
    \cap
    (\vr_j + \vs)
    =
    \Ijtop
    \cap
    (j \times \{0, g\} + \vb + \vs \msub \vc').
\]
Collecting the results about $\Ijbottom$, $I_j$, and $\Ijtop$,
we obtain
\[
    \vr_j + \vs =
    j \times \{0, g\} + \vb + \vs \usub \va' \msub \vc'.
\]
Thus
\[
    \{\vr_j + \vs \mid j \geq k\}
    =
    \type g k {\va'} {\vb + \vs} {\vc'}. \qedhere
\]
\end{proof}

We need one last technical lemma, stating that for $k$ large enough, each type can be chosen to be proper.

\begin{lemma} \label{th:to:proper}
For any $g$, $\va$, $\vb$, $\vc$,
there exist an integer $k$
and a proper type $(g, k', \va', \vb', \vc')$ such that
\[
    \type g k \va \vb \vc =
    \type g {k'} {\va'} {\vb'} {\vc'}.
\]
For any $j \geq k$, there exists $j'$ such that
$\type g j \va \vb \vc = \type g {j'} {\va'} {\vb'} {\vc'}$.
\end{lemma}

\begin{proof}
Since for any $m \in \integers$,
we have
\[
    \type g k \va \vb \vc =
    \type g k
    {\va \cap \integers_{\geq 0}}
    {\vb + \{m\}}
    {\vc \cap \integers_{\geq 0}},
\]
we can assume without loss of generality $\min(\vb) = 0$,
$\min(\va) \geq 0$, and $\min(\vc) \geq 0$.
If $g=0$ or $\max(\vb) < g$, then the type is already proper.
Otherwise, let $g \integers$ denote the set $\{g m \mid m \in \integers\}$.
Let
\[
    \vb' =
    (g \integers + \vb)
    \cap
    [0, g-1].
\]
By periodicity, we have
\[
    g \integers + \vb =
    g \integers + \vb'.
\]
Thus, for all $j \geq 0$,
\[
    j \times \{0,g\} + \vb
    \subseteq
    g \integers + \vb
    =
    g \integers + \vb',
\]
and, more precisely,
\[
    j \times \{0,g\} + \vb
    \subseteq
    (j + \lceil \max(\vb) / g \rceil) \times \{0,g\} + \vb'.
\]
We fix $k$ large enough
to ensure both $k > \lceil \max(\vb) / g \rceil$
and $\max(\va) + \max(\vc) < k g$.
The first constraint ensures
that for any $j \geq k$,
we have
\[
    (j \times \{0,g\} + \vb)
    \cap
    [0, \max(\vb)]
    =
    (k \times \{0,g\} + \vb)
    \cap
    [0, \max(\vb)],
\]
because all the points from $\{j g\} + \vb$
are greater than $\max(\vb)$.
The second constraint ensures that
the bottom and top pruning do not overlap.
We also define
\[
    \va'' =
    \left(
    (k \times \{0,g\} + \vb')
    \setminus
    (k \times \{0,g\} + \vb)
    \right)
    \cap
    [0, \max(\vb)]
\]
and the same on the conjugate for $\vc''$.
This ensures for all $j \geq k$
\[
    j \times \{0,g\} + \vb
    =
    (j + \max(\vb) / g) \times \{0,g\} + \vb' \usub \va'' \msub \vc''.
\]
Since the bottom and top pruning do not overlap,
for all $j \geq k$
\[
    j \times \{0,g\} + \vb \usub \va \msub \vc
    =
    (j + \max(\vb) / g) \times \{0,g\} + \vb' \usub \va'' \usub \va \msub \vc'' \msub \vc,
\]
so, combining the two bottom prunings and the two top prunings,
there exist $\va'$ and $\vc'$
such that 
the claim holds.
\end{proof}

After these technical lemmas, we are now ready to prove our main result.

\begin{proof}[Proof of Proposition~\ref{th:automatic:type:of:sumsets}]
Consider a finite family $S = \{\vs_1, \ldots, \vs_n\}$ of N-steps,
and let $S'$ denote the subset containing
all N-steps of size at least $2$.
We first build an automaton $\mA'$
like in the proposition
except that the types associated to the states need not be proper,
and that $S$ is replaced by $S'$.
Then we will build from it an automaton $\mA$ with proper types,
and finally add the N-steps of size $0$ and $1$.

We will first build the states of the automaton $\mA'$.
To each state $q$, we associate a subset $T_q$ of $S'$,
a type $(g_q, k_q, \va_q, \vb_q, \vc_q)$
and integers $(u_{q,\vt})_{\vt \in S' \setminus T_q}$
indexed by $S' \setminus T_q$.
Informally, the idea is that any sumset $\vs$ of N-steps from $S'$
where each N-step from $T_q$ appears a
``\emph{large enough number of times}'',
and each N-step $\vt$ from $S' \setminus T_q$
appears exactly $u_{q,\vt}$ times,
will have type $(g_q, k_q, \va_q, \vb_q, \vc_q)$.
This \emph{large enough number of times}
means here greater than some integer $m(T_q)$
that depends only on $T_q$.

To each subset $T \subseteq S'$,
we associate an integer $m(T)$
and we build several states of the automaton.
The transitions of the automaton will be added later.
We consider all subsets of $S'$ in an order
such that if $T \subseteq T'$,
the subset $T'$ is considered before $T$.
The integer $m(T)$ is defined
as the smallest integer satisfying
the following three lower bounds.
\begin{itemize}
\item[\rone.]
Let $m$, $g$, $\va$, $\vc$ denote the integers and integer sets
defined in Lemma~\ref{th:geqmS} for $S$ replaced by~$T$, so that
\[
    \sum_{\geq m} T
    \subseteq
    \type g 0 \va {\{0\}} \vc.
\]
We impose $m(T) \geq m$.
\item[\rtwo.]
For any $\vt \in S' \setminus T$,
we impose $m(T) \geq m(T \cup \{\vt\})$.
Because of the order of visit of the subsets,
$m(T \cup \{\vt\})$ is well defined
at the time we consider $T$.
\item[\rthree.]
For each $\vt \in S' \setminus T$ and any integer $u_{\vt} \in [0, m(T \cup \{\vt\}) - 1]$,
we consider the integer set
\[
    \vs =
    \sum_{\vt \in S' \setminus T}
    u_{\vt} \times \vt.
\]
Let $k$, $k'$, $\va'$, $\vb'$, $\vc'$
be defined as in Lemma~\ref{th:sumset:to:type}
so that
\[
     \type g k \va {\{0\}} \vc + \vs
    =
    \type g {k'} {\va'} {\vb'} {\vc'}.
\]
We impose
\[
    m(T) \geq
    \frac{k g}{\min_{\vt \in T}(\|\vt\|)}.
\]
A new state $q$ of the automaton $\mA'$ is created.
We associate to it
the subset $T_q = T$,
integers $(u_{q, \vt})_{\vt \in S'\setminus T_q} =
(u_{\vt})_{\vt \in S' \setminus T_q}$
indexed by $S' \setminus T_q$,
and the type
$(g_q, k_q, \va_q, \vb_q, \vc_q) =
(g, k', \va', \vb', \vc')$.
\end{itemize}
The initial state $q_0$ of the automaton $\mA'$
is the one corresponding to
$T_{q_0} = \emptyset$
and $u_{q_0,\vs} = 0$ for all $\vs \in S'$.

Let us now add transitions to the automaton,
ensuring it is complete.
For each state $q$ and each N-step $\vs \in S'$,
we add the following transition from $q$,
labeled by $\vs$.
\begin{itemize}
\item
If $\vs \in T_q$, the transition is a loop and points to $q$.
\item
If $\vs \in S' \setminus T_q$
and $u_{q,\vs} < m(T \cup \{\vs\}) - 1$,
then when $q$ was created, 
another state $q'$ was created corresponding to
$T_{q'} = T_q$
and $(u_{q', \vt})_{\vt \in S' \setminus T_q} = (u_{q, \vt} + \indic_{\vt = \vs})_{\vt \in S' \setminus T_q}$.
Then, the transition points to $q'$.
\item
If $\vs \in S' \setminus T_q$
and $u_{q,\vs} = m(T \cup \{\vs\}) - 1$,
we consider a set $T'$, maximal for the inclusion,
such that $T \subseteq T'$ and
for all $\vt \in T'$,
we have $u_{q,\vt} + \indic_{\vt = \vs} \geq m(T')$.
The maximality of $T'$ implies
$u_{q,\vt} < m(T')$ for all $\vt \in S' \setminus T'$.
Furthermore, step \rtwo ensures $m(T) \geq m(T')$,
so when considering $T'$, at step \rthree,
a state $q'$ was created satisfying
$T_q = T'$ and $u_{q',\vt} = u_{q,\vt}$
for all $\vt \in S' \setminus T'$.
In the automaton $\mA'$,
the transition from $q$ labeled by $\vs$ points to $q'$.
\end{itemize}
Except for loops, this automaton has no cycles.
Indeed, the non-loop transitions
induce the following order on the states.
Write $q \leq q'$ if
$T_q$ is a strict subset of $T_{q'}$,
or if $T_q = T_{q'}$ and
for all $\vs \in S' \setminus T_q$, we have
$u_{q,\vs} \leq u_{q',\vs}$.

We associate to each state $q$ of the automaton $\mA'$
the subset $U_q$ of $\integers_{\geq 0}^{|S'|}$
\[
    U_q =
    \{(u_{\vs})_{\vs \in S'} \mid
    u_{\vs} \geq m(T_q)
    \text{ if } \vs \in T_q,
    \ 
    u_{\vs} = u_{q, \vs}
    \text{ if } \vs \in S' \setminus T_q
    \}.
\]
By construction, for any transition labeled $\vs$
from a state $q$ to a state $q'$, we have
\begin{itemize}
\item
$
    \{ u_{\vt} + \indic_{\vt = \vs}
    \mid
    (u_{\vt})_{\vt \in S'} \in U_q \}
    \subseteq U_{q'},
$
\item
for any N-walk $(w_1, \ldots, w_n)$ of N-steps from $S'$,
containing $u_{\vs}$ occurrences of the N-step $\vs$
for all $\vs \in S'$,
the state $q$ reached in the automaton $\mA'$
after reading $(w_1, \ldots, w_n)$ satisfies
$(u_{\vs})_{\vs \in S'} \in U_q$.
\end{itemize}

Let $f$ denote the surjective function
from $\integers_{\geq 0}^{|S'|}$ to $(S',+)$
defined as
\[
    f((u_{\vs})_{\vs \in S'}) =
    \sum_{\vs \in S'} u_{\vs} \times \vs.
\]
We now prove that for any state $q$,
we have $f(U_q) \subset \type {g_q} {k_q} {\va_q} {\vb_q} {\vc_q}$.
By definition of $U_q$,
any integer set in $f(U_q)$ can be written as $\vr + \vs$
for some
$
    \vr \in \sum_{\geq m(T_q)}
$
and
$
    \vs = \sum_{\vt \in S' \setminus T_q} u_{q,\vt} \times \vt.
$
Step \rone ensures
$\vr \in \type {g_q} 0 \va {\{0\}} \vc$
for some integer sets $\va$, $\vc$.
Step \rthree ensures additionally
$\|\vr\| \geq k_q g_q$,
so in fact
$\vr$ belongs to $\type {g_q} {k_q} \va {\{0\}} \vc$.
By step \rthree again, this implies
$\vr + \vs \in
\type {g_q} {k_q} {\va_q} {\vb_q} {\vc_q}$.

Let us now gather the results of this proof.
Consider an N-walk $w = (\vs_1, \ldots, \vs_m)$
on the N-step family $S'$
and the state $q$ reached in the automaton $\mA'$
after reading $w$.
Let $u_{\vs}$ denotes the number of occurrences
of the N-step $\vs$ in $w$.
Then we have
$(u_{\vs})_{\vs \in S'} \in U_q$,
which implies
$\vs_1 + \cdots + \vs_m \in \type {g_q} {k_q} {\va_q} {\vb_q} {\vc_q}$.

In the automaton $\mA'$, the type associated to each state
is not necessarily proper.
Our next step is to build from $\mA'$
an automaton $\mA$ consisting only of proper types.
Lemma~\ref{th:to:proper} allows us to translate
each type $(g,k,\va,\vb,\vc)$ into a proper one,
provided that $k$ is large enough.
We consider all types $(g_q,k_q,\va_q,\vb_q,\vc_q)$
corresponding to states $q$ of the automaton $\mA'$,
and a value~$K$ large enough
so that Lemma~\ref{th:to:proper} is applicable to each type
$(g_q,K,\va_q,\vb_q,\vc_q)$.
We start with $\mA$ as a copy of $\mA'$.
Now let us cut all transitions leaving the initial state
in the automaton $\mA$.
Consider an integer $L$, fixed later in the proof.
For each N-walk on $S'$ of length at most $L$
with reachable points~$\vs$,
we consider the type $(0,0,\emptyset,\vs,\emptyset)$.
We create a new state corresponding to this type.
We add all transitions from the initial state to those
newly created states,
as well as transitions between them.
Consider a N-walk $w$ on $S'$ of length exactly $L$,
and the corresponding new state~$q$.
Reading~$w$ in the automaton $\mA'$ ends in some state $q'$.
We copy all transitions leaving $q'$
as transitions leaving $q$.
Let us choose $L = \max(g_q K + \|\vb_q\|)$
for all states $q$ of $\mA'$.
We update the value $k_q$ of each state $q$ to $K$.
Now any N-walk reaching in $\mA$ a state $q$ that existed in $\mA'$
has length at least $L$,
so the norm of its reachable points is at least $L$
(because all N-steps in $S'$ have size at least $2$).
By construction of $\mA'$,
the reachable points $\vr$ of $w$ have the form
\[
    j \times \{0,g_q\} + \vb \usub \va \msub \vc + \{m\}
\]
for some $j$ and $m$.
Because the norm of $\vr$ is
at least $g_q K + \|\vb_q\|$,
we have $j \geq K$,
so $\vr$ is indeed in type $(g_q,K,\va_q,\vb_q,\vc_q)$.
This proves that our new automaton $\mA$,
where all types are proper,
still matches N-walks to types
corresponding to their reachable points.

Finally, we extend the automaton $\mA$ to work with all N-steps $S$ (instead of only $S'$)
by introducing one by one N-steps of size $0$ or $1$,
as follows.
Assume $|\vs| = 1$, then the sumset with $\vs$ is a shift.
Since types are stable by shifts,
the automaton for $S' \cup \{\vs\}$
is obtained from $\mA'$ by adding loops labeled by $\vs$ on each state.
Since for any $\vb$, we have $\vb + \emptyset = \emptyset$,
inserting $\emptyset$ in $S'$
requires to add in $\mA'$
a new state corresponding to the type
$(0,0,\emptyset, \emptyset, \emptyset)$
and let all states point to it with transitions labeled by~$\emptyset$.
\end{proof}

In the next section we will discuss the behavior of N-excursions and N-meanders for general N-step sets.

\section{N-meanders and N-excursions with general N-steps}
\label{sec:general_NExcursions}

In this section we will show that the counting generating function (keeping track of the length and the maximum reachable point) of N-meanders is always algebraic.
We conjecture that the same is true for N-excursions, which we show in several further examples.

    \subsection{N-meanders}

Recall that a (classical) meander is a walk
that stays nonnegative,
and that an N-meander is an N-walk compatible
with at least one meander.

\begin{theorem}
    \label{theo:NMeanderalgebraic}
    Consider a finite family $S$ of nonempty N-steps and define the multiset $S_T = \{ \max(\vs) \mid \vs \in S \}$. 
    \begin{itemize}
        \item The generating function $D^+(1,y;t)$ of N-meanders with N-step set $S$, where $y$ marks the maximum reachable point and $t$ the length, is equal to the generating function of classical meanders with step set $S_T$ and is therefore algebraic.
        \item The generating function $D^+(0,0;t)$ of N-meanders with N-step set $S$ and ending with reachable point set $\{0\}$ is equal to the generating function of classical excursions with step set $S_T$ and is therefore algebraic.
    \end{itemize}
\end{theorem}

\begin{proof}
    Observe that an N-meander can be continued as long as the set of reachable points is non-empty. 
    This is equivalent to the fact that the top-path in the N-meander, \ie the one the furthest away from the $x$-axis, can be continued. 
    In fact, mapping the N-meander to its top path gives a bijection between N-meanders and classical meanders.
    This classical meander's step set consists of the maxima of each N-step, which is therefore $S_T$.    
    Note that the multiset can be converted into an ordinary set, by introducing colors for steps that appear multiple times.
    
    Then, the number of N-meanders is the number of classical meanders with step set $S_T$. 
    Moreover, the number of N-meanders with reachable point set $\{0\}$ is equal to the number of classical excursions with step set $S_T$.
    Hence, by, \eg the theory of Banderier and Flajolet~\cite[Theorem~2]{BaFl02} we know that its generating function is algebraic and can be computed explicitly.
\end{proof}

\begin{remark}[Step set $S_T$ with colors]  
    \label{rem:NMeanderColors}
    We can interpret the step set $S_T$ also as ordinary set $\widehat{S_T}$, after attaching ``colors'' to the steps. 
    For this purpose, recall that each N-step $\vs \in S$ is associated with a weight $p_{\vs}$. 
    Then, by the previous result, each step $j \in \widehat{S_T}$ gets a weight
    \[q_j = \sum_{\substack{\max(\vs)=j \\ \vs \in S}} p_{\vs}. \]
    This interpretation directly shows that the generating functions $D^+(1,y;t)$ and $D^+(0,0;t)$ depend only on these weights $q_j$, which generalizes the result of Corollary~\ref{cor:DyckNMeanderpq}.
\end{remark}

It remains to understand the behavior of N-excursions for general N-step sets, which is strongly coupled with the generic nature of the generating function of N-meanders in three variables (length, minimum and maximum reachable point).
This will turn out to be the most difficult problem. 

    \subsection{N-excursions}

Recall that a (classical) excursion is a meander with endpoint $0$,
and an N-excursion is an N-walk compatible
with at least one excursion.
First, consider the special case of a step set closed under vertical symmetry:
We define $-\vs = \{-s_1, \dots, -s_k\}$ for any $\vs = \{s_1,\dots,s_k\}$.
An N-step set $S$ is \emph{symmetric}, if for all $\vs \in S$ we also have $-\vs \in S$. 
The following result generalizes Corollary~\ref{cor:DyckNExcSymmetric}.

\begin{theorem}
    Consider a symmetric finite N-step family $S$,
    where each N-step $\vs$ comes with a weight $p_{\vs}$.
    Then the generating function $D^+(0,1;t)$ of N-excursions
    is symmetric in $p_{\vs}$ and $p_{-\vs}$ for all $\vs \in S$.
\end{theorem}

\begin{proof}
    We define an involution $\Phi : S \to S$ by mapping each N-step to its symmetric counter part
    \begin{align*}
        \Phi(\vs) = -\vs.
    \end{align*}
    As in Corollary~\ref{cor:DyckNExcSymmetric}, this extends to an involution between N-excursions:
    \begin{align*}
        w=(\vw_1,\dots,\vw_n) &\mapsto w'=(\vw'_1,\dots,\vw'_n), \quad \text{ where } \vw'_i = \Phi(\vw_{n+1-i}).
    \end{align*}
    The key observation is that if the classical excursion $v$ is compatible with $w$, then the time-reversed classical excursion $v'$ is compatible with $w'$; see, \eg \cite{Banderieretal2020Latticepathology} for such operations on lattice paths.    
    By the symmetry of the N-step set $S$, the N-steps of $w'$ all lie in $S$.
\end{proof}

We conjecture that Proposition~\ref{th:type:of:sumsets}
extends to N-meanders.

\begin{conjecture}
For any N-step set $S$,
there exists a finite family of types
such that the set of reachable points of any N-meander on $S$
belongs to one of those types.
\end{conjecture}

We also believe that a variant of Proposition~\ref{th:automatic:type:of:sumsets}
applies to N-meanders.
In order to express it,
we first define a variant of the automaton
defined in Proposition~\ref{th:automatic:type:of:sumsets}.
Indeed, as illustrated by Motzkin N-meanders (see Figure~\ref{fig:motzkin_Nmeanders_structure}),
information about the minimum and maximum reachable points
is needed in the transitions of the automaton
for N-meanders.

Let us recall that the minimum (\resp maximum) reachable point
of an N-meander $w$
is the minimal (\resp maximal) endpoint
of any meander compatible with $w$.

\begin{definition}
Given a finite N-step family $S$,
a \emph{meander automaton} $\mA$ on $S$
is composed of a finite set of states,
among which the state $0$ is the initial state.
It has a finite number of transitions.
Each transition is an arc from a state to a state,
decorated with a triplet $(\vs, \lowercond, \uppercond)$,
where
\begin{itemize}
\item $\vs$ is an N-step from $S$,
\item $\lowercond$ (\resp $\uppercond$) is a predicate
that inputs an integer $h$
and has the form $h = k$, $h<k$, or $h > k$ for some integer $k$.
\end{itemize}
Any N-meander $w$ (\ie sequence of N-steps)
is either rejected by the automaton $\mA$,
or it corresponds to a walk on the states of $\mA$
defined inductively as follows.
If $w$ is the empty sequence,
then the walk has length $0$, starting and ending
at the initial state $0$.
Otherwise, $w$ decomposes as a sequence
$w' = (\vs_1, \ldots, \vs_{\ell})$
followed by an N-step $\vs$.
If $w'$ is rejected by $\mA$,
then so is $w$.
Otherwise, let $j'$ denote the state reached by $w'$
on its walk on $\mA$.
If there is a transition from $j'$
of the form $(\vs, \lowercond, \uppercond)$
such that the minimum reachable point of the N-meander $w'$
satisfies $\lowercond$
and the maximum reachable point of $w'$
satisfies $\uppercond$,
then this transition must be unique
(deterministic automaton).
The state $j$ to which it points
is the state reached by $w$.
If no such transition exists,
then $\mA$ rejects the N-meander $w$.
\end{definition}

Examples of meander automata
are provided in Figures~\ref{fig:motzkin_Nmeanders_structure}
and~\ref{fig:bigexample}
(where the predicates $\lowercond$ and $\uppercond$
are omitted when they are equal to
the trivial predicate $\geq 0$).
Now that meander automata are defined,
we can conjecture an extension of
Proposition~\ref{th:automatic:type:of:sumsets}.

\begin{conjecture}
\label{conj:meander:automaton}
For any finite N-step family $S$,
there exists a meander automaton $\mA$
and a type associated to each of its states,
such that $\mA$ accepts exactly the N-meanders on $S$
with nonempty reachable points,
and any such N-meander belongs to the type
associated to the final state it reaches
during its walk on $\mA$.
\end{conjecture}

We believe proving the previous conjecture
would be an important step into proving the following
main conjecture of this paper.

\begin{conjecture}
    \label{conj:NExcursionsAlgebraic}
    For any N-step set, the generating function of N-excursions is algebraic.
\end{conjecture}

We implemented a Python package~\cite{gitlabproject} that helps to create the automaton for N-meanders (and also N-bridges) for any given N-step set starting from the origin. 
First, the user inputs a set of N-steps $S$.
Second, the user specifies the encountered types $\type{n_i}{k_i}{\va_i}{\vb_i}{\vc_i}$, $i=1,\dots,T$. 
Next, the program attaches the N-steps to the reachable points in the types, and checks to which type they lead. 
If there are reachable points that cannot be matched with any type, the program outputs them and the user needs to extend or change the chosen types. 
Once there are no more unmatched reachable points, the automaton is finished.
Then, the program builds the transition matrices. 

For this purpose the program chooses ``suitable candidates'' in the types, as these are in general infinite families.
Note that at this stage the program might miss cases, as the following strategy is a heuristic that worked for us in all considered cases:
For each type we choose three candidates $\vr$ depending on the minimum and maximum:
\begin{enumerate}
    \item \emph{Min large/Max large:} $\vr = \widetilde{\vr} + \{m\}$ where $\widetilde{\vr} = j \times \{0,n\} + \vb \usub \va \msub \vc$ such that $j = \max_i (k_i) +1$ and $m=-\min(\widetilde{\vr}) + \max_i \min \va_i$. 
    This is the generic case, in which there is no interaction with the boundary after appending an N-step and no bottom pruning necessary.
    \item \emph{Min small/Max large:} For each $\ell \in  \{0,\dots,c\}$ where $c:=-\min \min S$ we consider
    $\vr_{\ell} = \widetilde{\vr} + \{m + \ell\}$ where $\widetilde{\vr} = j \times \{0,n\} + \vb \usub \va \msub \vc$ such that $j = \max_i (k_i) +1$ and $m=-\min(\widetilde{\vr})$. 
    In this case the reachable points have minimum $\ell$ and a large maximum. 
    Now, at least one N-step after being appended to $\vr_{\ell}$ will lead to pruning.
    Because of that there is a ``jump'' in the minimum.
    \item \emph{Min small/Max small:} For each $\ell \in  \{0,\dots,c\}$ we consider
    $\vr_{\ell} = \widetilde{\vr} + \{m + \ell\}$ where $\widetilde{\vr} = k \times \{0,n\} + \vb \usub \va \msub \vc$ such that $j = \max_i (k_i) +1$ and $m=-\min(\widetilde{\vr})$. 
    These are the reachable points with smallest norm (note the choice of $j=k$ compared to before) and small minimum $\ell$. 
    Again, at least one N-step after being appended to $\vr_{\ell}$ will lead to pruning.
    In these cases the maximum might also lead to pruning.
\end{enumerate}

We define now for each type $i$ a generating function that keeps track of the length in $t$, the minimum in $x$, and the maximum in $y$ as follows: 
\begin{align*}
    M_i^+(x,y;t) &= \sum_{w \in M_1^+} x^{\min^+(w)} y^{\max^+(w)-\sigma_i} t^{|w|},
\end{align*}
where $\sigma_i$ is the minimal norm in type $i$. 
We group all of them in the column vector $M^+(x,y,t) = ( M_1^+(x,y;t), \dots,  M_N^+(x,y;t) )^T$.
As in the case of Dyck and Motzkin paths, these generating functions are related as follows
\begin{align*}
	M^+(x,y; t) &= 
    e_1 + t \Bigg( 
    A(x,y) \Big(M^+(x,y; t) - \sum_{\ell} [x^{\ell} y^{\geq 0}] M^+(x,y; t)\Big) \\
    &\quad+ \sum_{\ell} B_{\ell}(x,y) [x^{\ell} y^{>0}] M^+(0,y; t)  
    +\sum_{\ell} C_{\ell}(x,y) [x^{\ell} y^{0}] M^+(0,y; t) \Bigg),
\end{align*}
where the matrices $A(x,y)$, $B_{\ell}(x,y)$, and $C_{\ell}(x,y)$ describe the transitions between types when appending N-steps and the associated change in the minimum and maximum of the reachable points. 
The Python code then saves these matrices and the shifts $\sigma_i$ in Maple readable form into a text file.

Second, we implemented a Maple worksheet, in which this information can be loaded.
Then, this worksheet provides some help to analyze and solve the generating functions of N-meanders. 
For example, it provides methods to iteratively compute the first terms of the series expansions, and to guess the generating functions of N-meanders $M^+(1,1;t)$ and N-excursions $M^+(0,1;t)$. 

\smallskip

At the end, let us show how this code can help to analyze a model of N-walks on the example of N-steps $\{\{-1\}, \{-1,1\}, \{-2,1\}\}$.
Using this Python code we found the transition automaton of N-meanders shown in Figure~\ref{fig:bigexample} in Appendix~\ref{app:bigexample}.
In particular, it consists of two types of period $0$, two types of period $1$, one type of period $2$, and $3$ types of period $3$. 
The reason for the appearance of period $3$, is the N-step $\{-2,1\}$ whose norm is $3$.
Note that the choice of types is in general not unique.
It remains to show that these $8$ types characterize all N-meanders, which we omit here. 
As all N-steps have a negative component, Theorem~\ref{th:negative_min_N_steps} that we prove below then implies that the corresponding generating function of N-excursions is algebraic. 

In the accompanying Maple file we used the computed matrices $A$, $B_0$, $B_1$, and $C_0$ to guess an algebraic equation of degree $4$ for the generating function $M^+(0,1;t)$ of N-excursions. 
We tested it using the first $500$ coefficients, and we leave the proof of its correctness as an open problem.

\smallskip

Finally, we prove a particular case of our main conjecture~\ref{conj:NExcursionsAlgebraic}. 
We show that when the meander automaton is known \emph{and} the N-step set is of a special shape, then the generating function of N-excursions is indeed algebraic.

\begin{theorem} \label{th:negative_min_N_steps}
    Consider a finite N-step set $S$,
    where each N-step has a non-positive part,
    meaning $\max \min_{s \in S} s \leq 0$.
    Furthermore, let assume that there exists a meander automaton $\mA$
    with a type associated to each state,
    such that any N-meander on $S$ belongs
    to the type associated to it by its walk on $\mA$.
    Then, the generating function of N-excursions is algebraic.
\end{theorem}

\begin{proof}
    Recall that the norm $\|\vs\|$ on an N-step $\vs$
    is defined as $\max(\vs) - \min(\vs)$.
    Let $d = \max_{\vs \in S} \|\vs\|$
    and assume every N-step from $S$ has a nonpositive part,
    then by induction, for any N-meander $w$ on $S$,
    the minimum reachable point of $w$ is at most $d$.
    Then, we build a deterministic pushdown automaton as follows:
    First, we create $d+1$ states for each state of $\mA$,
    where each of them corresponds to a different value of the minimum. 
    Second, we add the transitions according to the transitions
    of $\mA$.
    Recall that any set in a type is uniquely characterized
    by its minimum and maximum. 
    As the minimum is known, it suffices to know the maximum,
    which we encode in the stack of a pushdown automaton. 
    Hence, N-excursions can be encoded by an unambiguous context-free grammar~\cite{Banderieretal2020Latticepathology,Duchon00,MerliniEtal1999Underdiagonal,BP08}, and therefore they have an algebraic generating function by the Chomsky--Schützenberger enumeration theorem~\cite{ChomskySchuetzenberger1963CF}; see also Section~\ref{sec:introalgebraic}.
\end{proof}

Note that for the cases of Dyck and Motzkin N-steps considered in Sections~\ref{sec:DyckNWalks} and \ref{sec:MotzkinNWalks} the previous theorem does not hold, as both contain the step $\{1\}$. 
Yet, as shown in Theorems~\ref{th:other_Dyck_Nmeanders} and~\ref{th:motzkin_N_meanders} the generating functions are still algebraic.

\section{Conclusion}

In this paper, motivated by the study
of encapsulation and decapsulation of protocols over networks,
we introduced nondeterministic walks (N-walks).
We solved their exact and asymptotic enumeration
for the particular cases of Dyck and Motzkin N-walks.
Our main tool was analytic combinatorics,
and more specifically the kernel method.
Furthermore, we showed that for arbitrary N-steps, the generating function of N-bridges
(N-walks that can reach $0$) is algebraic.
We also conjecture that the generating function
of N-excursions (N-walks that can reach $0$ while staying nonnegative)
is algebraic, and provided special cases, examples and code supporting this conjecture.

The generating function of any non-ambiguous context-free language
is algebraic.
Ambiguous context-free languages can have
algebraic or non-algebraic generating functions
(see, \eg \cite{koechlin2022new} and references therein).
We conjecture that for any finite N-step family,
the generating function of N-excursions is algebraic.
Is it the case that the underlying language is always
a non-ambiguous context-free one?
If not, which N-step families correspond to
non-ambiguous context-free languages?
We leave those very interesting questions as open problems.

The general study of encapsulation and decapsulation protocols over networks
proved too challenging on a first approach,
so we simplified it in two ways.
We considered only one protocol
and assumed the network to be a line of servers.
In future works, it would be very interesting to remove those limitations.
Extending our results to several protocols seems challenging and interesting,
as the set of reachable points loses most of the algebraic structure
that allowed us to use generating functions
with a finite number of variables.
A probabilistic approach might prove more adequate.
On the other hand, more general network geometries (with only one protocol)
seem amenable to the tools developed in this article.
A first natural model to consider would be series parallel graphs,
since they possess a natural source and sink.

\medskip

{\bf Acknowledgements:}
\label{sec:ack}
\anonymous{The authors would like to thank Mohamed Lamine Lamali
for providing the motivation of this work
and for his contributions to the short version of this paper.
We thank Mireille Bousquet-M\'elou and Rika Yatchak for their useful comments and suggestions. 
We warmly thank Florent Koechlin and Arnaud Carayol for interesting conversations on the (non-)ambiguity of the language of N-walks and N-excursions in particular.
We would also like to thank Sergey Dovgal and Philippe Duchon for pointing out the combinatorial proof of Corollary~\ref{cor:DyckNExcSymmetric}.
We are particularly grateful to the referees for their valuable suggestions, which greatly improved the paper, and especially for recommending the addition of Section~\ref{sec:introAC} to make it more self-contained.
Both authors were supported by the Rise project RandNET, grant no.: H2020-EU.1.3.3.
The second author was supported by the Exzellenzstipendium of the Austrian Federal Ministry of Education, Science and Research and Austrian Science Fund (FWF):~P~34142 and J~4162.}{The authors are grateful for the support of many colleagues and institutions, that will be cited in the final version. We are particularly grateful to the referees for their valuable suggestions, which greatly improved the paper, and especially for recommending the addition of Section~\ref{sec:introAC} to make it more self-contained.}

\bibliographystyle{mybiburl}
\bibliography{biblio-elie}

\newpage

\appendix

\section{Subclasses of Motzkin N-walks in the OEIS}
\label{sec:SubclassesNMotzkin}

In Theorem~\ref{th:motzkin_N_meanders} we have characterized the generating function of Motzkin N-meanders $M^+(x,y;t)$ and proved that it is always algebraic.
In this section we will consider several specific choices of the weights in the counting generating function $M^+(0,1;t)$ of Motzkin N-excursions and $M^+(1,1;t,)$ of Motzkin N-meanders and we will see that many lead to known sequences in the OEIS. 
In particular, we consider all counting classes of such weights, meaning that the weights are either equal to $0$ or $1$.
We checked all $127$ possibilities, and list the non-trivial ones here in three tables.
Note that these connections are conjectural and are based on the fact that the first $15$ elements of the sequence are the same.

\subsection{Motzkin N-excursions}
\label{sec:SubclassesNMotzkinExcursions}

The examples in Table~\ref{tab:NMotzkinDycklike} satisfy $p_{-1,1}=1$, while the ones in Tables~\ref{tab:NMotzkinNoPm11} and \ref{tab:NMotzkinPathConnections} satisfy $p_{-1,1}=0$. 
Note that the latter consist of only one type that is an interval.

\bigskip

\newcounter{act}
\newcommand{\actname}[1]{A#1}
\newcommand{\addactclass}{\refstepcounter{act}\actname{\theact}}
\newcommand{\refact}[1]{{\footnotesize\hyperref[#1]\actname{\ref{#1}}}}

\newcounter{bct}
\newcommand{\bctname}[1]{B#1}
\newcommand{\addbctclass}{\refstepcounter{bct}\bctname{\thebct}}
\newcommand{\refbct}[1]{{\footnotesize\hyperref[#1]\bctname{\ref{#1}}}}

\newcounter{cct}
\newcommand{\cctname}[1]{C#1}
\newcommand{\addcctclass}{\refstepcounter{cct}\cctname{\thecct}}
\newcommand{\refcct}[1]{{\footnotesize\hyperref[#1]\cctname{\ref{#1}}}}

\newcounter{dct}
\newcommand{\dctname}[1]{D#1}
\newcommand{\adddctclass}{\refstepcounter{dct}\dctname{\thedct}}
\newcommand{\refdct}[1]{{\footnotesize\hyperref[#1]\dctname{\ref{#1}}}}

\newcounter{ect}
\newcommand{\ectname}[1]{E#1}
\newcommand{\addectclass}{\refstepcounter{ect}\ectname{\theect}}
\newcommand{\refect}[1]{{\footnotesize\hyperref[#1]\ectname{\ref{#1}}}}

\newcounter{fct}
\newcommand{\fctname}[1]{F#1}
\newcommand{\addfctclass}{\refstepcounter{fct}\fctname{\thefct}}
\newcommand{\reffct}[1]{{\footnotesize\hyperref[#1]\fctname{\ref{#1}}}}

\newcounter{gct}
\newcommand{\gctname}[1]{G#1}
\newcommand{\addgctclass}{\refstepcounter{gct}\gctname{\thegct}}
\newcommand{\refgct}[1]{{\footnotesize\hyperref[#1]\gctname{\ref{#1}}}}

\newcounter{hct}
\newcommand{\hctname}[1]{H#1}
\newcommand{\addhctclass}{\refstepcounter{hct}\hctname{\thehct}}
\newcommand{\refhct}[1]{{\footnotesize\hyperref[#1]\hctname{\ref{#1}}}}

\newcounter{ict}
\newcommand{\ictname}[1]{I#1}
\newcommand{\addictclass}{\refstepcounter{ict}\ictname{\theict}}
\newcommand{\refict}[1]{{\footnotesize\hyperref[#1]\ictname{\ref{#1}}}}

\newcounter{jct}
\newcommand{\jctname}[1]{J#1}
\newcommand{\addjctclass}{\refstepcounter{jct}\jctname{\thejct}}
\newcommand{\refjct}[1]{{\footnotesize\hyperref[#1]\jctname{\ref{#1}}}}

\newcounter{kct}
\newcommand{\kctname}[1]{K#1}
\newcommand{\addkctclass}{\refstepcounter{kct}\kctname{\thekct}}
\newcommand{\refkct}[1]{{\footnotesize\hyperref[#1]\kctname{\ref{#1}}}}

\newcounter{lct}
\newcommand{\lctname}[1]{L#1}
\newcommand{\addlctclass}{\refstepcounter{lct}\lctname{\thelct}}
\newcommand{\reflct}[1]{{\footnotesize\hyperref[#1]\lctname{\ref{#1}}}}

\begin{table}[h!]
{\small
    \newcommand{\myrow}[9]{$#1$ & $#2$ & $#3$ & $#4$ & $#5$ & $#6$ & $#7$ & #8 & #9}
    \begin{center}
    \begin{tabular}{cccccccccc}
    \toprule
         &
         $p_{1}$ & $p_{-1}$ & $p_{0}$ & $p_{-1,0}$ & $p_{0,1}$ & $p_{-1,1}$ & $p_{-1,0,1}$ & OEIS & $e_n=$ \\
    \midrule
\addactclass &
\myrow{1}{0}{0}{0}{0}{1}{0}{\multirow{2}{*}{\OEISs{A126869}}}{$\binom{n}{\lfloor n/2 \rfloor}$ for $n$ even,}\\
\addactclass &
\myrow{0}{1}{0}{0}{0}{1}{0}{}{$0$ for $n$ odd}\\[2mm]
\addbctclass &
\myrow{1}{0}{1}{0}{0}{1}{0}{\multirow{2}{*}{\OEISs{A002426}}}{\multirow{2}{*}{$[x^n] (1 + x + x^2)^n$}}\\
\addbctclass &
\myrow{0}{1}{1}{0}{0}{1}{0}{}{}\\[2mm]
\addcctclass &
\myrow{0}{0}{0}{1}{0}{1}{0}{\multirow{2}{*}{\OEISs{A084174}}}{\multirow{2}{*}{$2^n - \frac{2n+1-(-1)^n}{4}$}}\\
\addcctclass &
\myrow{0}{0}{0}{0}{1}{1}{0}{}{}\\[2mm]
\adddctclass &
\myrow{0}{0}{0}{0}{0}{1}{1}{\OEISs{A051049}}{$2^{n} - \frac{1 - (-1)^{n}}{2}$}\\[2mm]
\addectclass &
\myrow{0}{0}{1}{0}{0}{1}{1}{\OEISs{A083313}}{$3^{n+1}-2^{n} - \delta_{n,0}$}\\
    \bottomrule
    \end{tabular}
    \end{center}
    \caption{N-Motzkin excursions with special choices of of its N-steps with connections to the \href{https://oeis.org/}{OEIS}. Note that many choices do not correspond to sequences in the OEIS, yet all are algebraic. Moreover, we have not listed the deterministic Dyck and Motzkin paths, neither constant $1$-sequences.}
    \label{tab:NMotzkinDycklike}
}
\end{table}

\begin{table}[h!]
{\small
    \newcommand{\myrow}[9]{$#1$ & $#2$ & $#3$ & $#4$ & $#5$ & $#6$ & $#7$ & #8 & #9}
    \begin{center}
    \begin{tabular}{cccccccccc}
    \toprule
         &
         $p_{1}$ & $p_{-1}$ & $p_{0}$ & $p_{-1,0}$ & $p_{0,1}$ & $p_{-1,1}$ & $p_{-1,0,1}$ & OEIS & $e_n=$ \\
    \midrule
\addfctclass &
\myrow{1}{0}{0}{1}{0}{0}{0}{\multirow{4}{*}{\OEISs{A001405}}}{\multirow{4}{*}{$\binom{n}{\lfloor n/2 \rfloor}$}}\\
\addfctclass &
\myrow{0}{1}{0}{0}{1}{0}{0}{}{}\\
\addfctclass &
\myrow{1}{0}{0}{0}{0}{0}{1}{}{}\\
\addfctclass &
\myrow{0}{1}{0}{0}{0}{0}{1}{}{}\\[2mm]
\addgctclass &
\myrow{1}{0}{1}{1}{1}{0}{0}{\multirow{4}{*}{\OEISs{A001700}}}{\multirow{4}{*}{$\binom{2n+1}{n+1}$}}\\
\addgctclass &
\myrow{1}{0}{1}{0}{1}{0}{1}{}{}\\
\addgctclass &
\myrow{0}{1}{1}{1}{1}{0}{0}{}{}\\
\addgctclass &
\myrow{0}{1}{1}{1}{0}{0}{1}{}{}\\[2mm]
\addhctclass &
\myrow{0}{0}{1}{1}{1}{0}{0}{\multirow{4}{*}{\OEISs{A000244}}}{\multirow{4}{*}{$3^n$}}\\
\addhctclass &
\myrow{0}{0}{1}{1}{0}{0}{1}{}{}\\
\addhctclass &
\myrow{0}{0}{1}{0}{1}{0}{1}{}{}\\
\addhctclass &
\myrow{0}{0}{0}{1}{1}{0}{1}{}{}\\[2mm]
\addictclass &
\myrow{1}{0}{1}{1}{0}{0}{0}{\multirow{8}{*}{\OEISs{A005773}}}{}\\
\addictclass &
\myrow{0}{1}{1}{0}{1}{0}{0}{}{}\\
\addictclass &
\myrow{1}{0}{0}{1}{1}{0}{0}{}{}\\
\addictclass &
\myrow{0}{1}{0}{1}{1}{0}{0}{}{$[x^{n-1}] (1 + x + x^2)^n$}\\
\addictclass &
\myrow{1}{0}{1}{0}{0}{0}{1}{}{$+ [x^n] (1 + x + x^2)^n$}\\
\addictclass &
\myrow{0}{1}{1}{0}{0}{0}{1}{}{}\\
\addictclass &
\myrow{0}{1}{0}{1}{0}{0}{1}{}{}\\
\addictclass &
\myrow{1}{0}{0}{0}{1}{0}{1}{}{}\\
    \bottomrule
    \end{tabular}
    \end{center}
    \caption{N-Motzkin excursions with special choices of of its N-steps with connections to the \href{https://oeis.org/}{OEIS}. Note that many choices do not correspond to sequences in the OEIS, yet all are algebraic. Moreover, we have not listed the deterministic Dyck and Motzkin paths, neither constant $1$-sequences.}
    \label{tab:NMotzkinNoPm11}
}
\end{table}

\begin{table}[h!]
{\small
    \newcommand{\myrow}[9]{$#1$ & $#2$ & $#3$ & $#4$ & $#5$ & $#6$ & $#7$ & #8 & #9}
    \begin{center}
    \begin{tabular}{ccccccccccc}
    \toprule
         &
         $p_{1}$ & $p_{-1}$ & $p_{0}$ & $p_{-1,0}$ & $p_{0,1}$ & $p_{-1,1}$ & $p_{-1,0,1}$ & OEIS & Domain & Steps \\
    \midrule
\addjctclass &
\myrow{1}{0}{0}{1}{0}{0}{1}{\multirow{2}{*}{\OEISs{A151281}}}{Nonnegative & \multirow{2}{*}{$\{-1,1_1, 1_2\}$}}\\
\addjctclass &
\myrow{0}{1}{0}{0}{1}{0}{1}{}{line $\mathbb{N}$}\\[2mm]
\addkctclass &
\myrow{1}{0}{1}{1}{0}{0}{1}{\multirow{4}{*}{\OEISs{A129637}}}{ & \multirow{4}{*}{$\{W, SE, SW, NW \}$}}\\
\addkctclass &
\myrow{0}{1}{1}{0}{1}{0}{1}{}{Triangular}\\
\addkctclass &
\myrow{1}{0}{0}{1}{1}{0}{1}{}{lattice}\\
\addkctclass &
\myrow{0}{1}{0}{1}{1}{0}{1}{}{}\\[2mm]
\addlctclass &
\myrow{1}{0}{1}{1}{1}{0}{1}{\multirow{2}{*}{\OEISs{A151251}}}{First & $\{(0, 0, 1), (0, 1, 0),(1, 1, 0), $}\\
\addlctclass &
\myrow{0}{1}{1}{1}{1}{0}{1}{}{octant $\mathbb{N}^3$ & $(1, 1, 1),(-1, -1, 0)\}$}\\
    \bottomrule
    \end{tabular}
    \end{center}
    \caption{N-Motzkin excursions related to (higher-dimensional) lattice paths that start at the origin and remain in the given domain.}
    \label{tab:NMotzkinPathConnections}
}
\end{table}

\clearpage

\subsection{Subclasses of Motzkin N-meanders in the OEIS}
\label{sec:SubclassesNMotzkinMeanders}

Recall that in Theorem~\ref{theo:NMeanderalgebraic} we proved that the generating function $M^+(1,y;t)$ is also associated with classical meanders with a weighed step set.
Hence, the classes below are all in bijection to one-dimensional lattice paths, which also automatically implies that they are algebraic. 

First note that whenever the N-step $\{-1\}$ is missing, \ie $p_{-1}=0$, the counting sequence is either $a^n$ for $a =2,3,4,5,6$.
This is due to the fact that then we can always append any N-step. 
We omit these cases.
All other sequences with connections to the OEIS are shown in Tables~\ref{tab:NMotzkinMeandersBinom}, \ref{tab:NMotzkinMeanders1D}, and \ref{tab:NMotzkinMeanders3D}.
Again, note that these have been built by checking the first $15$ elements in each sequence.

\begin{table}[h!]
{\small
    \newcommand{\myrow}[9]{$#1$ & $#2$ & $#3$ & $#4$ & $#5$ & $#6$ & $#7$ & #8 & #9}
    \begin{center}
    \begin{tabular}{ccccccccc}
    \toprule
         $p_{1}$ & $p_{-1}$ & $p_{0}$ & $p_{-1,0}$ & $p_{0,1}$ & $p_{-1,1}$ & $p_{-1,0,1}$ & OEIS & $e_n=$ \\
    \midrule
\myrow{0}{1}{0}{0}{0}{1}{0}{\multirow{2}{*}{\OEISs{A001405}}}{\multirow{2}{*}{$\binom{n}{\lfloor n/2 \rfloor}$}}\\
\myrow{0}{1}{0}{0}{0}{0}{1}{}{}\\[2mm]
\myrow{1}{1}{1}{1}{0}{0}{0}{\multirow{4}{*}{\OEISs{A001700}}}{\multirow{4}{*}{$\binom{2n+1}{n+1}$}}\\
\myrow{0}{1}{1}{1}{1}{0}{0}{}{}\\
\myrow{0}{1}{1}{1}{0}{1}{0}{}{}\\
\myrow{0}{1}{1}{1}{0}{0}{1}{}{}\\[2mm]
\myrow{1}{1}{0}{1}{0}{0}{0}{\multirow{7}{*}{\OEISs{A005773}}}{}\\
\myrow{0}{1}{1}{0}{1}{0}{0}{}{}\\
\myrow{0}{1}{0}{1}{1}{0}{0}{}{\multirow{2}{*}{$[x^{n-1}] (1 + x + x^2)^n$}}\\
\myrow{0}{1}{1}{0}{0}{1}{0}{}{\multirow{2}{*}{$+ [x^n] (1 + x + x^2)^n$}}\\
\myrow{0}{1}{0}{1}{0}{1}{0}{}{}\\
\myrow{0}{1}{1}{0}{0}{0}{1}{}{}\\
\myrow{0}{1}{0}{1}{0}{0}{1}{}{}\\
    \bottomrule
    \end{tabular}
    \end{center}
    \caption{N-Motzkin meanders with special choices of of its N-steps with connections to the \href{https://oeis.org/}{OEIS}. Note that many choices do not correspond to sequences in the OEIS, yet all are algebraic. }
    \label{tab:NMotzkinMeandersBinom}
}
\end{table}

\begin{table}[h!]
{\small
    \newcommand{\myrow}[9]{$#1$ & $#2$ & $#3$ & $#4$ & $#5$ & $#6$ & $#7$ & #8 & #9}
    \begin{center}
    \begin{tabular}{cccccccccc}
    \toprule
         $p_{1}$ & $p_{-1}$ & $p_{0}$ & $p_{-1,0}$ & $p_{0,1}$ & $p_{-1,1}$ & $p_{-1,0,1}$ & OEIS & Steps \\
    \midrule
\myrow{1}{1}{0}{0}{1}{0}{0}{\multirow{6}{*}{\OEISs{A151281}}}{\multirow{6}{*}{$\{-1,1_1, 1_2\}$}}\\
\myrow{1}{1}{0}{0}{0}{1}{0}{}{}\\
\myrow{0}{1}{0}{0}{1}{1}{0}{}{}\\
\myrow{1}{1}{0}{0}{0}{0}{1}{}{}\\
\myrow{0}{1}{0}{0}{1}{0}{1}{}{}\\
\myrow{0}{1}{0}{0}{0}{1}{1}{}{}\\
    \bottomrule
    \end{tabular}
    \end{center}
    \caption{N-Motzkin meanders related to one-dimensional lattice paths that start at the origin and remain in the right half-line $\mathbb{N}$.}
    \label{tab:NMotzkinMeanders1D}
}
\end{table}

\begin{table}[h!]
{\small
    \newcommand{\myrow}[9]{$#1$ & $#2$ & $#3$ & $#4$ & $#5$ & $#6$ & $#7$ & #8 & #9}
    \begin{center}
    \begin{tabular}{cccccccccc}
    \toprule
         $p_{1}$ & $p_{-1}$ & $p_{0}$ & $p_{-1,0}$ & $p_{0,1}$ & $p_{-1,1}$ & $p_{-1,0,1}$ & OEIS & Steps \\
    \midrule
\myrow{1}{1}{0}{0}{1}{1}{0}{\multirow{4}{*}{\OEISs{A151162}}}{}\\
\myrow{1}{1}{0}{0}{1}{0}{1}{}{$\{(-1, 0, 0), (1, 0, 0), $}\\
\myrow{1}{1}{0}{0}{0}{1}{1}{}{$(1, 0, 1), (1, 1, 0)\}$}\\
\myrow{0}{1}{0}{0}{1}{1}{1}{}{}\\[2mm]
\myrow{1}{1}{1}{1}{1}{0}{0}{\multirow{6}{*}{\OEISs{A151251}}}{}\\
\myrow{1}{1}{1}{1}{0}{1}{0}{}{}\\
\myrow{0}{1}{1}{1}{1}{1}{0}{}{$\{(-1, -1, 0), (0, 0, 1), (0, 1, 0), $}\\
\myrow{1}{1}{1}{1}{0}{0}{1}{}{$(1, 1, 0), (1, 1, 1)\}$}\\
\myrow{0}{1}{1}{1}{1}{0}{1}{}{}\\
\myrow{0}{1}{1}{1}{0}{1}{1}{}{}\\[2mm]
\myrow{1}{1}{1}{0}{1}{1}{0}{\multirow{8}{*}{\OEISs{A151253}}}{}\\
\myrow{1}{1}{0}{1}{1}{1}{0}{}{}\\
\myrow{1}{1}{1}{0}{1}{0}{1}{}{}\\
\myrow{1}{1}{0}{1}{1}{0}{1}{}{$\{(-1, 0, 0), (0, 0, 1), (1, 0, 0), $}\\
\myrow{1}{1}{1}{0}{0}{1}{1}{}{$(1, 0, 1), (1, 1, 0)\}$}\\
\myrow{1}{1}{0}{1}{0}{1}{1}{}{}\\
\myrow{0}{1}{1}{0}{1}{1}{1}{}{}\\
\myrow{0}{1}{0}{1}{1}{1}{1}{}{}\\[2mm]
\myrow{1}{1}{0}{0}{1}{1}{1}{\OEISs{A151254}}{\begin{tabular}{c}$\{(-1, 0, 0), (1, 0, 0), (1, 0, 1), $ \\$ (1, 1, 0), (1, 1, 1)\}$\end{tabular}}\\
    \bottomrule
    \end{tabular}
    \end{center}
    \caption{N-Motzkin meanders related to three-dimensional lattice paths that start at the origin and remain in the first octant $\mathbb{N}^3$.}
    \label{tab:NMotzkinMeanders3D}
}
\end{table}

\clearpage
\section{The transition automaton of a large N-step model}
\label{app:bigexample}

\begin{figure}[ht!]
\begin{center}
\includegraphics[scale=0.7]{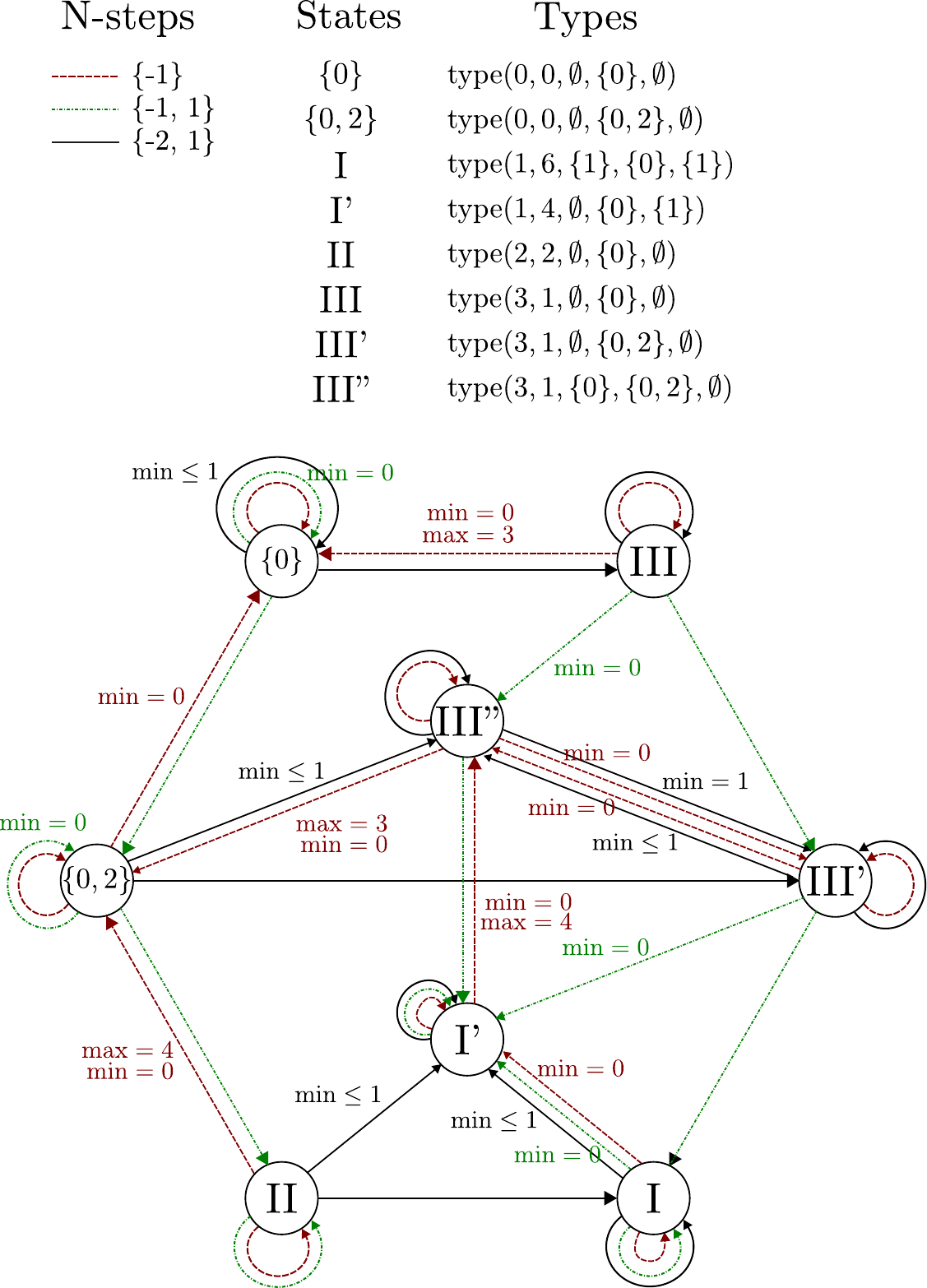}
\caption{Automata from Theorem~\ref{th:negative_min_N_steps} corresponding to the N-excursions on the N-steps $\{\{-1\}, \{-1,1\}, \{-2,1\}\}$.}
\label{fig:bigexample}
\end{center}
\end{figure}

\end{document}